\documentclass[12pt,reqno]{amsart}
\usepackage{amsmath,amsfonts,amscd,amssymb,graphicx,amsthm,enumerate,parskip,mathrsfs}  
 \usepackage[a4paper, margin=1in]{geometry}
\usepackage{times}

\theoremstyle{definition}
\newtheorem{theorem}{Theorem}[section]
\newtheorem{proposition}[theorem]{Proposition}
\newtheorem{lemma}[theorem]{Lemma}
\newtheorem{corollary}[theorem]{Corollary}

\newtheorem{remark}[theorem]{Remark}
 \numberwithin{equation}{section}
\numberwithin{equation}{section}
 
\newcommand{\di}{\displaystyle}
\newcommand{\pa}{\partial}
\newcommand{\vep}{\varepsilon}

\usepackage{cite}
\usepackage[colorlinks,citecolor=blue]{hyperref}

\begin{document}
  
\title[Local Index Theorem]{  Local Index Theorem for Cofinite Hyperbolic Riemann Surfaces}

\author{Lee-Peng Teo}

\address{Department of  Mathematics, Xiamen University Malaysia, Jalan Sunsuria, Bandar Sunsuria, 43900, Sepang, Selangor, Malaysia. }

\email{lpteo@xmu.edu.my}

\subjclass[2020]{Primary 30F60, 32G15, 32A55}

\date{\today}
\begin{abstract}
We discuss the local index theorem for cofinite Riemann surfaces in a pedagogical way, from a more computational perspective. Given a cofinite Riemann surface $X$, let $\Delta_n$ be the $n$-Laplacian and let   $N_n$ be the Gram matrix of a basis of holomorphic $n$-differentials on $X$.
The local index theorem says that on the Teichm\"uller space $T(X)$, the second variation of $\log\det\Delta_n-\log \det N_n$ can be written  as a sum of three symplectic forms $\omega_{\text{WP}}$, $\omega_{\text{TZ}}^{\text{cusp}}$ and $\omega_{\text{TZ}}^{\text{ell}}$. These are the symplectic forms  for the three K\"ahler metrics on $T(X)$ -- the Weil-Petersson metric, the parabolic Takhtajan-Zograf (TZ) metric and the elliptic Takhtajan-Zograf metric.

 Using Ahlfors' variational formulas and projection formulas, we derive explicitly integral formulas for the variations of $\log\det\Delta_n$ and $\log \det N_n$. The integrals are regular integrals that allow explicit computations. In the spirit of the Selberg trace formula, we identify the identity, hyperbolic, parabolic and elliptic contributions to the second variations of $\log\det\Delta_n$ and $\log \det N_n$. We showed that the Weil-Petersson term comes from the identity contribution, while the parabolic TZ metric and elliptic TZ metric terms come from parabolic and elliptic contributions respectively.  The hyperbolic contributions are cancelled. As a byproduct, we obtain alternative integral formulas for the parabolic TZ metric and the elliptic TZ metric. 
\end{abstract}

\keywords{Local Index Theorem, Cofinite Riemann Surfaces, Determinant of Laplacian, $n$-differentials,  Gram Matrix.}
\maketitle

 \section{Introduction}
 The purpose of this work is to give a computational proof of the local index theorem for cofinite hyperbolic Riemann surfaces. Inspired by the work \cite{BK}, the local index theorem was  first proved by   Takhtajan and Zograf in \cite{TZ_index_1} for compact hyperbolic Riemann surfaces. It was then extended to hyperbolic Riemann surfaces with cusp in \cite{TZ_index_2}, and finally to general cofinite hyperbolic Riemann surfaces in \cite{TZ_index_3}. A detailed exposition of the background of the local index theorem were given in \cite{TZ_index_1, TZ_index_2, TZ_index_3} .

The local index theorem for general cofinite Riemann surfaces says that  (see Section \ref{localindex} for more details)

\begin{align*}
\pa_{\mu}\pa_{\bar{\nu}}\log\frac{\det \Delta_n}{\det N_n}=&\frac{6n^2-6n+1}{12\pi}\langle \mu,\nu\rangle_{\text{WP}}-\frac{\pi}{9}\langle \mu, \nu\rangle_{\text{TZ}}^{\text{cusp}}-\sum_{j=1}^v \mathfrak{B}(m_j, n)\langle \mu, \nu\rangle_{\text{TZ}, j}^{\text{ell}}.
\end{align*}Here $\Delta_n$ is the determinant of $n$-Laplacian, $N_n$ is the Gram matrix of a basis of holomorphic $n$-differentials, and $\mathfrak{B}(m,n)$ is a constant depending on $m$ and $n$.
 
The derivation given in \cite{TZ_index_3} is based on the results from the previous papers  \cite{TZ_index_1, TZ_index_2}. In \cite{TZ_index_2} and \cite{TZ_index_3}, the parabolic and elliptic Takhtajan-Zograf (TZ) metrics $\langle \;\cdot\;, \;\cdot\;\rangle_{\text{TZ}}^{\text{cusp}}$ and   $\langle \;\cdot\;, \;\cdot\;\rangle_{\text{TZ}}^{\text{ell}}$ were introduced, which are K\"ahler metrics on the moduli space. These can be considered as the parabolic contribution and elliptic contribution to the local index theorem.

In the series of works  \cite{TZ_index_1, TZ_index_2, TZ_index_3}, Takhtajan and Zograf used a lot of properties of singular integral kernels and operator theory.  In this work, we rely less on properties of singular integrals, using explicit formulas whenever possible. This is in closer spirit to our work on universal Weil-Petersson Teichm\"uller space \cite{Memoir}, where  Ahlfors' approach to variational formulas \cite{Ahlfors_remarks, Ahlfors_curvature} were used extensively. 

  In the spirit of the Selberg trace formula (see for example \cite{Iwaniec}),  the  second variation of $\log\det N_n$  can be  split into the identity contribution, hyperbolic contribution, parabolic contribution, and elliptic contribution.  In this work, we  prove that the identity contribution gives the term
$$-\frac{6n^2-6n+1}{12\pi}\langle \mu,\nu\rangle_{\text{WP}}$$that involves the Weil-Petersson metric, while the parabolic  contribution and the elliptic contribution give respectively the terms involving the parabolic TZ metric and the elliptic TZ metric. The hyperbolic contribution is shown to be  equal to the second variation of $\log\det\Delta_n$. As a byproduct, we obtain an explicit formula for $\pa_{\mu}\pa_{\bar{\nu}}\log\det \Delta_n$. This is another main result of this work.

Finally, we would  like to mention that the approach used in this work is computational in nature. For some of the tedious calculations that are elementary, we use   computer algebra to do the calculation and no detailed steps would be provided. Due to the insufficiency of symbols, some of the symbols might be used for different purposes in different sections. 

\subsection*{Acknowledgements}
This research is supported by the XMUM Research Fund XMUMRF/2018-C2/IMAT/0003. I would
like to dedicate this work to  L. Takhtajan, who has provided valuable insights and comments during various phases of this work. 

\section{Preliminaries}
 In this section, we are going to present some definitions, facts and known results that are needed in this work.

 \smallskip
\subsection{Cofinite Hyperbolic Riemann Surfaces}\label{intro_surface}~

Let $X$ is be a cofinite hyperbolic Riemann surface and 
let $\mathbb{U}$ be the upper half plane. By uniformization theorem,   there is a discrete subgroup $\Gamma$ of $\text{PSL}\,(2,\mathbb{R})$ which acts discontinuously on $\mathbb{U}$ and so that the  quotient $\Gamma\backslash\mathbb{U}$ is biholomorphic to $X$ (i.e., $\Gamma\backslash\mathbb{U}\simeq X$).  If $X$ has genus $g$, $q$ cusps and $v$ ramification points, the group $\Gamma$ is generated by $2g$  hyperbolic elements $\alpha_1$, $\beta_1$,  $\ldots$, $\alpha_g$, $\beta_g$, $q$ parabolic elements $\kappa_1$, $\ldots$, $\kappa_q$, as well as $v$ elliptic elements $\tau_1, \ldots, \tau_v$ of orders $m_1, \ldots, m_v$ respectively, which can be assumed to satisfy $$2\leq m_1\leq m_2\leq \ldots \leq m_v.$$ 
The generators of $\Gamma$ satisfy the following nontrivial relations
\begin{gather*}
\alpha_1\beta_1\alpha_1^{-1}\beta_1^{-1}\ldots \alpha_g\beta_g\alpha_g^{-1}\beta_g^{-1}\kappa_1\ldots\kappa_q\tau_1\ldots\tau_v=I,\\
\tau_j^{m_j}=1,\hspace{1cm}1\leq j\leq v,
\end{gather*}where $I$ is the identity element. We say that the Riemann surface $X$ and the group $\Gamma$ are of type $(g;q;m_1, m_2, \ldots, m_v)$. This is also called the signature of the Riemann surface $X$.

Under the local coordinates $z=x+iy$ induced from $\mathbb{U}$, the hyperbolic metric density on $X$ is given by  $$\rho(z)= \frac{1}{y^{2}},$$ and the    area form is
$$dA(z)=\frac{dxdy}{y^2}=\rho(z)d^2z.$$ This is a metric with constant curvature $-1$.
The hyperbolic area of $X$
is  given by
$$|X|=2\pi\left\{2g-2+q+\sum_{j=1}^v\left(1-\frac{1}{m_j}\right)\right\}.$$
 In fact, a sufficent and necessary condition for $X$ to be hyperbolic is 
$$2g-2+q+\sum_{j=1}^v\left(1-\frac{1}{m_j}\right)>0.$$

If $\Gamma$ is a Fuchsian group so that $\Gamma\backslash\mathbb{U}\simeq X$, for any $\gamma\in \text{PSL}\,(2,\mathbb{R})$, $$\widetilde{\Gamma}=\gamma^{-1}\circ\Gamma\circ \gamma$$is  also a Fuchsian group and $\widetilde{\Gamma}\backslash\mathbb{U}\simeq X$. As such, we need to impose a marking and normalization on $\Gamma$. 
\begin{enumerate}[$\bullet$\;\;]
\item If $g\geq 1$, let $\alpha_1$ and $\alpha_2$ be the attracting fixed points of $a_1$ and $a_2$ respectively, and let $\beta_1$ be the repelling fixed point of $a_1$. We say that $\Gamma$ is marked and normalized if $\alpha_1=0, \alpha_2=1$ and $\beta_1=\infty$.
\item If $g=0$,  we must have $q+v\geq 3$. For $1\leq i\leq q$,   let $x_i$ be the fixed point of $\kappa_i$. If we also have $q\geq 3$, $\Gamma$ is marked and normalized if $x_1=0$, $x_2=1$ and $x_q=\infty$. 
\item The most complicated case is when $g=0$ and $q\leq 2$. In this case, we cannot normalize $\Gamma$ by using $0, 1$ and $\infty$. For $1\leq j\leq v$,  let $w_j$ be the fixed point of $\tau_j$.  If $q=2$, then $v\geq 1$. In this case, $\Gamma$ is marked and normalized if $x_1=0$, $x_2=\infty$ and $w_1=i$. If $q=1$, then $v\geq 2$. In this case, choose  $p_3\in\mathbb{U}$ such that $p_3\neq i$.  $\Gamma$ is marked and normalized if $x_1=\infty$, $w_1=i$ and $w_2=p_3$. If $q=0$, then $v\geq 3$. In this case, choose  $p_2$ and $p_3$ in $\mathbb{U}$ so that $i, p_2, p_3$ are distinct. $\Gamma$ is marked and normalized if $w_1=i, w_2=p_2, w_3=p_3$.
\end{enumerate}
In any of the cases above, there are three special points $p_1$, $p_2$ and $p_3$ on $\overline{\mathbb{U}}$ that are prescribed fixed points of elements of $\Gamma$. In case 1 and case 2, $p_1=0$, $p_2=1$ and $p_3=\infty$.

\smallskip
\subsection{Automorphic Forms, n-Laplacian Operators and Eisenstein series}\label{Laplace}~

Let $\hat{\mathbb{C}}=\mathbb{C}\cup\{\infty\}$. Given $X\simeq \Gamma\backslash\mathbb{U}$
and  integers $n$ and $m$, let $B_{n,m}(\Gamma)$ be the space of automorphic forms of type $(n,m)$, which are functions $\phi :\mathbb{U}\rightarrow\hat{\mathbb{C}}$ satisfying $$\phi(\gamma z)\gamma'(z)^n\overline{\gamma'(z)}^m=\phi(z),\hspace{1cm}\text{for all}\,\gamma\in\Gamma.$$ $B_{n,m}(\Gamma)$ can be identified with $B_{n,m}(X)$, the space of smooth sections of $\mathcal{K}^n\otimes \overline{\mathcal{K}}^m$, where
  $\mathcal{K}$ is the canonical bundle of $X$. When $m=0$, $B_{n}(X)=B_{n,0}(X)$ is 
 the space of $n$-differentials on $X$.   
 We denote  by $H_{n, m}^2(X)$ the subspace of $B_{n,m}(X)$ which have finite norm under the Hermittian inner product 
\begin{align}\label{eq3_8_1}\langle \phi, \psi\rangle =\iint\limits_X \phi(z)\overline{\psi(z)}\rho(z)^{1-n-m}d^2z.\end{align}
$H_{n, m}^2(X)$ is a Hilbert space. 

The complex structure and the metric on $X$ determine a connection
$$D:\partial_n\oplus \overline{\pa}_n:B_n(X)\rightarrow  B_{n+1}(X)\oplus B_{n,1}(X)$$ on the line bundle $\mathcal{K}^n$, which are given by
$$\overline{\pa}_n=\frac{\pa}{\pa \bar{z}},\hspace{1cm}\pa_n=\rho^n\frac{\pa}{\pa z}\rho^{-n}$$ locally. 
The $n$-Laplacian operator $\Delta_n=4\bar{\partial}^*_n\bar{\partial}_n$ is an operator on $H_n^2(X)=H_{n,0}^2(X)$ with explicit formula given by  $$\Delta_n= -4y^{2-2n}\frac{\pa}{\pa z}y^{2n}\frac{\pa}{\pa\bar{z}}.
$$This is a nonnegative self-adjoint  operator. In contrast to the convention used in \cite{TZ_index_1, TZ_index_2, TZ_index_3}, we put a factor $4$ in front of $\bar{\partial}^*_n\bar{\partial}_n$  so that when $n=0$, we get the usual Laplacian on functions $$\Delta_0=-y^2\left(\frac{\pa^2}{\pa x^2}+\frac{\pa^2}{\pa y^2}\right).$$
The nullspace of the linear transformation $\overline{\pa}_n: H_n^2(X)\rightarrow H_{n,1}^2(X)$ is called the space of integrable meromorphic $n$-differentials. They are functions on $\mathbb{U}$, automorphic with respect to $\Gamma$, which is holomorphic except possibly with poles at the elliptic fixed points. In the following, we just call them holomorphic $n$-differentials on $X$. Denote this space by $\Omega_n(X)$. It is a complex vector space with dimension $d_n$, where
\begin{align}\label{eq3_4_1}
d_n=\begin{cases} 1,\quad   n=0,\\
g,\quad   n=1,\\
\di (2n-1)(g-1)+(n-1)q+\sum_{j=1}^v\left[n\left(1-\frac{1}{m_j}\right)\right],\quad & n\geq 2.\end{cases}
\end{align}

For any positive integer $n$, there is   an operator $\mathscr{P}_n: H_n^2(X)\rightarrow \Omega_n(X)$ that projects an integrable $n$-differential to the subspace of holomorphic $n$-differentials. Let $K_n(z,w)$ be the kernel of $\mathscr{P}_n$.  Using the coordinates on $\mathbb{U}$,   when $n\geq 2$, $K_n(z,w)$ is given by
\begin{align}\label{eq3_5_1}
K_n(z,w)= \sum_{\gamma\in\Gamma}\mathcal{K}(\gamma z, w)\gamma'(z)^n= \sum_{\gamma\in\Gamma}\mathcal{K}( z, \gamma w)\overline{\gamma'(w)}^n,\end{align}
where \begin{align}\label{eq2_17_2}\mathcal{K}_n(z,w)=(-1)^n\frac{2^{2n-2}(2n-1)}{\pi}\frac{1}{(z-\bar{w})^{2n}}.\end{align}
If $\psi(z)\in H_n^2(X)$, then
$$\left(\mathscr{P}_n\psi\right)(z)=\int_X K_n(z,w)\psi(w)\rho(w)^{1-n}d^2w\in \Omega_n(X).$$ When $\phi\in \Omega_n(X)$,
then
\begin{align}\label{eq3_1_1} \mathscr{P}_n\phi =\phi.\end{align} In other words, restricted to the finite dimensional vector space $\Omega_n(X)$, $\mathscr{P}_n$ is the identity operator.

Notice that for fixed $w$, $K_n(z,w)$ is a holomorphic $n$-differential in $z$. For fixed $z$, $K_n(z,w)$ is an anti-holomorphic $n$-differential in $w$.

The spectrum of $\Delta_n$ consists of a discrete part and a continuous part. To describe the continuous spectrum, one needs to use the Eisenstein series.

For $1\leq i\leq q$, let $x_i $ be the fixed point of $\kappa_i$. Then $x_i$ is a representative of the cusp associated to $\kappa_i$. Let $\sigma_i\in\text{PSL}\,(2,\mathbb{R})$ be an element that conjugates $\kappa_i$ to $\di \begin{pmatrix} 1 & \pm1\\ 0 & 1\end{pmatrix}$, namely,
$$\sigma_i^{-1}\kappa_i\sigma_i=\begin{pmatrix} 1 & \pm1\\ 0 & 1\end{pmatrix}.$$Then $\sigma_i(\infty)=x_i$.

Given an element $\gamma\in \Gamma$, let $\Gamma_{\gamma}$ be the stabilizer of $\gamma$ in $\Gamma$. It consists of all $\alpha\in \Gamma$ such that $\alpha \gamma \alpha^{-1} =\gamma$. 

 If $B$ is  the parabolic subgroup generated by $\di \begin{pmatrix} 1 & 1\\ 0 & 1\end{pmatrix}$, then $\di \Gamma_{\kappa_i}=\sigma_i B\sigma_i^{-1}$ is the stabilizer of the cusp $\kappa_i$ in $\Gamma$. 

Given a nonnegative integer $n$, the Eisenstein series associated to the cusp $x_i$ is defined by
\begin{align*}
E_i(z,s; n)=\sum_{\gamma\in \Gamma_{\kappa_i}\backslash\Gamma}\left[\text{Im}\,\left(\sigma_i^{-1}\gamma z\right)\right]^{s-n}\left[(\sigma_i^{-1} \gamma)'(z)\right]^{n}
\end{align*}when $\text{Re}\,s>1$. To simplify notation, throoughout this paper,  $\gamma(z)$ is written as $\gamma z$,   $\sigma^{-1}\circ\gamma$ is written as $\sigma^{-1}\gamma$, and   $(\sigma^{-1}\circ\gamma)(z)$ is written  as $\sigma^{-1}\gamma z$. The Eisenstein series $E_i(z,s;n)$ can be analytically extended to be a meromorphic function on $\mathbb{C}$ \cite{Hejhal_2, Fischer, Venkov, Iwaniec}.

When $y\rightarrow\infty$,
$$
E_i(\sigma_j z,s;n)\sigma_j'(z)^n=\delta_{ij}y^{s-n}+\varphi_{ij}(s;n)y^{1-s-n}+\text{exponentially decaying terms}.
$$The $q\times q$ matrix $\Phi(s;n)=[\varphi_{ij}(s;n)]$ is called the scattering matrix. The function $\varphi(s;n)$ is defined to be the determinant of $\Phi(s;n)$, namely,
$$\varphi(s;n)=\det \Phi(s;n).$$

The spectral theory for the Riemann surface $X $ states that
  there is a countable orthonormal system $\{u_k\}_{k\geq 0}$ of eigenfunctions of $\Delta_n$ with eigenvalues $0=\lambda_0\leq \lambda_1\leq \lambda_2\leq \ldots$, and eigenpackets given by the Eisenstein series so that for any $g\in H_n^2(X)$,
\begin{align*}g(z)=&\sum_{k=0}^{\infty}\langle g, u_k\rangle u_k(z)+\frac{1}{4\pi}\sum_{j=1}^q\int_{-\infty}^{\infty}\left\langle g, E_j\left(\cdot,\frac{1}{2}+ir;n\right)\right\rangle E_j\left(z,\frac{1}{2}+ir;n\right)dr.\end{align*}

\smallskip
\subsection{Teichm\"uller Space}~
Given a Riemann surface of type $(g;q;m_1, m_2, \ldots, m_v)$, the Teichm\"uller space $T(X)$ of the Riemann surface $X$ is the universal cover of the moduli space of Riemann surfaces with signature $(g;q;m_1, m_2, \ldots, m_v)$. This space is a complex manifold with dimension $d_2$, the dimension of the space of holomorphic quadratic differentials $\Omega_2(X)$  \eqref{eq3_4_1}. Since we are working with variations, we only consider hyperbolic Riemann surfaces $X$ with $d_2>0$. In the following, we briefly present some facts about the Teichm\"uller space $X$ that are needed in this work.

Given $X\sim \Gamma\backslash\mathbb{U}$, let $L^{\infty}(X)$ be the subspace of $B_{-1,1}(X)$  consists of $\mu$ with finite sup-norm 
$$\Vert\mu\Vert_{\infty}=\sup_{z\in \mathbb{U}}|\mu(z)|.$$ It is called the space of Beltrami differentials on $X$.  Let $L^{\infty}(X)_1$ be the subset of $L^{\infty}(X)$ consists of those $\mu$ such that $$\Vert\mu\Vert_{\infty}<1. $$ Given  $\mu\in L^{\infty}(X)_1$, extend $\mu$ so that it is zero outside $\mathbb{U}$. Let $f^{\mu}:\hat{\mathbb{C}}\rightarrow\hat{\mathbb{C}}$ be the unique quasiconformal mapping such that
\begin{gather*}
f^{\mu}_{\bar{z}}(z)=\mu(z)f^{\mu}_z(z),\\
f^{\mu}(p_i)=p_i,\hspace{1cm} i=1,2, 3,
\end{gather*}where $p_1, p_2, p_3$ are the fixed points involved in the normalization of $\Gamma$. 

Let $\tilde{\mu}$ be the extension of $\mu$ to $\mathbb{C}$ by reflection. Namely, for $z\in \mathbb{L}$, define
$$\tilde{\mu}(z)=\overline{\mu(\bar{z})},$$ and let $\tilde{\mu}(z)=0$ for $z\in \mathbb{R}$. Let $f_{\mu}: \hat{\mathbb{C}}\rightarrow\hat{\mathbb{C}}$ be the unique quasiconformal mapping such that
\begin{gather*}
(f_{\mu})_{\bar{z}}(z)=\tilde{\mu}(z)(f_{\mu})_z(z),
\end{gather*}and $f_{\mu}$ fixes $0$, 1 and $\infty$. Then $f_{\mu}$ maps $\mathbb{U}$ to $\mathbb{U}$, and $f^{\mu}\circ f_{\mu}^{-1}$ is holomorphic on $\mathbb{U}$. 
For each $\gamma\in \Gamma$,
$f_{\mu}\circ\gamma\circ (f_{\mu})^{-1}$ is also an element of $\text{PSL}\,(2,\mathbb{R})$. Thus $$\Gamma_{\mu}=f_{\mu}\circ \Gamma\circ f_{\mu}^{-1}$$ is a Fuchsian group and we can identify the Riemann surface 
$X_{\mu}=\Gamma_{\mu}\backslash\mathbb{U}$ with the Beltrami differential $\mu$. The group $$\Gamma^{\mu}=f^{\mu}\circ \Gamma\circ (f^{\mu})^{-1}$$ is not a subgroup of $\text{PSL}\,(2,\mathbb{R})$, but it is a subgroup of $\text{PSL}\,(2,\mathbb{C})$. If $\Omega=f^{\mu}(\mathbb{U})$, then $\Omega$ is a component of the domain of discontinuity of the group action of $\Gamma^{\mu}$ on $\hat{\mathbb{C}}$. $X^{\mu}=\Gamma^{\mu}\backslash\Omega$ is a Riemann surface biholomorphic to $X_{\mu}$.

Two Beltrami differentials $\mu$ and $\nu$ in $L^{\infty}(X)_1$ are said to be equivalent, denoted by $\mu\sim \nu$, if and only if $f_{\mu}(z)=f_{\nu}(z)$ for all $z\in S^1$, if and only if  $f^{\mu}(z)=f^{\nu}(z)$ for all $z$ on the lower half plane $\mathbb{L}$. When $\mu\sim \nu$, $\Gamma_{\mu}=\Gamma_{\nu}$. The Teichm\"uller space $T(X)$ is defined as the set of equivalence classes of Beltrami differentials, namely,
$$T(X)=L^{\infty}(X)_1/\sim.$$Each point $[\mu]$ of $T(X)$   corresponds to a Riemann surface $X_{\mu}=\Gamma_{\mu}\backslash \mathbb{U}$. There is a complex manifold structure on $T(X)$ so that the canonical mapping
$$\Phi: L^{\infty}(X)_1\rightarrow T(X)$$ is a holomorphic submersion. With respect to this complex structure, the surface $X^{\mu}$ varies holomorphically with respect to $\mu$ (the moduli).

Given any two Riemann surfaces $X_1=\Gamma_1\backslash\mathbb{U}$ and $X_2=\Gamma_2\backslash\mathbb{U}$ which have the same signature $(g;q;,m_1, m_2, \ldots, m_v)$, there exists a normalized quasiconformal mapping $w:\hat{\mathbb{C}}\rightarrow\hat{\mathbb{C}}$ such that 
$$\Gamma_2=w \circ\Gamma_1\circ w^{-1}.$$$w$ fixes $\mathbb{U}$, $\mathbb{R}$ and $\mathbb{L}$ respectively. Let $\tilde{\lambda}$ be the Beltrami differential of $w$, namely,
$$\tilde{\lambda} =\frac{w_{\bar{z}}}{w_z},$$ and let 
$$\lambda =\tilde{\lambda}\bigr|_{\mathbb{U}}$$be the restriction of $\tilde{\lambda}$ to $\mathbb{U}$. Then $\lambda\in L^{\infty}(X_1)_1$. If $\nu\in L^{\infty}(X_2)_1$ and $f^{\nu}$ is the corresponding normalized quasiconformal mapping, $f^{\nu}\circ w$ is a quasiconformal mapping with Beltrami differential $\mu\in L^{\infty}(X)_1$, where $\mu$ is related to $\nu$ by
$$\nu =\left(\frac{\mu-\lambda}{1-\mu\overline{\lambda}}\frac{w_z}{\overline{w_z}}\right)\circ w^{-1}.$$This defines a mapping $T(X_1)\rightarrow T(X_2)$ taking $\mu$ to $\nu$,   establishes a biholomorphism between the Teichm\"uller spaces $T(X_1)$ and $T(X_2)$.

Let $\Omega_{-1, 1}(X)$ be the space of harmonic Beltrami differentials on $X$. They are functions $\mu: \mathbb{U}\rightarrow \hat{\mathbb{C}}$ that can be written as
$$\mu(z)=y^2\overline{\phi(z)}$$ for some $\phi\in \Omega_2(X)$. The space
$\Omega_{-1,1}(X)$ is a vector space of complex dimension
$$d_2=3g-3+q+v.$$
A local coordinate system at the point $[0]\in T(X)$ can be established by choosing a basis $\{\mu_1, \ldots, \mu_{d_2}\}$ of $\Omega_{-1,1}(X)$, and mapping a point $(\varepsilon_1, \ldots, \varepsilon_{d_2})$ in a neighbourhood of $\mathbf{0}$ in $\mathbb{C}^{d_2}$ to $T(X)$ by 
$$(\varepsilon_1, \ldots, \varepsilon_{d_2})\mapsto \varepsilon_1\mu_1+\ldots+ \varepsilon_{d_2}\mu_{d_2}.$$As such, the tangent space of $T(X)$ at $X$ can be identified with $\Omega_{-1, 1}(X)$. 

At any other point $[\lambda]$ on the Teichm\"uller space $T(X)$, we can use the biholomorphism between $T(X)$ and $T(X_{\lambda})$ to establish local coordinates at $[\lambda]$. Henceforth, when we consider derivative on $T(X)$, it is sufficient to compute locally at the origin. 
\smallskip
\subsection{Variational Formulas}\label{variationalformula}~

Given a function $F: T(X)\rightarrow\mathbb{C}$ and  $\mu\in\Omega_{-1, 1}(X)$, $\pa_{\mu}F(X)$ and $ {\pa}_{\bar{\mu}}F(X)$ are defined as 
\begin{align*}
\pa_{\mu}F(X)=&\left.\frac{\pa}{\pa\varepsilon}\right|_{\varepsilon=0}F(X^{\varepsilon\mu}),\\
\pa_{\bar{\mu}}F(X)=&\left.\frac{\pa}{\pa\bar{\varepsilon}}\right|_{\varepsilon=0}F(X^{\varepsilon\mu}).
\end{align*}
Let $\mathscr{F}=\left\{\phi:\hat{\mathbb{C}}\rightarrow\hat{\mathbb{C}}\right\}$ be the space of functions from $\hat{\mathbb{C}}$ to $\hat{\mathbb{C}}$. A function $\Phi: T(X)\rightarrow \mathscr{F} $ satisfying
$$\phi^{\mu}=\Phi([\mu])\in B_{n,m}(X_{\mu}) $$  defines  a family of $(n,m)$ tensors on $T(X)$. The Lie derivatives of the family in the direction of $\mu $ and $\bar{\mu}$ are defined as 
\begin{align*}
(L_{\mu})\phi(z)=\left.\frac{\pa }{\pa\varepsilon}\right|_{\varepsilon=0}(f^{\varepsilon\mu})^*\phi^{\varepsilon\mu}(z),\\
(L_{\bar{\mu}})\phi(z)=\left.\frac{\pa }{\pa\bar{\varepsilon}}\right|_{\varepsilon=0}(f^{\varepsilon\mu})^*\phi^{\varepsilon\mu}(z),
\end{align*} where
$$(f^{\varepsilon\mu})^*\phi^{\varepsilon\mu}(z)=\phi^{\varepsilon\mu}\circ f^{\varepsilon\mu}(z)(f_z^{\varepsilon\mu}(z))^n\overline{(f_z^{\varepsilon\mu}(z))^m}$$is the pull-back of $\phi^{\varepsilon\mu}$ from $X^{\varepsilon\mu}$ to $X$.

Given $\mu\in \Omega_{-1, 1}(X)$, the function $J_{\mu}=f^{\mu}\circ f_{\mu}^{-1}: \mathbb{U}\rightarrow\Omega$ establishes a biholomorphism between $X_{\mu}$ and $X^{\mu}$.
Let $\rho^{\mu}(z)$ be the hyperbolic metric density on $X^{\mu}$. Then $$\rho^{\mu}(z)=\frac{|\left(J_{\mu}^{-1}\right)_z(z)|^2}{\left[\text{Im}\left(J_{ \mu}^{-1}(z)\right)\right]^2}.$$  It follows that
  $$(f^{\mu})^*\rho^{\mu}(z)=\frac{|\left(f_{\mu} \right)_z(z)|^2}{\left[\text{Im}\left(f_{ \mu} (z)\right)\right]^2}.$$ 

As we mentioned earlier, the complex structure on $T(X)$ is such that function $f^{\varepsilon\mu}$ varies holomorphically with respect to moduli, namely,
\begin{align*}
\frac{\pa}{\pa\bar{\varepsilon}}f^{\varepsilon\mu}=0.
\end{align*}

In some cases, it is still useful to work with $f_{\mu}$ which fixes the upper half plane rather than $f^{\mu}$ which is holomorphic with respect to moduli.
From classical theory of quasiconformal mappings, we know that
\begin{align}
F[\mu](z)=\left.\frac{\pa}{\pa\varepsilon}\right|_{\varepsilon=0}f_{\varepsilon\mu}(z)=-\frac{1}{\pi}\int\limits_{\mathbb{U}}R(z,\zeta)\mu(\zeta)d^2\zeta,\label{eq2_17_9}
\end{align}where
$$R(z,\zeta)=\frac{1}{\zeta-z}+\frac{z-1}{\zeta}-\frac{z}{\zeta-1}.$$

 There are a few useful variational formulas in \cite{Ahlfors_curvature}.
Given $\mu\in\Omega_{-1,1}(X)$, let
\begin{align*}
\mathscr{L}_{ \mu}(z,w)=\frac{(f_{ \mu})_z (z)(f_{ \mu})_w(w)}{\left(f_{ \mu}(z)- f_{ \mu}(w)\right)^2}\quad\text{and}\quad
\mathscr{M}_{ \mu}(z, w)=\frac{(f_{ \mu})_z (z)\overline{(f_{ \mu})_w(w)}}{\left(f_{ \mu}(z)-\overline{ f_{ \mu}(w)}\right)^2}.
\end{align*}
The variations of these two kernels are given in the following proposition.
\begin{proposition}[\cite{Ahlfors_curvature}]\label{prop_vary} Let $\mu\in\Omega_{-1,1}(X)$, and let $z, w\in \mathbb{U}$. Then
\begin{align*}
 \frac{\pa}{\pa\varepsilon} \mathscr{L}_{\varepsilon \mu}(z,w)=&-\frac{1}{\pi}\int_{\mathbb{U}} \mu(\zeta) \mathscr{L}_{\varepsilon\mu}(z,\zeta)\mathscr{L}_{\varepsilon\mu}(\zeta, w)d^2\zeta,\end{align*}
\begin{align*}
 \frac{\pa}{\pa\varepsilon} \mathscr{M}_{\varepsilon \mu}(z, w)=&-\frac{1}{\pi}\int_{\mathbb{U}} \mu(\zeta) \mathscr{L}_{\varepsilon\mu}(z,\zeta)\mathscr{M}_{\varepsilon\mu}(\zeta, w)d^2\zeta,\end{align*}
\begin{align*}
 \frac{\pa}{\pa\bar{\varepsilon}} \mathscr{L}_{\varepsilon \mu}(z,w)=&-\frac{1}{\pi}\int_{\mathbb{U}} \overline{\mu(\zeta) }\mathscr{M}_{\varepsilon\mu}(z,\zeta)\mathscr{M}_{\varepsilon\mu}(w,\zeta)d^2\zeta,\end{align*}
\begin{align*}
 \frac{\pa}{\pa\bar{\varepsilon}} \mathscr{M}_{\varepsilon \mu}(z,w)=&-\frac{1}{\pi}\int_{\mathbb{U}} \overline{\mu(\zeta) }\mathscr{M}_{\varepsilon\mu}(z,\zeta)\overline{\mathscr{L}_{\varepsilon\mu}( \zeta, w)}d^2\zeta.\end{align*}
\end{proposition}

Finally, we need to quote the following important result in \cite{Wolpert_ChernForm}.

\begin{proposition}\label{vary_differential}
If $\mu$ and $\nu$ are in $\Omega_{-1,1}(X)$, then  $L_{\bar{\nu}}\mu$ is orthogonal to the space $\Omega_{-1,1}(X)$ under the Weil-Petersson metric, and it is given by
\begin{align}\label{eq2_22_12}L_{\bar{\nu}}\mu=-4\frac{\pa}{\pa\bar{z}}\rho(z)^{-1}\frac{\pa}{\pa\bar{z}}(\Delta_0+2)^{-1}(\mu\bar{\nu}).\end{align}
 
\end{proposition}
At first sight, one might thought that this formula is awkward since   the complex structure of $T(X)$ would lead one to expect  that $\mu$ defines a holomorphic vector field. However,  since $L_{\bar{\nu}}\mu$ is orthogonal to $\Omega_{-1,1}(X)$, it is indeed "zero" as a tangent vector of $T(X)$.  This term  is needed when one needs to do a second variation with respect to $\bar{\nu}$, and the first variation has not been projected to the space of holomorphic quadratic differentials $\Omega_{2}(X)$, the dual space of $\Omega_{-1,1}(X)$.

\smallskip
\subsection{Weil-Petersson Metric and Takhtajan-Zograf Metric}~

On the Teichm\"uller space $T(X)$, there are a few metrics that arise naturally in the local index theorem.

Identifying the tangent space of $T(X)$ at $X$ with $\Omega_{-1,1}(X)$, the space of harmonic Beltrami differentials on $X$, the Weil-Petersson metric on $T(X)$ is the metric induced by the natural inner product
$$\langle \mu, \nu\rangle_{\text{WP}}=\int_X\mu(z)\overline{\nu(z)}dA(z).$$The integral over $X$ is actually over a fundamental domain $F$ of $X$ on $\mathbb{U}$. When there is no confusion, we will continue to write it as an integral over $X$.

The Weil-Petersson metric is a K\"ahler metric with symplectic form given by
$$\omega_{\text{WP}}(\mu,\nu)=\frac{i}{2}\langle \mu, \nu\rangle_{\text{WP}}.$$

Now we define the  Takhtajan-Zograf metric. 
If $q\geq 1$, there are $q$ cusps on the Riemann surface $X$, each of which corresponds to a parabolic generator $\kappa_i$. As in Section \ref{Laplace}, there is an Eisenstein series $E_i(z,s)=E_i(z,s;0)$ associated with this cusp. The Takhtajan-Zograf (TZ) metric associated with this cusp is given by
\begin{align*}
\langle\mu, \nu\rangle_{\text{TZ}, i}^{\text{cusp}}=\int_X \mu(z)\overline{\nu(z)}E_i(z,2)dA(z).
\end{align*}The parabolic Takhtajan-Zograf metric on $T(X)$ is simply the sum of the TZ metric associated to each cusp. Namely,
$$\langle\mu, \nu\rangle_{\text{TZ}}^{\text{cusp}}=\sum_{i=1}^q\langle\mu, \nu\rangle_{\text{TZ}, i}^{\text{cusp}}.$$This is also a K\"ahler metric with symplectic form
$$\omega_{\text{TZ}}^{\text{cusp}}(\mu,\nu)=\frac{i}{2}\langle \mu, \nu\rangle_{\text{TZ}}^{\text{cusp}}.$$

Let $G(z,w)$ be the kernel of the operator $2(\Delta_0+2)^{-1}$ on the space of functions on $X$. Using the coordinates on $\mathbb{U}$, it is given by
\begin{align}\label{eq3_3_1}G(z,w)=\sum_{\gamma\in \Gamma}\mathcal{G}(\gamma z, w),\end{align} where
\begin{align}\label{eq3_1_2}\mathcal{G}(z,w)=\frac{2u+1}{2\pi}\log\frac{u+1}{u}-\frac{1}{\pi},\end{align}and $u=u(z,w)$ is the point-pair invariant 
\begin{align}\label{eq1_24_1}
u(z,w)=\frac{|z-w|^2}{4\text{Im}\,z\;\text{Im}\,w}=-\frac{(z-w)(\bar{z}-\bar{w})}{(z-\bar{z})(w-\bar{w})}.
\end{align}
 If $v\geq 1$, there are $v$ ramification points on the Riemann surface $X$, each of which corresponds to an elliptic generator $\tau_j$. Let $w_j$ be the fixed point of $\tau_j$ on the upper half plane. The Takhtajan-Zograf (TZ) metric associated with this ramification point is given by
\begin{align}\label{eq3_4_2}
\langle\mu, \nu\rangle_{\text{TZ}, j}^{\text{ell}}=2\int_X G(w_j, z)\mu(z)\overline{\nu(z)}dA(z).
\end{align} This is also a K\"ahler metric with symplectic form
$$\omega_{\text{TZ},j}^{\text{ell}}(\mu,\nu)=\frac{i}{2}\langle \mu, \nu\rangle_{\text{TZ}}^{\text{ell},j}.$$

\smallskip
\subsection{Local Index Theorem}\label{localindex}~

Given a positive integer $n$, the space $\Omega_n(X)$ of integrable holomorphic $n$-differentials on $X$ is a complex vector space of dimension $d_n$. Since we only consider Riemann surfaces $X$ with $d_2>0$ so that the Teichm\"uller space $T(X)$ has positive dimension, and eq.\ \eqref{eq3_4_1} shows that $d_n$ is an increasing function of $n$, $d_n>0$ for all $n\geq 2$. However, the complex dimension of holomorphic one-differentials has dimension $g$, the genus of the Riemann surface $X$. So it can be zero.

When $d_n>0$, choose a basis $\{\phi_1, \ldots, \phi_{d_n}\}$ for each Riemann surface $X^{\mu}$ in $T(X)$ so that each $\phi_j, 1\leq j\leq d_n$ varies holomorphically on $T(X)$. This means that for any $\mu\in \Omega_{-1,1}(X)$, 
$$\frac{\pa}{\pa\bar{\varepsilon}}\left[(f^{\varepsilon\mu})^*\phi_j^{\varepsilon\mu}\right](z)=0.$$
One can use Bers integral operator to prove that such a family of basis exists \cite{Bers_integral, Mcintyre_Teo}.  

Define the Gram matrix $N_n$, a $d_n\times d_n$ matrix, so that its $(k,l)$-entry is given by
$$(N_n)_{k,l} =\langle\phi_k, \phi_l\rangle=\int_X \phi_k(z)\overline{\phi_l)(z)}\rho(z)^{-n}dA(z).$$ $N_n$ is a positive definite Hermittian matrix. The logarithmic derivative of its determinant, $\log\det N_n$, defines a real-valued function on $T(X)$. If we choose a different basis $\{\psi_1, \ldots, \psi_{d_n}\}$ that also varies holomorphically with respect to moduli, then there is a $d_n\times d_n$ matrix  $P$ so that
$$\psi_k=\sum_{l=1}^{d_n}P_{k,l}\phi_{l}.$$ If
$ \widetilde{N_n}$ is the Gram  matrix defined by the basis 
$\{\psi_1, \ldots, \psi_{d_n}\}$, we find that
$$\widetilde{N}_n=PN_nP^*,$$ where $P^*$ is the Hermittian conjugate of $P$. It follows that
$$\log\det \widetilde{N}_n=\log \det P+\overline{\log \det P}+\log \det N_n.$$ 
Since both  $\{\phi_1, \ldots, \phi_{d_n}\}$ and $\{\psi_1, \ldots, \psi_{d_n}\}$ vary holomorphically  on $T(X)$, each of  the entries in $P$ is a holomorphic function on $T(X)$. Hence, $\log\det P+\overline{\log\det P}$ is a harmonic function on $T(X)$. This shows that
the difference $$\log\det\widetilde{N}_n-\log\det N_n$$ is a harmonic function. This implies that
for any $\mu$ and $\nu$ in $\Omega_{-1, 1}(X)$,
$$\pa_{\mu}\pa_{\bar{\nu}}\log\det \widetilde{N}_n=\pa_{\mu}\pa_{\bar{\nu}}\log\det N_n.$$

The local index theorem  \cite{TZ_index_1, TZ_index_2, TZ_index_3}  is a theorem about the $\pa\bar{\pa}$ of $$\log \frac{\det \Delta_n}{\det N_n}$$ on $T(X)$. This can be interpreted as the first Chern class of a determinant line bundle on $T(X)$ \cite{TZ_index_2, TZ_index_3}. In this work, we  only consider the case where $n$ is a positive integer.
 
\begin{theorem}[Local Index Theorem \cite{TZ_index_3}]~\\Let $n$ be a positive integer and let $X$ be a cofinite hyperbolic Riemann surface such that the Teichm\"uller space $T(X)$ has positive dimension. On the Teichm\"uller space $T(X)$, 
\begin{align*}
\pa_{\mu}\pa_{\bar{\nu}}\log\frac{\det \Delta_n}{\det N_n}=&\frac{6n^2-6n+1}{12\pi}\langle \mu,\nu\rangle_{\text{WP}}-\frac{\pi}{9}\langle \mu, \nu\rangle_{\text{TZ}}^{\text{cusp}}-\sum_{j=1}^v\mathfrak{B}(m_j,n)\langle \mu, \nu\rangle_{\text{TZ}, j}^{\text{ell}}.
\end{align*}Here $\mathfrak{B}(m,n)$ is a constant depending on $m$ and $n$. (See Theorem \ref{thm2_24_2}.)

\end{theorem}

In the following sections, we are going to give a proof of this theorem that are more computational. In some sense, the  $n=1$ case should be the easiest one. However, in our approach which uses extensively Poincar$\acute{\text{e}}$ series (which is the sum of a function over the Fuchsian group elements), such series   fails to converge in the $n=1$ case. Hence, it in the following, we will concentrate on the case where $n\geq 2$. Whenever there is a problem with the $n=1$ case, we will discuss it in Section \ref{n1} an alternative method to solve the problem.

\smallskip

\section{Some Important Formulas}\label{importantformula}
Recall that  $\mu$ is a harmonic differential if there exists a holomorphic function $\phi$ such that
$$\mu(z)=y^2\overline{\phi(z)}.$$
 In this section, we want to derive an alternative expression for the integral
$$-\frac{1}{\pi}\int_{\mathbb{U}}\frac{\mu(u)}{(u-z)^2(u-\bar{w})^2}d^2u$$which by Proposition \ref{prop_vary} is   the variation of the kernel
$$\mathcal{H}(z,w)=\frac{1}{(z-\bar{w})^2}.$$ Proposition \ref{prop2_2_1} is crucial in our subsequent computations.

\begin{proposition}\label{prop2_2_1}
Given that $\mu$ is a harmonic Beltrami differential on $\mathbb{U}$, 
\begin{equation}\label{eq2_24_8}
\begin{split}\int_{\mathbb{U}}\frac{\mu(u)}{(u-z)^2(u-\bar{w})^2}d^2u=&\int_{\mathbb{U}}\mu(\zeta)\left\{\frac{(\bar{w}-\bar{z})^2}{(\zeta-\bar{z})^2(\zeta-\bar{w})^2(z-\bar{w})^2}\right.\\
&\hspace{1cm}\left.+\frac{2(\bar{w}-\bar{z})^2(z-\bar{z})}{(\zeta-\bar{z})^3(\zeta-\bar{w})(z-\bar{w})^3} \right\}d^2\zeta.
\end{split}\end{equation}
 
\end{proposition}
\begin{proof}
Using the definition of harmonic Beltrami differentials and \eqref{eq3_1_1}, we find that
\begin{align}\label{eq2_6_6}
\mu(z)=-\frac{3(z-\bar{z})^2}{\pi}\int_{\mathbb{U}}\frac{\mu(\zeta)}{(\zeta-\bar{z})^4}d^2\zeta.
\end{align}Therefore, 
$$\int_{\mathbb{U}}\frac{\mu(u)}{(u-z)^2(u-\bar{w})^2}d^2u=\int_{\mathbb{U}}\mu(\zeta)\Xi(\zeta, z, w)d^2\zeta,$$where
\begin{align*}
\Xi(\zeta, z, w)=-\frac{3}{\pi}\int_{\mathbb{U}}\frac{(u-\bar{u})^2}{(z-u)^2(u-\bar{w})^2(\zeta-\bar{u})^4}d^2u.
\end{align*}It is easier to compute this integral if we use coordinates in the unit disc $\mathbb{D}$. Notice that the mapping
\begin{align}\label{eq2_2_2} h(z)= i\frac{1+z}{1-z}\end{align} maps   $\mathbb{D}$ biholomorphically onto $\mathbb{U}$. Under this mapping,
\begin{align}\label{eq2_7_1}\frac{h'(\zeta)h'(\eta)}{\left(h(\zeta)-h(\eta)\right)^2}=\frac{1}{(\zeta-\eta)^2},\hspace{1cm}\frac{h'(\zeta)\overline{h'(\eta)}}{\left(h(\zeta)-\overline{h(\eta)}\right)^2}=-\frac{1}{(1-\zeta\bar{\eta})^2}.\end{align}Therefore,
\begin{align*}
\widetilde{\Xi}(\zeta, z, w)=&\Xi\left(h((\zeta), h(z), h(w)\right)h'(\zeta)^2h'(z)\overline{h'(w)}\\
=&-\frac{3}{\pi}\int_{\mathbb{D}}\frac{(1-|u|^2)^2}{(z-u)^2(1-u\bar{w})^2(1-\zeta\bar{u})^4}d^2u.
\end{align*}If $z$ is a point in $\mathbb{D}$ and 
\begin{align}\label{eq2_24_11}\sigma(v)=\frac{v-z}{1-v\bar{z}},\end{align}$\sigma$ is a linear fractional transformation that maps $\mathbb{D}$ onto $\mathbb{D}$, and $\sigma(z)=0$. Moreover,
$$\widetilde{\Xi}(\zeta, z, w)=\widetilde{\Xi}\left(\sigma(\zeta), \sigma(z), \sigma(w)\right)\sigma'(\zeta)^2\sigma'(z)\overline{\sigma'(w)}.$$ Let us first compute
$$\widetilde{\Xi}(\zeta_1, 0, w_1)=-\frac{3}{\pi}\int_{\mathbb{D}}\frac{(1-|u|^2)^2}{u^2(1-u\overline{w_1})^2(1-\zeta_1\bar{u})^4}d^2u.$$
Using Taylor expansion, the fact that
\begin{equation}\label{eq2_2_3}\int_{\mathbb{D}}u^m\bar{u}^nd^2u=0\hspace{1cm}\text{if}\;m\neq n,\end{equation}  and the fact that when $m$ is a positive integer,
\begin{align}\label{eq2_2_4}
\frac{1}{(1-x)^m}=\sum_{k=0}^{\infty} \frac{(k+m-1)!}{k\,!\;(m-1)!}x^k,
\end{align}we have
\begin{align*}
\widetilde{\Xi}(\zeta_1, 0, w_1)=&-\frac{1}{2\pi}\sum_{k=0}^{\infty} (k+1)(k+2)(k+3)^2\zeta_1^k\overline{w}^{k+2}\int_{\mathbb{D}}|u|^{2k}(1-|u|^2)^2d^2u\\
=&-\sum_{k=0}^{\infty}  (k+3)\zeta_1^k\overline{w}^{k+2}\\
=&-\frac{\overline{w_1}^2}{(1-\zeta_1\overline{w_1})^2}-\frac{2\overline{w_1}^2}{1-\zeta_1\overline{w_1}}.
\end{align*}This gives
\begin{align*}
\widetilde{\Xi}(\zeta, z, w)=-\frac{(\bar{w}-\bar{z})^2}{(1-\zeta\bar{z})^2(1-\zeta\bar{w})^2(1-z\bar{w})^2}-\frac{2(\bar{w}-\bar{z})^2(1-|z|^2)}{(1-\zeta\bar{z})^3(1-\zeta\bar{w})(1-z\bar{w})^3}.
\end{align*}Hence,
$$\Xi(\zeta, z, w)=\frac{(\bar{w}-\bar{z})^2}{(\zeta-\bar{z})^2(\zeta-\bar{w})^2(z-\bar{w})^2}+\frac{2(\bar{w}-\bar{z})^2(z-\bar{z})}{(\zeta-\bar{z})^3(\zeta-\bar{w})(z-\bar{w})^3}.$$
 \end{proof}

Now we are going to use the formula \eqref{eq2_24_8} to derive some well-known results for the hyperbolic metric density $\rho$. By definition, 
$$L_{\mu}\rho=L_{\mu}\left[\frac{-4}{(z-\bar{z})^2}\right]=\frac{4}{\pi}\int_{\mathbb{U}}\frac{\mu(u)}{(u-z)^2(u-\bar{z})^2}d^2u.$$ 
Setting $w=z$ in \eqref{eq2_24_8} immediately gives the well known result of Wolpert  \cite{Wolpert_ChernForm} that 
$$L_{\mu}\rho=0.$$In \cite{Memoir}, we have re-derived another important result of Wolpert  \cite{Wolpert_ChernForm} which says that 
 \begin{align}\label{eq2_24_9}L_{\mu}L_{\bar{\nu}}\rho=2\rho(\Delta_0+2)^{-1}(\mu\bar{\nu}).\end{align}Here we give yet another derivation similar to the spirit of the proof given for Proposition \ref{prop2_2_1}.

We first start with an important proposition, which is Corollary 6.4 in \cite{Memoir}.
\begin{proposition}\label{integral2_24_1}
Let $\mu$ and $\nu$ be harmonic Beltrami differentials on $\mathbb{U}$,  and let $\mathcal{G}(z,w)$ be the kernel of $2(\Delta_0+2)^{-1}$ defined by \eqref{eq3_1_2}. Then
\begin{equation}\label{eq2_24_10}\begin{split}
&2(\Delta_0+2)^{-1}(\mu\bar{\nu})(z)\\=&\int_{\mathbb{U}}\mathcal{G}(z,w)\mu(w)\overline{\nu(w)}\rho(w)^{1-n}d^2w\\
=&\frac{2}{\pi^2}\int_{\mathbb{U}}\int_{\mathbb{U}}\mu(\zeta)\overline{\nu(\eta)}\left\{ \frac{(z-\bar{z})^2}{( \zeta -\bar{z})^2(\zeta-\bar{\eta})^2(z-\bar{\eta})^2}
+ \frac{(z-\bar{z})^3}{( \zeta -\bar{z})^3(\zeta-\bar{\eta}) (z-\bar{\eta})^3}\right\}d^2\zeta d^2\eta.
\end{split}\end{equation}
\end{proposition}
In \cite{Memoir}, we derived this as a corollary to the proof of \eqref{eq2_24_9}. Essentially we applied the operator $(\Delta_0+2)$ to both sides of \eqref{eq2_24_10} and showed that they give equal results. The conclusion follows from the invertibility of the operator $(\Delta_0+2)$.  Here we give a proof without relying on the later fact.
\begin{proof}
By \eqref{eq2_6_6},
\begin{align*}
2(\Delta_0+2)^{-1}(\mu\bar{\nu})(z)=&\frac{1}{\pi^2}\int_{\mathbb{U}}\int_{\mathbb{U}}\mu(\zeta)\overline{\nu(\eta)}\Lambda(z,\zeta, \eta)d^2\zeta d^2\eta,
\end{align*}where
\begin{align*}
\Lambda(z,\zeta,\eta)=&-36\int_{\mathbb{U}}\mathcal{G}(z,w)\frac{(w-\bar{w})^2}{(\zeta-\bar{w})^4(w-\bar{\eta})^4}d^2w.
\end{align*}

Using the same notation as in the proof of Proposition \ref{prop2_2_1}, let  
 \begin{align}\label{eq2_7_3}\check{\mathcal{G}}(z,w)=\mathcal{G}(h(z), h(w))=\frac{2\check{u}(z,w)+1}{2\pi}\log\frac{\check{u}(z,w)+1}{\check{u}(z,w)}-\frac{1}{\pi}, 
\end{align}where
$$\check{u}(z,w)=u(h(z), h(w))=\frac{|z-w|^2}{(1-|z|^2)(1-|w|^2)}.$$Then
\begin{align*}
\widetilde{\Lambda}(z,\zeta,\eta)=&\Lambda(h(z), h(\zeta), h(\eta))h'(\zeta)^2\overline{h'(\eta)}^2\\
=&36\int_{\mathbb{D}}\check{\mathcal{G}}(z,w)\frac{(1-|w|^2)^2}{(1-\zeta\bar{w})^4(1-w\bar{\eta})^4}d^2w.
\end{align*}
By a direct computation,
\begin{align*}
\check{\mathcal{G}}(0,w)=\frac{1}{\pi}\left(\frac{1+|w|^2}{2(1-|w|^2)}\log \frac{1}{|w|^2}-1\right).
\end{align*}
Therefore,
\begin{align*}
&\widetilde{\Lambda}(0,\zeta,\eta)\\=&\frac{1}{ \pi}\int_{\mathbb{D}}\left\{ \frac{(1-|w|^2) (1+|w|^2)}{2}\log\frac{1}{|w|^2}-(1-|w|^2)^2\right\}
\\&\hspace{3cm}\times\sum_{k=2}^{\infty}(k^3-k)(\zeta\bar{w})^{k-2}\sum_{m=2}^{\infty}(m^3-m)(w\bar{\eta})^{m-2}d^2w\\
=&\frac{1}{ \pi}\sum_{k=2}^{\infty}(k^3-k)^2(\zeta\bar{\eta})^{k-2}\int_{\mathbb{D}}\left\{ \frac{(1-|w|^2) (1+|w|^2)}{2}\log\frac{1}{|w|^2}-(1-|w|^2)^2\right\}
|w|^{2k-4}d^2w.
\end{align*}This integral is elementary and we find that
$$\widetilde{\Lambda}(0,\zeta,\eta)=2\sum_{k=2}^{\infty}k(\zeta\bar{\eta})^{k-2}=\frac{2}{(1-\zeta\bar{\eta})^2}+\frac{2}{1-\zeta\bar{\eta}}.$$Therefore, with $\sigma$ given by \eqref{eq2_24_11}, we find that
\begin{align*}
\widetilde{\Lambda}(z,\zeta,\eta)=&\widetilde{\Lambda}(\sigma(z), \sigma(\zeta), \sigma(\eta))\sigma'(\zeta)^2\overline{\sigma'(\eta)}^2\\
=&\frac{2(1-|z|^2)^2}{(1-\zeta\bar{z})^2(1-\zeta\bar{\eta})^2(1-z\bar{\eta})^2}+\frac{2(1-|z|^2)^3}{(1-\zeta\bar{z})^3(1-\zeta\bar{\eta})(1-z\bar{\eta})^3}.
\end{align*}The assertion follows.
\end{proof}

Now we proceed to derive the second variation of the hyperbolic metric density. By Proposition \ref{prop_vary},
\begin{align*}
L_{\mu}L_{\bar{\nu}}\rho=&-\frac{4}{\pi^2}\int_{\mathbb{U}}\int_{\mathbb{U}}\mu(\zeta)\overline{\nu(\eta)}\left\{\frac{1}{(z-\bar{\eta})^2(\zeta-\bar{\eta})^2(\zeta-\bar{z})^2}
+\frac{1}{(z-\zeta)^2(\zeta-\bar{\eta})^2(\bar{\eta}-\bar{z})^2}\right\}d^2\zeta d^2\eta.
\end{align*}By Proposition \ref{prop2_2_1},
\begin{align*}
&\int_{\mathbb{U}}\frac{\mu(\zeta)}{(z-\zeta)^2(\zeta-\bar{\eta})^2(\bar{\eta}-\bar{z})^2}d^2\zeta\\=&\int_{\mathbb{U}}\mu(\zeta)
\left\{\frac{1}{(\zeta-\bar{z})^2(\zeta-\bar{\eta})^2(z-\bar{\eta})^2}+\frac{2(z-\bar{z})}{(\zeta-\bar{z})^3(\zeta-\bar{\eta}) (z-\bar{\eta})^3}\right\}d^2\zeta.
\end{align*}It follows that 
\begin{align*}
L_{\mu}L_{\bar{\nu}}\rho=&-\frac{8}{\pi^2}\int_{\mathbb{U}}\int_{\mathbb{U}}\mu(\zeta)\overline{\nu(\eta)}\left\{\frac{1}{(z-\bar{\eta})^2(\zeta-\bar{\eta})^2(\zeta-\bar{z})^2}
+\frac{(z-\bar{z})}{(\zeta-\bar{z})^3(\zeta-\bar{\eta}) (z-\bar{\eta})^3}\right\}d^2\zeta d^2\eta.
\end{align*}By Proposition \ref{integral2_24_1}, this is equal to
$$2\rho (\Delta_0+2)^{-1}(\mu\bar{\nu})(z).$$

Let us collect these important results about the first and second variations of the hyperbolic metric density in the proposition below. 

\begin{proposition} \label{varymetric}
Given $\mu, \nu \in \Omega_{-1,1}(X)$, the first and second variations of the hyperbolic metric density $\rho$ is given by
$$L_{\mu}\rho=L_{\bar{\mu}}\rho=0,$$ and
$$L_{\mu}L_{\bar{\nu}}\rho=2\rho(\Delta_0+2)^{-1}(\mu\bar{\nu}).$$
\end{proposition}

\smallskip
\section{The Second Variation of $\log\det N_n$}\label{a1}
Let $n$ be a positive integer and assume that the space of integrable holomorphic $n$-differentials $\Omega_n(X)$ has positive dimension.
 Given a family of basis 
$\{\phi_1^{\mu}, \ldots, \phi_{d_n}^{\mu}\}$ of $\Omega_n(X^{\mu})$ that varies  holomorphically with
respect to moduli,  we want to  compute $\pa\bar{\pa}\log\det N_n$ on $T(X)$. In this section, we fixed a positive integer $n$, and abbreviate $N_n$ as $N$, $d_n$ as $d$ when no confusion arises.

\begin{lemma} Let  $n$ be a positive integer and let $\{\phi_1, \ldots, \phi_d\}$ be a basis of $\Omega_n(X)$. Define the 
  $d\times d$ matrix $N$ by
\begin{align}\label{eq1_26_1}N_{k,l}=\int_X\phi_k(z)\overline{\phi_l(z)}\rho(z)^{1-n}d^2z.\end{align} Then 
$N$ is a positive definite Hermitian matrix. In particular, $N$ is invertible.
\end{lemma}
\begin{proof}
From the definition, it is obvious that
$$\overline{N_{k,l}}=N_{l,k}.$$ This implies that $N=N^*$. Namely, $N$ is Hermitian. To prove the positive definiteness, let $\mathbf{c}=(c_1, \ldots, c_d)$ be a vector in $\mathbb{C}^d$. Then
\begin{align*}
\sum_{k=1}^d\sum_{l=1}^dc_k N_{k,l}\overline{c_l}=\int_X\left(\sum_{k=1}^dc_k\phi_k(z)\right)\overline{\left(\sum_{l=1}^dc_l\phi_l(z)\right)}\rho(z)^{1-n}d^2z=\left\langle \phi, \phi\right\rangle,
\end{align*}where
$$\phi=\sum_{k=1}^dc_k\phi_k.$$ This shows that
$$\sum_{k=1}^d\sum_{l=1}^dc_k N_{k,l}\overline{c_l}\geq 0$$ and equality holds if and only if $\phi=0$, if and only if $\mathbf{c}=\mathbf{0}$. Hence, $N$ is positive definite, and thus it is invertible.
\end{proof}

\begin{lemma} \label{lemma1_27_1} Let $n\geq 2$ and let $\{\phi_1, \ldots, \phi_d\}$ be a basis of $\Omega_n(X)$. If 
$N$ is the $d\times d$ matrix defined by \eqref{eq1_26_1},
then
\begin{align}\label{eq3_1_3} \sum_{k=1}^d\sum_{l=1}^d(N^{-1})_{kl}\phi_l(z)\overline{\phi_k(w)}=K_n(z,w), \end{align}
where $K_n(z,w)$ is the kernel of the projection operator $\mathscr{P}_n: H_n^2(X)\rightarrow \Omega_n(X)$.

When $n=1$, there is a special basis of $\Omega_1(X)$ and \eqref{eq3_1_3} is used to define the projection kernel $K_1(z,w)$ using this special basis. The formula \eqref{eq3_1_3} then holds for any other basis of $\Omega_1(X)$.
\end{lemma}
\begin{proof}
For fixed $z$, $K_n(z,w)$ is an integrable anti-holomorphic $n$-differential in $w$. Since $\left\{\overline{\phi_1}, \ldots, \overline{\phi_d}\right\}$ is a basis of the space of integrable anti-holomorphic $n$-differentials of $X$, we find that there exists $c_l(z)$, $l=1, \ldots, d$, such that
$$K_n(z,w)=\sum_{l=1}^dc_l(z)\overline{\phi_l(w)}.$$
Multiply both sides by $\phi_k(w)\rho(w)^{1-n}$ and integrate over $w\in X$, we find that
\begin{align*}
\sum_{l=1}^dN_{k,l}c_l(z)=&\sum_{l=1}^dc_l(z)\int_X\phi_k(w)\overline{\phi_l(w)}\rho(w)^{1-n}d^2w\\
=&\int_XK_n(z,w)\phi_k(w)\rho(w)^{1-n}d^2w\\
=&\phi_k(w).
\end{align*}
This shows that
$$N\begin{bmatrix} c_1(z)\\\vdots\\c_d(z)\end{bmatrix}=\begin{bmatrix} \phi_1(z)\\\vdots\\\phi_d(z)\end{bmatrix}.$$ Since $N$ is invertible, we find that
$$ \begin{bmatrix} c_1(z)\\\vdots\\c_d(z)\end{bmatrix}=N^{-1}\begin{bmatrix} \phi_1(z)\\\vdots\\\phi_d(z)\end{bmatrix}.$$In other words,
$$c_l(z)=\sum_{k=1}^d (N^{-1})_{lk}\phi_k(z).$$ The assertion follows.
\end{proof}

Now we consider the   second variations of $\log\det N_n$.

\begin{proposition}\label{prop1_26_1}
Let $n$ be a positive integer. If $\mu\in\Omega_{-1,1}(X)$, then 
\begin{equation}\label{v1}
\frac{\pa}{\pa \varepsilon}\log \det N^{\varepsilon\mu}= \text{Tr}\,\left((N^{\varepsilon\mu})^{-1}\frac{\pa}{\pa\varepsilon}N^{\varepsilon\mu}\right),\end{equation}
\begin{equation}\label{v2}\begin{split}
\frac{\pa^2}{\pa\bar{\varepsilon} \pa\varepsilon}\log\det N^{\varepsilon\mu}=& -\text{Tr}\,\left((N^{\varepsilon\mu})^{-1}\left(\frac{\pa}{\pa\bar{\varepsilon}} N^{\varepsilon\mu}\right)(N^{\varepsilon\mu})^{-1}\left(\frac{\pa}{\pa \varepsilon} N^{\varepsilon\mu}\right)\right)  \\&+\text{Tr}\,\left((N^{\varepsilon\mu})^{-1}\frac{\pa^2}{\pa\bar{\varepsilon}\pa\varepsilon}N^{\varepsilon\mu}\right).\end{split}
\end{equation}
\end{proposition}
\begin{proof}
Since $N$ is a positive definite Hermittian matrix, there is a unitary matrix $U$ and a diagonal matrix $D$ with positive entries such that
$$N=UDU^*.$$ 
If $\lambda_1, \ldots, \lambda_{d}$ are the diagonal entries of $D$, then 
$$\log\det N=\log\det D=\sum_{i=1}^{d}\log \lambda_i.$$Hence,
$$\frac{\pa}{\pa \varepsilon}\log \det N^{\varepsilon\mu}=\sum_{i=1}^{d} (\lambda_i^{\varepsilon\mu})^{-1}\frac{\pa}{\pa\varepsilon}\lambda_i^{\varepsilon\mu}=\text{Tr}\,\left((D^{\varepsilon\mu})^{-1}\frac{\pa}{\pa\varepsilon}D^{\varepsilon\mu}\right).$$
On the other hand,
\begin{align*}
\frac{\pa}{\pa\varepsilon}N^{\varepsilon\mu}=\left(\frac{\pa}{\pa\varepsilon}U^{\varepsilon\mu}\right)D^{\varepsilon\mu}(U^{\varepsilon\mu})^*+
U^{\varepsilon\mu}\left(\frac{\pa}{\pa\varepsilon}D^{\varepsilon\mu}\right) (U^{\varepsilon\mu})^*+
U^{\varepsilon\mu} D^{\varepsilon\mu}  \frac{\pa}{\pa\varepsilon}(U^{\varepsilon\mu})^*.
\end{align*}
Since $(N^{\varepsilon\mu})^{-1}=U^{\varepsilon\mu}\left(D^{\varepsilon\mu}\right)^{-1}(U^{\varepsilon\mu})^*$, and  $U^{\varepsilon\mu}(U^{\varepsilon\mu})^*=I$, we have
\begin{align*}
(N^{\varepsilon\mu})^{-1}\frac{\pa}{\pa\varepsilon}N^{\varepsilon\mu}=&U^{\varepsilon\mu}\left(D^{\varepsilon\mu}\right)^{-1}(U^{\varepsilon\mu})^*\left(\frac{\pa}{\pa\varepsilon}U^{\varepsilon\mu}\right)D^{\varepsilon\mu}(U^{\varepsilon\mu})^*+
U^{\varepsilon\mu}  \frac{\pa}{\pa\varepsilon}(U^{\varepsilon\mu})^*\\&+
U^{\varepsilon\mu}\left(D^{\varepsilon\mu}\right)^{-1}\left(\frac{\pa}{\pa\varepsilon}D^{\varepsilon\mu}\right) (U^{\varepsilon\mu})^*.
\end{align*}
Since $U^{\varepsilon\mu}(U^{\varepsilon\mu})^*=I$, we have
$$\left(\frac{\pa}{\pa\varepsilon}U^{\varepsilon\mu}\right)(U^{\varepsilon\mu})^*+U^{\varepsilon\mu}\frac{\pa}{\pa\varepsilon}(U^{\varepsilon\mu})^*=0.$$Using also $\text{Tr}\,(AB)=\text{Tr}\,(BA)$, we have
\begin{align*}
&\text{Tr}\,\left((N^{\varepsilon\mu})^{-1}\frac{\pa}{\pa\varepsilon}N^{\varepsilon\mu}\right)\\=&
\text{Tr}\,\left(\left(\frac{\pa}{\pa\varepsilon}U^{\varepsilon\mu}\right)D^{\varepsilon\mu}(U^{\varepsilon\mu})^*U^{\varepsilon\mu}\left(D^{\varepsilon\mu}\right)^{-1}(U^{\varepsilon\mu})^*\right) +
\text{Tr}\,\left(U^{\varepsilon\mu}  \frac{\pa}{\pa\varepsilon}(U^{\varepsilon\mu})^*\right)\\&+\text{Tr}\,\left(\left(D^{\varepsilon\mu}\right)^{-1}\left(\frac{\pa}{\pa\varepsilon}D^{\varepsilon\mu}\right) (U^{\varepsilon\mu})^*U^{\varepsilon\mu}\right)\\
=&\text{Tr}\,\left(\left(\frac{\pa}{\pa\varepsilon}U^{\varepsilon\mu}\right)(U^{\varepsilon\mu})^*+U^{\varepsilon\mu}\frac{\pa}{\pa\varepsilon}(U^{\varepsilon\mu})^*\right)+\text{Tr}\,\left((D^{\varepsilon\mu})^{-1}\frac{\pa}{\pa\varepsilon}D^{\varepsilon\mu}\right)
\\=&\frac{\pa}{\pa\varepsilon}\log \det N^{\varepsilon\mu}.
\end{align*}This proves \eqref{v1}.

For the second variation, we then have
\begin{align*}
\frac{\pa^2}{\pa\bar{\varepsilon} \pa\varepsilon}\log\det N^{\varepsilon\mu}= \text{Tr}\,\left(\frac{\pa}{\pa\bar{\varepsilon}}(N^{\varepsilon\mu})^{-1}\frac{\pa}{\pa\varepsilon}N^{\varepsilon\mu}\right)+ \text{Tr}\,\left((N^{\varepsilon\mu})^{-1}\frac{\pa^2}{\pa\bar{\varepsilon}\pa\varepsilon}N^{\varepsilon\mu}\right).
\end{align*}Since $(N^{\varepsilon\mu})^{-1} N^{\varepsilon\mu}=I$, we have
$$\left(\frac{\pa}{\pa\bar{\varepsilon}}(N^{\varepsilon\mu})^{-1}\right) N^{\varepsilon\mu}+(N^{\varepsilon\mu})^{-1}\left(\frac{\pa}{\pa\bar{\varepsilon}} N^{\varepsilon\mu}\right)=0,$$ which implies that
\begin{align}\label{eq3_9_6}\frac{\pa}{\pa\bar{\varepsilon}}(N^{\varepsilon\mu})^{-1} =-(N^{\varepsilon\mu})^{-1}\left(\frac{\pa}{\pa\bar{\varepsilon}} N^{\varepsilon\mu}\right)(N^{\varepsilon\mu})^{-1}.\end{align}
Hence,
\begin{align*}
\frac{\pa^2}{\pa\bar{\varepsilon} \pa\varepsilon}\log\det N^{\varepsilon\mu}=&  -\text{Tr}\,\left((N^{\varepsilon\mu})^{-1}\left(\frac{\pa}{\pa\bar{\varepsilon}} N^{\varepsilon\mu}\right)(N^{\varepsilon\mu})^{-1}\left(\frac{\pa}{\pa \varepsilon} N^{\varepsilon\mu}\right)\right)+ \text{Tr}\,\left((N^{\varepsilon\mu})^{-1}\frac{\pa^2}{\pa\bar{\varepsilon}\pa\varepsilon}N^{\varepsilon\mu}\right).
\end{align*}This proves \eqref{v2}.
\end{proof}

To find explicit formulas for the second variation of $\log\det N_n$, 
we need  the variation formula for the kernel $K_n(z,w)$.

\begin{proposition}\label{prop1_26_2}  Let $n\geq 2$. Given $\mu\in \Omega_{-1,1}(X)$, the first  variation of the projection kernel $K(z,w)=K_n(z,w)$ is given by
\begin{align}\label{eq2_22_1}
(L_{\mu}K_n)(z,w)=&(-1)^{n-1}\frac{2^{2n-2}(2n-1)n}{\pi^2}
\sum_{\gamma\in\Gamma} \int\limits_{\mathbb{U}}\mu(\zeta)\frac{\gamma'(z)^n}{\left( \gamma(z) -\overline{w}\right)^{2n-2}(\gamma z-\zeta)^2(\zeta-\bar{w})^2}d^2\zeta.
\end{align} In particular, this shows that $(L_{\mu}K_n)(z,w)$ is antiholomorphic in $w$.

\end{proposition}
\begin{proof}
On the surface $X_{\mu}=\Gamma_{\mu}\backslash\mathbb{U}$, the kernel $K(z,w)$ is given by
$$K_{ \mu}(z,w)=(-1)^n\frac{2^{2n-2}(2n-1)}{\pi}\sum_{\gamma_{\mu}\in\Gamma_{\mu}}\frac{\gamma_{\mu}'(z)^n}{(\gamma_{\mu} z-\bar{w})^{2n}}.$$Hence, on $X^{\mu}$,
$$K^{\mu}(z,w)=K_{\mu}(J_{\mu}^{-1}(z), J_{\mu}^{-1}(w)) (J_{\mu}^{-1})_z(z)^n\overline{\left(J_{\mu}^{-1}\right)_w(w)^n}.$$
Using the fact that $J_{\varepsilon\mu}^{-1}=f_{\varepsilon\mu}\circ (f^{\varepsilon\mu})^{-1}$, we find that
\begin{align*}(L_{\mu}K)(z,w) =&\left.\frac{\pa}{\pa\varepsilon}\right|_{\varepsilon=0}K^{\varepsilon\mu}(f^{\varepsilon\mu}(z), f^{\varepsilon\mu}(w))
f^{\varepsilon\mu}_z(z)^n\overline{f^{\varepsilon\mu}_w(w)^n}\\
=&\left.\frac{\pa}{\pa\varepsilon}\right|_{\varepsilon=0}K_{\varepsilon\mu}(f_{\varepsilon\mu}(z), f_{\varepsilon\mu}(w))
(f_{\varepsilon\mu})_z(z)^n\overline{(f_{\varepsilon\mu})_w(w)^n}.
\end{align*}Since
\begin{align*}
\gamma_{\vep\mu}\circ f_{\vep\mu} =f_{\vep\mu}\circ \gamma,
\end{align*}we have
\begin{align*}
&K_{\varepsilon\mu}(f_{\varepsilon\mu}(z), f_{\varepsilon\mu}(w))
(f_{\varepsilon\mu})_z(z)^n\overline{(f_{\varepsilon\mu})_w(w)^n}\\=&(-1)^n\frac{2^{2n-2}(2n-1)}{\pi}
\sum_{\gamma\in\Gamma}\left\{\frac{(f_{\varepsilon\mu})_z(\gamma z)\gamma'(z)\overline{(f_{\varepsilon\mu})_w(w)}}{\left(f_{\varepsilon\mu}(\gamma(z))-\overline{f_{\varepsilon\mu}(w)}\right)^2}\right\}^n.\end{align*}Using Proposition \ref{prop_vary}, we have
\begin{align*}
 \frac{\pa}{\pa\varepsilon} \frac{(f_{\varepsilon\mu})_z(\gamma z)\gamma'(z)\overline{(f_{\varepsilon\mu})_w(w)}}{\left(f_{\varepsilon\mu}(\gamma(z))-\overline{f_{\varepsilon\mu}(w)}\right)^2}=&-\frac{1}{\pi}\int\limits_{\mathbb{U}}\mu(\zeta)\frac{(f_{\varepsilon\mu})_z(\gamma z)\gamma'(z) (f_{\varepsilon\mu})_{\zeta}(\zeta)}{\left(f_{\varepsilon\mu}(\gamma(z))- f_{\varepsilon\mu}(\zeta)\right)^2}\frac{(f_{\varepsilon\mu})_{\zeta}(\zeta) \overline{(f_{\varepsilon\mu})_w(w)}}{\left(f_{\varepsilon\mu}(\zeta)-\overline{f_{\varepsilon\mu}(w)}\right)^2}d^2\zeta.
\end{align*}
Therefore,
\begin{equation}\label{eq1_29_6}\begin{split}
& \frac{\pa}{\pa\varepsilon} K_{\varepsilon\mu}(f_{\varepsilon\mu}(z), f_{\varepsilon\mu}(w))
(f_{\varepsilon\mu})_z(z)^n\overline{(f_{\varepsilon\mu})_w(w)^n}\\=&(-1)^{n-1}\frac{2^{2n-2}(2n-1)n}{\pi^2}
\sum_{\gamma\in\Gamma}\left\{\frac{(f_{\varepsilon\mu})_z(\gamma z)\gamma'(z)\overline{(f_{\varepsilon\mu})_w(w)}}{\left(f_{\varepsilon\mu}(\gamma(z))-\overline{f_{\varepsilon\mu}(w)}\right)^2}\right\}^{n-1}\\&\hspace{3cm}\times \int\limits_{\mathbb{U}}\mu(\zeta)\frac{(f_{\varepsilon\mu})_z(\gamma z)\gamma'(z) (f_{\varepsilon\mu})_{\zeta}(\zeta)}{\left(f_{\varepsilon\mu}(\gamma(z))- f_{\varepsilon\mu}(\zeta)\right)^2}\frac{(f_{\varepsilon\mu})_{\zeta}(\zeta) \overline{(f_{\varepsilon\mu})_w(w)}}{\left(f_{\varepsilon\mu}(\zeta)-\overline{f_{\varepsilon\mu}(w)}\right)^2}d^2\zeta.
\end{split}\end{equation}It follows that
\begin{align*}
(L_{\mu}K)(z,w)=&(-1)^{n-1}\frac{2^{2n-2}(2n-1)n}{\pi^2}
\sum_{\gamma\in\Gamma} \int\limits_{\mathbb{U}}\mu(\zeta)\frac{\gamma'(z)^n}{\left( \gamma(z) -\overline{w}\right)^{2n-2}(\gamma z-\zeta)^2(\zeta-\bar{w})^2}d^2\zeta.
\end{align*}
\end{proof}

As we discussed in Section \ref{variationalformula}, when we compute the variations of families of  differentials on the Teichm\"uller space $T(X)$ with respect to $\mu\in \Omega_{-1,1}(X)$, we need to apply pull-back by the map $f^{\varepsilon\mu}$ that is holomorphic on $\mathbb{L}$ to conform with the complex structure of $T(X)$. However, since the kernels we are using are defined in terms of coordinates on $\mathbb{U}$, we find that essentially we need to apply pull-back by the map $f_{\varepsilon\mu}$ that maps the upper-half plane onto itself. This point can be quite confusing to those first learn how to work with variations on Teichm\"uller spaces and so we show the details of the derivations in the proof of Proposition \ref{prop1_26_2}.

In the following, we are going to use similar technique to find variations of various families of kernels on the Teichm\"uller space without going into the details.

Now we use Proposition \ref{prop1_26_1}  to derive an integral formula for $\pa_{\mu}\pa_{\bar{\nu}}\log\det N_n$. 
\begin{theorem}\label{thm1_29_1}
Let $n$ be a positive integer. Given $\mu, \nu\in \Omega_{-1,1}(X)$,  $\pa_{\mu}\pa_{\bar{\nu}}\log\det N_n $ is given by 
\begin{equation}\label{eq1_29_5}\begin{split}
\pa_{\mu}\pa_{\bar{\nu}}\log\det N_n 
=&\int_X\int_X(L_{\mu}K_n)(z,w)(L_{\bar{\nu}}K_n)(w,z)\rho(w)^{1-n}\rho(z)^{1-n}d^2wd^2z\\
&+(1-n)\int_{X} K_n(z, z)\rho(z)^{1-n}\int_{X}G(z,w) \mu(w)\overline{\nu(w)}\rho(w)d^2w d^2z
\\&-\int_{X}K_n( z, z)\rho(z)^{1-n}\mu(z)\overline{\nu(z)}d^2z.
\end{split}\end{equation}

\end{theorem}
\begin{proof}
We only need to consider the case where $\mu=\nu$. The general case follows by polarization.
Recall that
\begin{align*}
(N^{\varepsilon\mu})_{kl}=&\int_{X^{\varepsilon\mu}}\phi_k^{\varepsilon\mu}(z)\overline{\phi_l^{\varepsilon\mu}(z)}\left(\rho^{\varepsilon\mu}(z)\right)^{1-n}d^2z\\
=&\int_{X}(f^{\varepsilon\mu})^*\phi_k^{\varepsilon\mu}(z)\overline{(f^{\varepsilon\mu})^*\phi_l^{\varepsilon\mu}(z)}\left((f^{\varepsilon\mu})^*\rho^{\varepsilon\mu}(z)\right)^{1-n}\left(1-|\varepsilon\mu(z)|^2\right)d^2z.
\end{align*}Each of the $\phi_k^{\varepsilon\mu}$, $1\leq k\leq d_n$ and $f^{\varepsilon\mu}$ varies holomorphically with respect to moduli. Together with Proposition \ref{varymetric} we find that
\begin{align*}
\left.\frac{\pa}{\pa\varepsilon}\right|_{\varepsilon=0}(N^{\varepsilon\mu})_{kl}=&\int_{X}\left(L_{\mu}\phi_k \right)(z)\overline{ \phi_l(z)} \rho (z)^{1-n} d^2z,\\
\left.\frac{\pa}{\pa\bar{\varepsilon}}\right|_{\varepsilon=0}(N^{\varepsilon\mu})_{kl}=&\int_{X} \phi_k(z) \overline{ (L_{ \mu}\phi_l)(z)} \rho (z)^{1-n} d^2z,
\end{align*}and
\begin{align*}
\left.\frac{\pa^2}{\pa\bar{\varepsilon}\pa\varepsilon}\right|_{\varepsilon=0}(N^{\varepsilon\mu})_{kl}=&\int_{X}\left(L_{\mu}\phi_k \right)(z)\overline{(L_{ \mu}\phi_l)(z)} \rho (z)^{1-n} d^2z\\
&+2(1-n) \int_X\phi_k(z)\overline{\phi_l(z)}\rho(z)^{1-n} \left[(\Delta_0+2)^{-1}(|\mu(z)|^2)\right](z) d^2z\\
&-\int_X\phi_k(z)\overline{\phi_l(z)}\rho(z)^{1-n}|\mu(z)|^2d^2z.
\end{align*}Using Proposition \ref{prop1_26_1}, we find that
\begin{align*}
\left.\frac{\pa^2}{\pa\bar{\varepsilon}\pa\varepsilon}\right|_{\varepsilon=0}\log\det N^{\varepsilon\mu}=T_1+T_2+T_3+T_4,\end{align*}where, by using Lemma \ref{lemma1_27_1}, we have
\begin{align*}T_1=&-\sum_{l=1}^d\sum_{k=1}^d\sum_{q=1}^d\sum_{m=1}^d(N^{-1})_{qm}\int_{X} \phi_m(z) \overline{ (L_{\mu}\phi_l)(z)} \rho (z)^{1-n} d^2z\\&\hspace{3cm}\times(N^{-1})_{lk}\int_{X}\left(L_{\mu}\phi_k \right)(w)\overline{ \phi_q(w)} \rho (w)^{1-n} d^2w\\
=&-\sum_{l=1}^d\sum_{k=1}^d(N^{-1})_{lk}\int_X\int_XK(z,w)\left(L_{\mu}\phi_k \right)(w) \overline{ (L_{ \mu}\phi_l)(z)} \rho(z)^{1-n}\rho(w)^{1-n}d^2wd^2z,\\ 
T_2=&
\sum_{l=1}^d\sum_{k=1}^d(N^{-1})_{lk}\int_{X}\left(L_{\mu}\phi_k \right)(z)\overline{(L_{\mu}\phi_l)(z)} \rho (z)^{1-n} d^2z,\\
T_3=&2(1-n)\sum_{l=1}^d\sum_{k=1}^d(N^{-1})_{lk} \int_X\phi_k(z)\overline{\phi_l(z)}\rho(z)^{1-n} \left[(\Delta_0+2)^{-1}(|\mu|^2)\right](z) d^2z,\\
=&(1-n)  \int_X K(z,z)\rho(z)^{1-n} \int_XG(z,w)|\mu(w)|^2\rho(w)d^2w  d^2z,\\
T_4=&-\sum_{l=1}^d\sum_{k=1}^d(N^{-1})_{lk}\int_X\phi_k(z)\overline{\phi_l(z)}\rho(z)^{1-n}|\mu(z)|^2d^2z,\\
=&-  \int_X K(z,z)\rho(z)^{1-n}  |\mu(z)|^2   d^2z.
\end{align*}
 Now we want to simplify $T_1+T_2$.
For any $1\leq k\leq d$,
\begin{align*}
\int_{X^{\varepsilon\mu}}K^{\varepsilon\mu}(z,w)\phi_k^{\varepsilon\mu}(w)\rho^{\varepsilon\mu}(w)^{1-n}d^2w=\phi^{\varepsilon\mu}_k(z).
\end{align*}Therefore,
\begin{align}\label{eq1_29_8}
\left(L_{\mu}\phi_k\right)(z)-\int_XK(z,w)(L_{\mu}\phi_k)(w)\rho(w)^{1-n}d^2w =\int_X(L_{\mu}K)(z,w)\phi_k(w)\rho(w)^{1-n}d^2w.
\end{align}  
Hence, for any $1\leq k,l\leq d$,
\begin{align*}
&\int_{X}\left(L_{\mu}\phi_k \right)(z)\overline{(L_{\mu}\phi_l)(z)} \rho (z)^{1-n} d^2z\\&-\int_X\int_XK(z,w)\left(L_{\mu}\phi_k \right)(w) \overline{ (L_{ \mu}\phi_l)(z)} \rho(z)^{1-n}\rho(w)^{1-n}d^2wd^2z\\
=&\int_X\int_X\left(L_{\mu}K\right)(z,w) \phi_k (w) \overline{ (L_{ \mu}\phi_l)(z)} \rho(z)^{1-n}\rho(w)^{1-n}d^2wd^2z.
\end{align*}Therefore, $T_1+T_2$ can be simplied as
\begin{align}\label{eq1_17_3}
T_1+T_2=& \sum_{l=1}^d\sum_{k=1}^d(N^{-1})_{lk}\int_X\int_X(L_{\mu}K)(z,w) \phi_k(w) \overline{ (L_{ \mu}\phi_l)(z)} \rho(z)^{1-n}\rho(w)^{1-n}d^2wd^2z.
\end{align}
Similar to \eqref{eq1_29_8}, we find that for any $1\leq l\leq d$,
\begin{align*}
\overline{\left(L_{\mu}\phi_l\right)(z)}=\int_XK(\zeta,z)\overline{(L_{\mu}\phi_l)(\zeta)}\rho(\zeta)^{1-n}d^2\zeta +\int_X(L_{\bar{\mu}}K)(\zeta, z))\overline{\phi_l(\zeta)}\rho(\zeta)^{1-n}d^2\zeta,
\end{align*} and therefore  
\begin{align*}
& \sum_{l=1}^d\sum_{k=1}^d(N^{-1})_{lk}\int_X(L_{\mu}K)(z,w) \phi_k(w) \overline{ (L_{ \mu}\phi_l)(z)} \rho(z)^{1-n} d^2z
 \\=& \sum_{l=1}^d\sum_{k=1}^d(N^{-1})_{lk}\int_X \int_X(L_{\mu}K)(z,w) \phi_k(w) K(\zeta,z)\overline{(L_{\mu}\phi_l)(\zeta)}\rho(\zeta)^{1-n} \rho(z)^{1-n} d^2zd^2\zeta\\
&+\sum_{l=1}^d\sum_{k=1}^d(N^{-1})_{lk}\int_X \int_X(L_{\mu}K)(z,w) \phi_k(w) (L_{\bar{\mu}}K)(\zeta, z))\overline{\phi_l(\zeta)}\rho(\zeta)^{1-n}  \rho(z)^{1-n}d^2z d^2\zeta\\
=&J_1+J_2.
\end{align*}
By definition of projection kernel, we have
\begin{align}\label{eq1_28_1}
K(z,w)=\int_{X}K(z,\zeta)K(\zeta, w)\rho(\zeta)^{1-n}d^2\zeta.
\end{align}Taking variation implies that
\begin{align*}
(L_{\mu}K)(z,w)=& \int_{X}(L_{\mu}K)(z,\zeta)K(\zeta, w)\rho(\zeta)^{1-n}d^2\zeta+\int_{X}K(z,\zeta)(L_{\mu}K)(\zeta, w)\rho(\zeta)^{1-n}d^2\zeta.\end{align*}By Proposition \ref{prop1_26_2}, $(L_{\mu}K)(z,w)$ is anti-holomorphic in $w$. Hence,  
\begin{align}\label{eq1_29_1} \int_{X}(L_{\mu}K)(z,\zeta)K(\zeta, w)\rho(\zeta)^{1-n}d^2\zeta=(L_{\mu}K)(z,w),\end{align} and thus
\begin{align}\label{eq1_29_2}\int_{X}K(z,\zeta)(L_{\mu}K)(\zeta, w)\rho(\zeta)^{1-n}d^2\zeta=0.\end{align}
This shows that $J_1=0$. For $J_2$, Lemma \ref{lemma1_27_1} gives
$$J_2= \int_X \int_X(L_{\mu}K)(z,w) K(w,\zeta) (L_{\bar{\mu}}K)(\zeta, z)) \rho(\zeta)^{1-n}  \rho(z)^{1-n}d^2z d^2\zeta.$$Using \eqref{eq1_29_1}, we then find that
$$T_1+T_2=\int_X \int_X(L_{\mu}K)(z,w)   (L_{\bar{\mu}}K)(w, z) \rho(w)^{1-n}  \rho(z)^{1-n}d^2z d^2w.$$This completes the proof of \eqref{eq1_29_5}.
\end{proof}

 In Appendix \ref{interesting_formula}, we prove an alternative formula \eqref{eq1_29_10} for $\pa_{\mu}\pa_{\bar{\nu}}\log\det N_n$ which might be interesting of its own right.

If $g(z)$ is a function such that 
$$g(\gamma z)\gamma'(z)\overline{\gamma'(z)}=g(z)\hspace{1cm}\text{for all}\;\gamma\in \Gamma, \;z\in \mathbb{U},$$then
$$\int_X\sum_{\gamma\in \Gamma}g(\gamma z)\gamma'(z)\overline{\gamma'(z)}d^2z=\int_{\mathbb{U}}g(z)d^2z.$$ Using this, one can obtain from Theorem \ref{thm1_29_1} that 

\begin{proposition}\label{prop2_22_1}Let $n\geq 2$. Given $\mu, \nu\in\Omega_{-1,1}(X)$, 
$$\pa_{\mu}\pa_{\bar{\nu}}\log\det N_n=\int_X \sum_{\gamma\in \Gamma}\mathscr{A}(\gamma z, z, \mu, \nu)\gamma'(z)^n\rho(z)^{1-n}d^2z,$$ where  
\begin{equation}\label{eq2_2_1}\begin{split}
 \mathscr{A}(z',z, \mu, \nu)
=&\int_{\mathbb{U}}(L_{\mu}\mathcal{K}_n)(z',w)(L_{\bar{\nu}}\mathcal{K}_n)(w,z)\rho(w)^{1-n} d^2w \\
&+(1-n)  \mathcal{K}_n(z', z) \int_{ \mathbb{U}}\mathcal{G}(z,w) \mu(w)\overline{\nu(w)}\rho(w)d^2w  - \mathcal{K}_n( z', z) \mu(z)\overline{\nu(z)}.
\end{split}\end{equation}
 The function $\mathscr{A}(z', z, \mu, \nu)$ has the property that for any $\sigma\in \text{PSL}(2,\mathbb{R})$,
$$\mathscr{A}\left(\sigma z', \sigma z, \mu\circ\sigma'\frac{\overline{\sigma'}}{\sigma'}, \nu\circ\sigma'\frac{\overline{\sigma'}}{\sigma'}\right)\sigma'(z')^n\overline{\sigma'(z)}^n=\mathscr{A}(z',z, \mu, \nu).$$ It is an automorphic form of type $(n,0)$ in $z'$, and an automorphic form of type $(0,n)$ in $z$.
\end{proposition}

\begin{remark}\label{remark2_2_1} 
One can write $\pa_{\mu}\pa_{\bar{\nu}}\log\det N_n$ as a sum of four terms:
$$\pa_{\mu}\pa_{\bar{\nu}}\log\det N_n=\mathscr{E}_0+\mathscr{E}_{H}+\mathscr{E}_{P}+\mathscr{E}_{E},$$
where 
\begin{align*}
\mathscr{E}_{0}=&\int_X  \mathscr{A}(z, z, \mu, \nu)\gamma'(z)^n\rho(z)^{1-n}d^2z,\\
\mathscr{E}_{H}=&\int_X \sum_{\substack{\gamma\in \Gamma\\\gamma \;\text{is hyperbolic}}}\mathscr{A}(\gamma z, z, \mu, \nu)\gamma'(z)^n\rho(z)^{1-n}d^2z,\\
\mathscr{E}_{P}=&\int_X \sum_{\substack{\gamma\in \Gamma\\\gamma \;\text{is parabolic}}}\mathscr{A}(\gamma z, z, \mu, \nu)\gamma'(z)^n\rho(z)^{1-n}d^2z,\\
\mathscr{E}_{E}=&\int_X \sum_{\substack{\gamma\in \Gamma\\\gamma \;\text{is elliptic}}}\mathscr{A}(\gamma z, z, \mu, \nu)\gamma'(z)^n\rho(z)^{1-n}d^2z,
\end{align*}are respectively the identity, hyperbolic, parabolic and elliptic contributions.

\end{remark}

Now let us simplify the expression in $\pa_{\mu}\pa_{\bar{\nu}}\log\det N_n$ that involves the term
$$\int_X\int_{\mathbb{U}}\left(L_{\mu}\mathcal{K}_n\right)(\gamma z, w)\gamma'(z)^n \left(L_{\bar{\nu}}\mathcal{K}_n\right)(w, z)\rho(w)^{1-n}\rho(z)^{1-n}d^2w d^2z.$$   First we have the following lemma.

\begin{proposition}\label{prop2_20_1} For $n\geq 1$, let
$  \mathcal{K}_n(z,w)$ be the kernel defined by \eqref{eq2_17_2}. If $\mu\in \Omega_{-1, 1}(X)$, then 
\begin{equation}\label{eq2_21_1}\begin{split}
\frac{\pa}{\pa\bar{z}} (L_{\mu}\mathcal{K}_n)(z,w)=&-2n\mu(z)\frac{(\bar{w}-\bar{z})}{(z-\bar{w})(z-\bar{z})}\mathcal{K}_n(z,w).
\end{split}
\end{equation}
\end{proposition}
\begin{proof}
Let
$$\mathcal{H}(z,w)=\frac{1}{(z-\bar{w})^{2n}},$$which is a constant times $\mathcal{K}_n(z,w)$. It is sufficient to prove \eqref{eq2_21_1} with $\mathcal{H}(z,w)$ in place of $\mathcal{K}_n(z,w)$. By Proposition \ref{prop1_26_2}, 
\begin{align}\label{eq2_22_4}
(L_{\mu}\mathcal{H})(z,w)=&-\frac{n}{\pi}\frac{1}{(z-\bar{w})^{2n-2}}\int_{\mathbb{U}}\frac{\mu(\zeta)}{(\zeta-z)^2(\zeta-\bar{w})^2}d^2\zeta\end{align}
One can apply $\pa_{\bar{z}}$ directly to \eqref{eq2_22_4} to prove \eqref{eq2_21_1}. However, this involves working with singular kernels. To bypass this complication, we apply Proposition \ref{prop2_2_1} to rewrite $(L_{\mu}\mathcal{H})(z,w)$ as
\begin{align*}
(L_{\mu}\mathcal{H})(z,w)=&-\frac{n}{\pi}\frac{1}{(z-\bar{w})^{2n}}\int_{\mathbb{U}} \mu(\zeta)\left\{\frac{(\bar{z}-\bar{w})^2}{(\zeta-\bar{z})^2(\zeta-\bar{w})^2}+\frac{2(\bar{z}-\bar{w})^2(z-\bar{z})}{(\zeta-\bar{z})^3(\zeta-\bar{w})(z-\bar{w})} \right\}d^2\zeta.
\end{align*}Notice that now the kernel does not have singularities. Applying $\pa_{\bar{z}}$ give
\begin{align*}
\frac{\pa}{\pa\bar{z}} (L_{\mu}\mathcal{K})(z,w)=n\mathcal{K}(z,w)\times \mathfrak{B}(z,w),
\end{align*}where
\begin{align*}
\mathfrak{B}(z,w)=&-\frac{1}{\pi}\int_{\mathbb{U}}\mu(\zeta)\frac{\pa}{\pa\bar{z}}\left\{\frac{(\bar{z}-\bar{w})^2}{(\zeta-\bar{z})^2(\zeta-\bar{w})^2}+\frac{2(\bar{z}-\bar{w})^2(z-\bar{z})}{(\zeta-\bar{z})^3(\zeta-\bar{w})(z-\bar{w})} \right\}d^2\zeta\\
=&\frac{6}{\pi}\frac{(z-\bar{z})(\bar{w}-\bar{z})}{(z-\bar{w})}\int_{\mathbb{U}}\frac{\mu(\zeta)}{(\zeta-\bar{z})^4}d^2\zeta\\
=&-\frac{2(\bar{w}-\bar{z})}{(z-\bar{w})(z-\bar{z})}\mu(z).
\end{align*}This proves  \eqref{eq2_21_1}.  
\end{proof}
Now we have the following important result. 
\begin{proposition}\label{prop2_22_2}
Let $n\geq 2$. Given   $\mu, \nu \in \Omega_{-1,1}(X)$, if $\gamma$ is an element of $\Gamma$, then 
\begin{equation}\label{eq2_22_3}\begin{split}
& \int_X\int_{\mathbb{U}}\left(L_{\mu}\mathcal{K}_n\right)(\gamma z, w)\gamma'(z)^n \left(L_{\bar{\nu}}\mathcal{K}_n\right)(w, z)\rho(w)^{1-n}\rho(z)^{1-n}d^2w d^2z\\
=&- \frac{n}{\pi^2}\int_X\int_{\mathbb{U}} \frac{\mu(z)\overline{\nu( \eta)}}{(z-\bar{\eta})^2(\gamma z-\bar{\eta})^2}  \frac{(  z-\bar{ z})^{2n-2} }{(\gamma z-\bar{z} )^{2n-2}}    \gamma'(z) ^n  d^2\eta d^2 z\\
& - \frac{2(n-1)}{\pi^2}\int_X\int_{\mathbb{U}} \frac{\mu(z)\overline{\nu( \eta)}}{(z-\bar{\eta})^3(\gamma z-\bar{\eta}) }  \frac{(  z-\bar{ z})^{2n-1} }{(\gamma z-\bar{z} )^{2n-1}}    \gamma'(z) ^n  d^2\eta d^2 z\\
&+\int_X \mathcal{K}_n(\gamma z, z)\gamma'(z)^n\mu(z)\overline{\nu(z)} \rho(z)^{1-n}d^2z.
\end{split}\end{equation}
\end{proposition}
 \begin{proof} To simplify notation, let $\mathcal{K}(z,w)=\mathcal{K}_n(z,w)$.
Using the complex conjugate of Proposition \ref{prop2_2_1}, we find that
\begin{align*}
(L_{\bar{\nu}}\mathcal{K})(w,z)=&(-1)^{n-1}\frac{2^{2n-2}(2n-1)n}{\pi^2}\frac{1}{(w-\bar{z})^{2n-2}}\int_{\mathbb{U}}\frac{ \overline{\nu(\eta)}}{(w- \bar{\eta})^2( \bar{\eta}-\bar{z})^2}d^2\eta\\
=&(-1)^{n-1}\frac{2^{2n-2}(2n-1)n}{\pi^2}\frac{(w-z)^2}{(w-\bar{z})^{2n}}\int_{\mathbb{U}}\frac{\overline{\nu(\eta)}}{(w-\bar{\eta})^2(z-\bar{\eta})^2}d^2\eta\\
&+2(-1)^{n-1}\frac{2^{2n-2}(2n-1)n}{\pi^2}\frac{(w-z)^2(z-\bar{z})}{(w-\bar{z})^{2n+1}}\int_{\mathbb{U}}\frac{\overline{\nu(\eta)}}{(w-\bar{\eta})(z-\bar{\eta})^3}d^2\eta.
\end{align*}
This implies that
$$\rho(z)^{1-n}(L_{\bar{\nu}}\mathcal{K})(w,z)=\frac{\pa}{\pa \bar{z}}\left[\rho(z)^{1-n}\mathfrak{K}(w,z)\right],$$ where
\begin{equation}\label{eq3_1_4}\begin{split}
\mathfrak{K}(w,z)=&(-1)^{n}\frac{2^{2n-2}n}{\pi^2}\frac{(z-\bar{z})(w-z) }{(w-\bar{z} )^{2n-1}}\int_{\mathbb{U}}\frac{\overline{\nu( \eta)}}{(w-\bar{\eta})^2(z-\bar{\eta})^2}d^2\eta\\
&+(-1)^{n}\frac{2^{2n-2}(2n-1)}{\pi^2}\frac{(z-\bar{z})^2(w-z) }{(w-\bar{z})^{2n}} \int_{\mathbb{U}}\frac{\overline{\nu( \eta)}}{(w-\bar{\eta}) (z-\bar{\eta})^3}d^2\eta.
\end{split}\end{equation}Using change of variables and   integration by parts, one has
\begin{align*}
& \int_X\int_{\mathbb{U}}\left(L_{\mu}\mathcal{K}\right)(\gamma z, w)\gamma'(z)^n \left(L_{\bar{\nu}}\mathcal{K}\right)(w, z)\rho(w)^{1-n}\rho(z)^{1-n}d^2w d^2z\\
=& \int_X\int_{\mathbb{U}}\left(L_{\mu}\mathcal{K}\right)(\gamma z, \gamma w)\gamma'(z)^n \overline{\gamma'(w)}^n\left(L_{\bar{\nu}}\mathcal{K}\right)(\gamma w, z)\gamma'(w)^n\rho(w)^{1-n}\rho(z)^{1-n}d^2w d^2z\\
=&\int_{\mathbb{U}} \int_X\left(L_{\mu}\mathcal{K}\right)(z, w)  \frac{\pa}{\pa \bar{z}}\left[\rho(z)^{1-n}\mathfrak{K}(\gamma w,z)\gamma'(w)^n\right]\rho(w)^{1-n}  d^2z d^2w\\
=&- \int_{\mathbb{U}}\int_X\frac{\pa}{\pa\bar{z}}\left(L_{\mu}\mathcal{K}\right)(z,\  w)  \mathfrak{K}(\gamma w,z) \gamma'(w)^n\rho(z)^{1-n}\rho(w)^{1-n}  d^2z d^2w.
\end{align*}Proposition \ref{prop2_20_1} then gives
\begin{align*}
& \int_X\int_{\mathbb{U}}\left(L_{\mu}\mathcal{K}\right)(\gamma z, w)\gamma'(z)^n \left(L_{\bar{\nu}}\mathcal{K}\right)(w, z)\rho(w)^{1-n}\rho(z)^{1-n}d^2w d^2z\\
=&2n \int_X\int_{\mathbb{U}}\mu(z)\frac{(\bar{w}-\bar{z})}{(  z-\bar{w})(  z-\bar{z})} \mathcal{K}(  z,w) \mathfrak{K}(\gamma w,z)\gamma'(w)^n \rho(z)^{1-n}\rho(w)^{1-n}d^2w d^2z\\
=& \mathscr{M}_1 +\mathscr{M}_2,
\end{align*}where $\mathscr{M}_1$ and $\mathscr{M}_2$ come from the two terms of $\mathfrak{K}(w,z)$ \eqref{eq3_1_4}.
\begin{align*}
\mathscr{M}_1 =&- \frac{2n^2}{\pi^2}\int_X\int_{\mathbb{U}}\int_{\mathbb{U}}\mu(z)\overline{\nu( \eta)}\frac{(\bar{w}-\bar{z})}{(  z-\bar{w})(  z-\bar{z})} \mathcal{K}(  z,w) \frac{(z-\bar{z})^{2n-1} }{(\gamma w-\bar{z} )^{2n-1}}\\&\hspace{6cm}\times\frac{(\gamma w-z)}{(\gamma w-\bar{\eta})^2(z-\bar{\eta})^2}\gamma'(w)^n\rho(w)^{1-n}d^2\eta d^2w d^2z,\\
\mathscr{M}_2 =&- \frac{2n(2n-1)}{\pi^2}\int_X\int_{\mathbb{U}}\int_{\mathbb{U}}\mu(z)\overline{\nu( \eta)}\frac{(\bar{w}-\bar{z})}{(  z-\bar{w})(  z-\bar{z})} \mathcal{K}(  z,w) \frac{(z-\bar{z})^{2n} }{(\gamma w-\bar{z} )^{2n}}\\&\hspace{6cm}\times\frac{(\gamma w-z)}{(\gamma w-\bar{\eta})(z-\bar{\eta})^3}\gamma'(w)^n\rho(w)^{1-n}d^2\eta d^2w d^2z.
\end{align*}The integrals with respect to $w$ can be computed explicitly. First we notice that if 
$\phi(w)$ is a holomorphic $n$-differential, then
\begin{align}\label{eq2_21_2}
\int_{\mathbb{U}}\frac{(\bar{w}-\bar{z})}{(  z-\bar{w})(  z-\bar{z})} \mathcal{K}(  z,w)  \rho(w)^{1-n} d^2w=&-\frac{1}{2n}\rho(z)^{n}\frac{\pa}{\pa z}\rho(z)^{-n}\phi(z).
\end{align}This can be proved by writting
$$\frac{(\bar{w}-\bar{z})}{(  z-\bar{w})(  z-\bar{z})}=\frac{1}{z-\bar{w}}-\frac{1}{z-\bar{z}},$$ and using the fact that
\begin{align*}
\int_{\mathbb{U}}\frac{1}{z-\bar{w}}\mathcal{K}(z,w)\phi(w)\rho(w)^{1-n}d^2w=&-\frac{1}{2n}\frac{\pa}{\pa z}\int_{\mathbb{U}} \mathcal{K}(z,w)\phi(w)\rho(w)^{1-n}d^2w\\
=&-\frac{1}{2n}\frac{\pa\phi(z)}{\pa z}.
\end{align*}From this, we find that
\begin{align*}
\mathscr{M}_1 =& \frac{n}{\pi^2}\int_X\int_{\mathbb{U}} \frac{\mu(z)\overline{\nu( \eta)}}{(z-\bar{\eta})^2} \frac{1}{z-\bar{z}}\left.\frac{\pa}{\pa w}\right|_{w=z}\left\{\frac{(\gamma w-\overline{\gamma w})^{2n} }{(\gamma w-\bar{z} )^{2n-1}} \frac{(\gamma w-z)}{(\gamma w-\bar{\eta})^2}\frac{1}{\overline{\gamma'(w)}^n}\right\}d^2\eta   d^2z,\\
\mathscr{M}_2 =& \frac{2n-1}{\pi^2}\int_X\int_{\mathbb{U}} \frac{\mu(z)\overline{\nu( \eta)}}{(z-\bar{\eta})^3}  \left.\frac{\pa}{\pa w}\right|_{w=z}\left\{\frac{(\gamma w-\overline{\gamma w})^{2n} }{(\gamma w-\bar{z} )^{2n}} \frac{(\gamma w-z)}{\gamma w-\bar{\eta}}\frac{1}{\overline{\gamma'(w)^n}}\right\}d^2\eta  d^2z.\end{align*}
Now notice that
\begin{align*}
&\frac{1}{z-\bar{z}}\left.\frac{\pa}{\pa w}\right|_{w=z}\left\{\frac{(\gamma w-\overline{\gamma w})^{2n} }{(\gamma w-\bar{z} )^{2n-1}} \frac{(\gamma w-z)}{(\gamma w-\bar{\eta})^2}\frac{1}{\overline{\gamma'(w)}^n}\right\}\\=&
 \frac{\pa}{\pa z} \left\{\frac{1}{z-\bar{z}}\frac{(\gamma z-\overline{\gamma z})^{2n} }{(\gamma z-\bar{z} )^{2n-1}} \frac{(\gamma z-z)}{(\gamma z-\bar{\eta})^2}\frac{1}{\overline{\gamma'(z)}^n}\right\}-\frac{(\gamma z-\overline{\gamma z})^{2n} }{(\gamma z-\bar{z} )^{2n-1}} \frac{ 1}{(\gamma z-\bar{\eta})^2}\frac{1}{\overline{\gamma'(z)}^n}
\left.\frac{\pa}{\pa z}\right|_{w=z}\frac{\gamma w-z}{z-\bar{z}}\\
=& \frac{\pa}{\pa z} \left\{\frac{1}{z-\bar{z}}\frac{(\gamma z-\overline{\gamma z})^{2n} }{(\gamma z-\bar{z} )^{2n-1}} \frac{(\gamma z-z)}{(\gamma z-\bar{\eta})^2}\frac{1}{\overline{\gamma'(z)}^n}\right\}+\frac{(\gamma z-\overline{\gamma z})^{2n} }{(\gamma z-\bar{z} )^{2n-2}} \frac{ 1}{(\gamma z-\bar{\eta})^2(z-\bar{z})^2}\frac{1}{\overline{\gamma'(z)}^n}.
\end{align*}
Using
$$\frac{\pa \mu}{\pa z}=\frac{2}{z-\bar{z}}\mu(z),$$ integration by parts give
\begin{align*}
\mathscr{M}_1 =& - \frac{2n}{\pi^2}\int_X\int_{\mathbb{U}} \frac{\mu(z)\overline{\nu( \eta)}}{(z-\bar{\eta})^3(\gamma z-\bar{\eta})^2}  \frac{(  z-\bar{ z})^{2n-2} }{(\gamma z-\bar{z} )^{2n-1}}  (\gamma z-z) (\bar{z}-\bar{\eta}) \gamma'(z) ^n  d^2\eta d^2 z\\
&+  \frac{n}{\pi^2}\int_X\int_{\mathbb{U}} \frac{\mu(z)\overline{\nu( \eta)}}{(z-\bar{\eta})^2(\gamma z-\bar{\eta})^2}  \frac{(  z-\bar{ z})^{2n-2} }{(\gamma z-\bar{z} )^{2n-2}}   \gamma'(z) ^n  d^2\eta d^2 z\\
=&\mathscr{M}_3+\mathscr{M}_4.
\end{align*}Similarly, using
\begin{align*}
&\frac{1}{z-\bar{\eta}}\left.\frac{\pa}{\pa w}\right|_{w=z}\left\{\frac{(\gamma w-\overline{\gamma w})^{2n} }{(\gamma w-\bar{z} )^{2n}} \frac{(\gamma w-z)}{\gamma w-\bar{\eta}}\frac{1}{\overline{\gamma'(w)^n}}\right\}\\
=& \frac{\pa}{\pa z}\left\{\frac{1}{z-\bar{\eta} }\frac{(\gamma z-\overline{\gamma z})^{2n} }{(\gamma z-\bar{z} )^{2n}} \frac{(\gamma z-z)}{\gamma z-\bar{\eta}}\frac{1}{\overline{\gamma'(z)}^n}\right\}-\frac{(\gamma z-\overline{\gamma z})^{2n} }{(\gamma z-\bar{z} )^{2n}} \frac{1}{\gamma z-\bar{\eta}}\frac{1}{\overline{\gamma'(z)}^n}\left.\frac{\pa}{\pa z}\right|_{w=z}\frac{\gamma w-z}{z-\bar{\eta}}\\
=& \frac{\pa}{\pa z}\left\{\frac{1}{z-\bar{\eta} }\frac{(\gamma z-\overline{\gamma z})^{2n} }{(\gamma z-\bar{z} )^{2n}} \frac{(\gamma z-z)}{\gamma z-\bar{\eta}}\frac{1}{\overline{\gamma'(z)}^n}\right\}+\frac{(\gamma z-\overline{\gamma z})^{2n} }{(\gamma z-\bar{z} )^{2n}} \frac{1}{( z-\bar{\eta})^2}\frac{1}{\overline{\gamma'(z)}^n},
\end{align*}we find that
\begin{align*}
\mathscr{M}_2 =& -\frac{2(2n-1)}{\pi^2}\int_X\int_{\mathbb{U}} \frac{\mu(z)\overline{\nu( \eta)}}{(z-\bar{\eta})^4(\gamma z-\bar{\eta}) } 
  \frac{(  z-\bar{  z})^{2n-1} }{(\gamma z-\bar{z} )^{2n}} (\gamma z-z)(\bar{z}-\bar{\eta})  \gamma'(z)^nd^2\zeta d^2z\\
&+\frac{ 2n-1}{\pi^2}\int_X\int_{\mathbb{U}} \frac{\mu(z)\overline{\nu( \eta)}}{(z-\bar{\eta})^4  } \frac{(  z-\bar{  z})^{2n} }{(\gamma z-\bar{z} )^{2n}}\gamma'(z)^nd^2\zeta d^2z\\
=&\mathscr{M}_5+\mathscr{M}_6.
\end{align*}Now using the identity
\begin{align*}
\frac{(\gamma z-z)(\bar{z}-\bar{\eta})}{(z-\bar{\eta})(\gamma z-\bar{z})}=1-\frac{(\gamma z-\bar{\eta})(z-\bar{z} )}{(z-\bar{\eta})(\gamma z-\bar{z})},
\end{align*}we find that
\begin{align}
\mathscr{M}_3=& - \frac{2n}{\pi^2}\int_X\int_{\mathbb{U}} \frac{\mu(z)\overline{\nu( \eta)}}{(z-\bar{\eta})^3(\gamma z-\bar{\eta})^2}  \frac{(  z-\bar{ z})^{2n-2} }{(\gamma z-\bar{z} )^{2n-1}}  (\gamma z-z) (\bar{z}-\bar{\eta}) \gamma'(z) ^n  d^2\eta d^2 z\nonumber\\=& - \frac{2n}{\pi^2}\int_X\int_{\mathbb{U}} \frac{\mu(z)\overline{\nu( \eta)}}{(z-\bar{\eta})^2(\gamma z-\bar{\eta})^2}  \frac{(  z-\bar{ z})^{2n-2} }{(\gamma z-\bar{z} )^{2n-2}}    \gamma'(z) ^n  d^2\eta d^2 z\label{eq2_21_5}\\
& + \frac{2n}{\pi^2}\int_X\int_{\mathbb{U}} \frac{\mu(z)\overline{\nu( \eta)}}{(z-\bar{\eta})^3(\gamma z-\bar{\eta}) }  \frac{(  z-\bar{ z})^{2n-1} }{(\gamma z-\bar{z} )^{2n-1}}    \gamma'(z) ^n  d^2\eta d^2 z,\label{eq2_21_6}
\end{align}
\begin{align}
\mathscr{M}_5 =& -\frac{2(2n-1)}{\pi^2}\int_X\int_{\mathbb{U}} \frac{\mu(z)\overline{\nu( \eta)}}{(z-\bar{\eta})^4(\gamma z-\bar{\eta}) } 
  \frac{(  z-\bar{  z})^{2n-1} }{(\gamma z-\bar{z} )^{2n}} (\gamma z-z)(\bar{z}-\bar{\eta})  \gamma'(z)^nd^2\zeta d^2z \nonumber\\
=&-\frac{2(2n-1)}{\pi^2}\int_X\int_{\mathbb{U}} \frac{\mu(z)\overline{\nu( \eta)}}{(z-\bar{\eta})^3(\gamma z-\bar{\eta}) } 
  \frac{(  z-\bar{  z})^{2n-1} }{(\gamma z-\bar{z} )^{2n-1}}   \gamma'(z)^nd^2\zeta d^2z\label{eq2_21_7}\\
&+\frac{2(2n-1)}{\pi^2}\int_X\int_{\mathbb{U}} \frac{\mu(z)\overline{\nu( \eta)}}{(z-\bar{\eta})^4 } 
  \frac{(  z-\bar{  z})^{2n} }{(\gamma z-\bar{z} )^{2n}}   \gamma'(z)^nd^2\zeta d^2z.\label{eq2_21_8}
\end{align}
Notice that \eqref{eq2_21_5} is $-2$ of $\mathscr{M}_4$, \eqref{eq2_21_8} is 2 times $\mathscr{M}_6$, while \eqref{eq2_21_6} and \eqref{eq2_21_7} are multiples of each other. Combining the terms give
\begin{align*}
\mathscr{M}_1+\mathscr{M}_2=&\frac{3(2n-1)}{\pi^2}\int_X\int_{\mathbb{U}} \frac{\mu(z)\overline{\nu( \eta)}}{(z-\bar{\eta})^4 } 
  \frac{(  z-\bar{  z})^{2n} }{(\gamma z-\bar{z} )^{2n}}   \gamma'(z)^nd^2\zeta d^2z\\
& - \frac{n}{\pi^2}\int_X\int_{\mathbb{U}} \frac{\mu(z)\overline{\nu( \eta)}}{(z-\bar{\eta})^2(\gamma z-\bar{\eta})^2}  \frac{(  z-\bar{ z})^{2n-2} }{(\gamma z-\bar{z} )^{2n-2}}    \gamma'(z) ^n  d^2\eta d^2 z\\
& - \frac{2(n-1)}{\pi^2}\int_X\int_{\mathbb{U}} \frac{\mu(z)\overline{\nu( \eta)}}{(z-\bar{\eta})^3(\gamma z-\bar{\eta}) }  \frac{(  z-\bar{ z})^{2n-1} }{(\gamma z-\bar{z} )^{2n-1}}    \gamma'(z) ^n  d^2\eta d^2 z.
\end{align*}Finally, the complex conjugate of \eqref{eq2_6_6} gives 
\begin{align*}
&\frac{3(2n-1)}{\pi^2}\int_X\int_{\mathbb{U}} \frac{\mu(z)\overline{\nu( \eta)}}{(z-\bar{\eta})^4 } 
  \frac{(  z-\bar{  z})^{2n} }{(\gamma z-\bar{z} )^{2n}}   \gamma'(z)^nd^2\zeta d^2z\\
=&-\frac{(2n-1)}{\pi}\int_{X}\mu(z)\overline{\nu(z)}  \frac{(  z-\bar{  z})^{2n-2} }{(\gamma z-\bar{z} )^{2n}}   d^2z \\
=&(-1)^n \frac{2^{2n-2}(2n-1)}{\pi}\int_X \frac{\gamma'(z)^n}{(\gamma z-\bar{z})^{2n}}\mu(z)\overline{\nu(z)} \rho(z)^{1-n}d^2z\\
=&\int_X \mathcal{K}(\gamma z, z)\gamma'(z)^n\mu(z)\overline{\nu(z)} \rho(z)^{1-n}d^2z.
\end{align*}The assertion follows.
\end{proof}

Using this proposition, we can express $\pa_{\mu}\pa_{\bar{\nu}}\log\det N_n$ given in Proposition \ref{prop2_22_1} more explicitly as

\begin{theorem}\label{thm2_22_3} Let $n\geq 2$. Given $\mu,\nu\in \Omega_{-1,1}(X)$,  
\begin{align}\label{eq2_24_3} \pa_{\mu}\pa_{\bar{\nu}}\log\det N_n=\mathscr{X}+\mathscr{Y}+\mathscr{Z},\end{align} where
\begin{equation}\label{eq2_22_6}\begin{split}
\mathscr{X} =&-\frac{n}{\pi^2}\sum_{\gamma\in\Gamma}\int_X\int_{\mathbb{U}}  \frac{\mu(z)\overline{\nu( w)}}{(z-\bar{w})^2(\gamma z-\bar{w})^2}  \frac{(  z-\bar{ z})^{2n-2} }{(\gamma z-\bar{z} )^{2n-2}}    \gamma'(z) ^n  d^2w  d^2z,\\
\mathscr{Y}=&-\frac{2(n-1)}{\pi^2}\sum_{\gamma\in\Gamma}\int_X\int_{\mathbb{U}}   \frac{\mu(z)\overline{\nu( w)}}{(z-\bar{w})^3(\gamma z-\bar{w}) }  \frac{(  z-\bar{ z})^{2n-1} }{(\gamma z-\bar{z} )^{2n-1}}    \gamma'(z) ^n  d^2w d^2z,\\
\mathscr{Z}=&\frac{(n-1)(2n-1)}{\pi}\sum_{\gamma\in\Gamma}\int_X\int_{\mathbb{U}} \frac{(z-\bar{z})^{2n-2}}{(\gamma z-\bar{z})^{2n}}\gamma'(z)^n  \mathcal{G}(z,w) \mu(w)\overline{\nu(w)}\rho(w)d^2wd^2z.
\end{split}\end{equation} 
 
\end{theorem}

At first sight, one might think that the formula \eqref{eq2_22_3} is more complicated than applying \eqref{eq2_22_1} for both $(L_{\mu}\mathcal{K}_n)(z,w)$ and $(L_{\bar{\mu}}\mathcal{K}_n)(w,z)$. However, the latter will make  $$\int_X\int_{\mathbb{U}}\left(L_{\mu}\mathcal{K}_n\right)(\gamma z, w)\gamma'(z)^n \left(L_{\bar{\nu}}\mathcal{K}_n\right)(w, z)\rho(w)^{1-n}\rho(z)^{1-n}d^2w d^2z$$ into a 4-fold integral. In some sense, \eqref{eq2_22_3} is a simplication of the 4-fold integral by integrating out two of the variables. One notice that there is a lack of symmetry in the right-hand side of the formula \eqref{eq2_22_3} as exhibited in the left-hand side. However, we will see that this formula is precisely what we needed when  the second variation of $\log\det N_n$ is compared to the second variation of $\log\Delta_n$. 

Corresponding to the three terms \eqref{eq2_22_6} in $\pa_{\mu}\pa_{\bar{\nu}}\log\det N_n$, for $C=0, H, P, E$ equal respectively to identity, hyperbolic, parabolic and elliptic, the $C$ contribution $\mathscr{E}_C$ to  $\pa_{\mu}\pa_{\bar{\nu}}\log\det N_n$ can be written as 
$$\mathscr{E}_C=\mathscr{X}_C+\mathscr{Y}_C+\mathscr{Z}_C,$$where $\mathscr{X}_C$, $\mathscr{Y}_C$, $\mathscr{Z}_C$ are obtained respectively from $\mathscr{X}$, $\mathscr{Y}$ and $\mathscr{Z}$ by replacing the summation over $\gamma$ in $\Gamma$ to summation over those $\gamma$ that are in the subset of type $C$ elements.

Using \eqref{eq2_22_6}, it is easy to compute the identity contribution $\mathscr{E}_0$ to $\pa_{\mu}\pa_{\bar{\nu}}\log\det N_n$. 
\begin{theorem}\label{thm2_22_2} Let $n\geq 2$. 
Given $\mu,\nu\in\Omega_{-1,1}(X)$, the identity  contribution to $\pa_{\mu}\pa_{\bar{\nu}}\log\det N_n$ is given by
$$\mathscr{E}_0=-\frac{6n^2-6n+1}{12\pi}\langle \mu, \nu\rangle_{\text{WP}}.$$
\end{theorem}
 \begin{proof}
From \eqref{eq2_22_6}, we immediately find that for the identity contribution  
\begin{align*}
\mathscr{X}_{0}+\mathscr{Y}_0=& -\frac{ 3n-2}{\pi^2}\int_{X}\int_{\mathbb{U}} \frac{\mu(z)\overline{\nu( w)}}{(z-\bar{w})^4}d^2w d^2z,\\
\mathscr{Z}_{0}=&- \frac{(n-1)(2n-1) }{4\pi}\int_X \rho(z)  \int_{X}G(z,w) \mu(w)\overline{\nu(w)}\rho(w)d^2w d^2z.\end{align*}
 Using \eqref{eq2_6_6}, we find that
$$\mathscr{X}_{0}+\mathscr{Y}_0=  -\frac{ 3n-2}{12\pi }\int_{X}\mu(z)\overline{\nu(z)}\rho(z)d^2z=-\frac{ 3n-2}{12\pi }\langle \mu, \nu\rangle_{\text{WP}}.$$

Using the fact that for the function $\mathscr{I}:\mathbb{U}\rightarrow\mathbb{C}$ with constant value 1,
$$(\Delta_0+2)\mathscr{I}(z)=2\mathscr{I}(z),$$ we find that
$$2(\Delta_0+2)^{-1}\mathscr{I}(z)=\mathscr{I}(z).$$ This implies that
$$\int_{X}G(z, w)\rho(z)d^2z =1.$$ Hence,
\begin{align*}
\mathscr{Z}_{0}
=&-\frac{(2n-1)(n-1)}{4\pi}\int_{X} \mu(w)\overline{\nu(w)}\rho(w)d^2w\\
=&-\frac{ (2n-1)(n-1) }{4\pi}\langle \mu, \nu\rangle_{\text{WP}}
\end{align*}
Therefore,
$$\mathscr{E}_0=-\frac{3(2n^2-3n+1)+3n-2}{\pi}\langle \mu, \nu\rangle_{\text{WP}}=-\frac{6n^2-6n+1}{12\pi}\langle \mu, \nu\rangle_{\text{WP}}.$$

\end{proof}
 In Section \ref{vary_det}, we are going to compute $\pa_{\mu}\pa_{\bar{\nu}}\log\Delta_n$ and show that it is equal to the hyperbolic contribution $\mathscr{E}_H$ to $\pa_{\mu}\pa_{\bar{\nu}}\log\det N_n$.

 In Section  \ref{a3} and Section \ref{a4}, we compute respectively the  parabolic contribution $\mathscr{E}_P$ and the elliptic contribution $\mathscr{E}_E$ to $\pa_{\mu}\pa_{\bar{\nu}}\log\det N_n$. These will complete our proof of the local index theorem. We will use techniques for computing Selberg trace formula (see for example \cite{Hejhal_2, Iwaniec, Teo_Sigma}) for this purpose.

 \smallskip
\section{The Second Variation of $\log\det \Delta_n$}\label{vary_det}

Let $n$ be a positive integer. For a cofinite Riemann surface of  type $(g;q;m_1, \ldots, m_v)$, the regularized determinant of $n$-Laplacian $\det\!'\Delta_n$ is defined in the following way (see \cite{Faddeev, Efrat, Koyama_3} or our discussion in \cite{Teo_Sigma}). Let $s$ be a positive number. We first compute the determinant of $\Delta_n+s(s+2n-1)$ using   zeta regularization. Using the terminologies in Section \ref{Laplace},   let
\begin{align*}
\zeta_X(w,s)= \sum_{k=0}^{\infty}\frac{1}{(\lambda_k+s(s+2n-1))^w} -\frac{1}{4\pi}\int_{-\infty}^{\infty}
\frac{1}{\di \left[\left(s+n-\tfrac{1}{2}\right)^2+r^2 \right]^w}\frac{\varphi'}{\varphi}\left(\frac{1}{2}+ir\right)dr
\end{align*}be the spectral zeta function of $X$ associated with the operator $\Delta_n+s(s+2n-1)$. This expression is well-defined when $\text{Re}\,w$ is large enough.
 It can be analytically continued to a neighbourhood of $w=0$. The zeta regularized determinant $\det (\Delta_n+s(s+2n-1))$ is defined as
\begin{align}\label{eq_det} 
\det (\Delta_n+s(s+2n-1))=\exp\left(-\left.\frac{\pa}{\pa w}\zeta_X(w,s)\right|_{w=0}\right).\end{align}
Since $\Delta_n$ has zero eigenvalues with multiplicity $d_n$, $\det (\Delta_n+s(s+2n-1))$ contains the term $[s(s+2n-1)]^{d_n}$.  The regularized determinant of $n$-Laplacian, $\det\!^{\prime}\Delta_n$, is defined by  removing this contribution from the zero eigenvalue. Namely,
\begin{equation}\label{eq314_2}\begin{split}\hspace{2cm}
\det\!^{\prime}\Delta_n=&\lim_{s\rightarrow 0}\frac{\det\left(\Delta_n+s(s+2n-1)\right)}{[s(s+2n-1)]^{d_n}}\\=&\frac{1}{(2n-1)^{d_n}}\lim_{s\rightarrow 0}\frac{\det\left(\Delta_n+s(s+2n-1)\right)}{s^{d_n}}.
\end{split}\end{equation}

The hyperbolic elements in $\Gamma$ can be divided into conjugacy classes. Let $\Pi$ be the set of these conjugacy classes.   For each representative $\gamma$ of a conjugacy class $[\gamma]$, all the elements in that class can be written as
$\alpha\gamma\alpha^{-1}$, where $\alpha$ runs trhough elements of $\Gamma_{\gamma}\backslash\Gamma$. For each $\gamma$, there is a unique $\gamma_0$ such that $\gamma_0$ is primmitive and $\gamma=\gamma_0^{\ell}$ for some positive integer $\ell$. Notice that $\Gamma_{\gamma}=\Gamma_{\gamma_0}$. 

When $\text{Re}\,s>1$, the
 Selberg zeta function of $X=\Gamma\backslash\mathbb{U}$ is  defined by the absolutely convergent product
\begin{align}\label{eq2_23_4}Z(s)=\prod_{[\gamma]\in P}\prod_{k=0}^{\infty}\left(1-\lambda(\gamma)^{-s-k}\right),
\end{align}where  $P$ is the set of  primitive hyperbolic conjugacy classes. For a hyperbolic element $\gamma$ of $\Gamma$,   $\lambda=\lambda(\gamma)>1$ is the multiplier of $\gamma$. It is the unique number larger than 1 for which there is a $\text{PSL}\,(2,\mathbb{R})$-element $\sigma$ such that $$\sigma^{-1} \gamma \sigma=\begin{pmatrix}  \lambda^{1/2} & 0\\0& \lambda^{-1/2}\end{pmatrix}.$$Any elements in the same conjugacy class has the same multiplier.

In \cite{Teo_Sigma}, we have proved the   explicit formula for $\det\!'\Delta_n$.
\begin{theorem}[\cite{Teo_Sigma}]\label{determinant}
When $n$ is a positive integer, the regularized determinant of the $n$-Laplacian $\Delta_n$ of $X$ is given by $$
\det\!^{\prime}\Delta_n=\begin{cases} \mathcal{C}_1Z'(1),\quad &n=1\\\mathcal{C}_nZ(n),\quad &n\geq 2\end{cases},
$$where
 \begin{equation}\label{cn}\begin{split}
\mathcal{C}_n=& \Bigl[ (2\pi)^{2n-1}\Gamma_2(2n)^2\Gamma(2n)^{2n-1}\Bigr]^{\frac{|X|}{4\pi}}\\&\times\prod_{j=1}^v\left\{m_j^{    \frac{ 2\alpha_{m_j}(-n)+1-m_j}{2m_j}}
  \prod_{r=1}^{m_j-1}\Gamma\left(\frac{r}{m_j}\right)^{\frac{2\alpha_{m_j}(r-n)+1-m_j}{2m_j}}\prod_{r=0}^{m_j-1}\Gamma\left(\frac{2n+r}{m_j}\right)^{\frac{2\alpha_{m_j}(r+n)+1-m_j}{2m_j}}\right\}\\
&\times  \left[\frac{2^{2n-1} }{ \pi\Gamma(2n)}\right]^{q/2}(2n-1)^{-d_n}\left(n-\tfrac{1}{2}\right)^{\frac{A}{2}}e^{B\left(n-\frac{1}{2}\right)^2+D},
\end{split}\end{equation}$\Gamma_2(s)$ is the  Alekseevskii-Barnes double gamma function, and $A$, $B$ and $D$ are the constants 
\begin{equation}
\label{constant}\begin{split}
A=&\text{Tr}\,\left[I-\Phi\left(\frac{1}{2}\right)\right],\\
B=&-\frac{|X|}{2\pi},\\
D=&\frac{|X|}{\pi}\zeta'(-1)+\frac{q}{2}\log(2\pi)
+\sum_{j=1}^v\left(\frac{m_j^2-1}{6m_j}-\frac{\alpha_{m_j}(n)(m_j-\alpha_{m_j}(n))}{m_j}\right)\log m_j.
\end{split}\end{equation}If $m$ is a positive integer greater than 1, and $k$ is an integer, $\alpha_m(k)$ is defined to be the least positive residue modulo $m$. 
\end{theorem}

The constant $A$ is an even integer \cite{Teo_LMP_2020}. Since the Teichm\"uller space $T(X)$ is connected \cite{Lehto}, $A$ must be a constant.  
Theorem \ref{determinant} shows that on the Teichm\"uller space $T(X)$, there is a constant $\mathcal{C}_n$ so that 
$$\det\!^{\prime}\Delta_n=\begin{cases}\mathcal{C}_1\log Z'(1),\quad & n=1,\\\mathcal{C}_n \log Z(n), \quad & n\geq 2\end{cases}.$$ 

Hence, for any $\mu, \nu \in\Omega_{-1,1}(X)$, 
\begin{align*}
\pa_{\mu}\pa_{\bar{\nu}}\log\det\!^{\prime}\Delta_n =\begin{cases}
\pa_{\mu}\pa_{\bar{\nu}}\log Z'(1),\quad & n=1,\\
\pa_{\mu}\pa_{\bar{\nu}}\log Z(n),\quad & n\geq 2.\end{cases}
\end{align*}

In fact, the logarithmic derivative of the Selberg zeta function $Z(s)$ appears in the Selberg trace formula as  the hyperbolic contribution. Hence,  it is natural to expect that $\pa_{\mu}\log\det\Delta_n$ and $\pa_{\mu}\pa_{\bar{\nu}}\log\det\Delta_n$ only involves hyperbolic contribution.

Since the definition of $Z(s)$   involves the multiplier of hyperbolic elements in $\Gamma$,
let us first  derive the variation of the multiplier of a hyperbolic element in $\Gamma$, a result that was proved in \cite{Gardiner}.

Let 
$$\gamma=\begin{pmatrix} a & b\\ c & d\end{pmatrix}$$ be a hyperbolic element of $\Gamma$. Given $\mu\in \Omega_{-1,1}(X)$, let
$$\gamma_{\varepsilon\mu}=\begin{pmatrix} a^{\varepsilon\mu} & b^{\varepsilon\mu}\\ c^{\varepsilon\mu} & d^{\varepsilon\mu}\end{pmatrix}$$ be the hyperbolic element of $\Gamma_{\varepsilon\mu}$ such thhat
$$\gamma_{\varepsilon\mu}\circ f_{\varepsilon\mu}=f_{\varepsilon\mu}\circ \gamma.$$There exists $\sigma_{\varepsilon\mu}\in \text{PSL}(2,\mathbb{R})$ such that 
$$\sigma_{\varepsilon\mu}^{-1}\circ\gamma_{\varepsilon\mu} \circ\sigma_{\varepsilon\mu}=\begin{pmatrix}\lambda_{\varepsilon\mu}^{1/2} & 0\\0&\lambda_{\varepsilon\mu}^{-1/2}\end{pmatrix},$$where $\lambda_{\varepsilon\mu}=\lambda(\gamma_{\varepsilon\mu})>1$ is the multiplier of $\gamma_{\varepsilon\mu}$. Let 
\begin{align}\label{eq3_2_1}\widehat{f}_{\varepsilon}=\sigma_{\varepsilon\mu}^{-1}\circ f_{\varepsilon\mu}\circ \sigma,\end{align} then $\widehat{f}_{\varepsilon}$ is a quasicomformal mapping with Beltrami differential
\begin{align}\label{eq2_23_5}\breve{\mu}=\mu\circ\sigma\frac{\overline{\sigma'}}{\sigma'}.\end{align}
Let \begin{align}\label{eq2_23_3}\dot\lambda=\left.\frac{\pa}{\pa\varepsilon}\right|_{\varepsilon=0}\lambda_{\varepsilon\mu}.\end{align}
From  $\gamma_{\varepsilon\mu}\circ f_{\varepsilon\mu}=f_{\varepsilon\mu}\circ\gamma$ and \eqref{eq3_2_1}, we have
$$\lambda_{\varepsilon\mu}\widehat{f}_{\varepsilon}(z)=\widehat{f}_{\varepsilon}(\lambda z).$$Taking partial derivative with respect to $\varepsilon$, we have
\begin{align*}\frac{\dot{\lambda}z}{\lambda}=\frac{1}{\lambda}\left.\frac{\pa}{\pa \varepsilon}\right|_{\varepsilon=0}\widehat{f}_{\varepsilon }(\lambda z) - \left.\frac{\pa}{\pa \varepsilon}\right|_{\varepsilon=0}\widehat{f}_{\varepsilon }(z)=\frac{1}{\lambda}F[\breve{\mu}](\lambda z)-F[\breve{\mu}](z),
\end{align*}where
$$F [\breve{\mu}](z)= \left.\frac{\pa}{\pa \varepsilon}\right|_{\varepsilon=0}\widehat{f}_{\varepsilon }(z).$$By \eqref{eq2_17_9}, we find that
 \begin{align*}
\frac{\dot{\lambda}z}{\lambda}=&-\frac{1}{\pi}\int_{\mathbb{U}} \breve{\mu}(\zeta)R(\lambda z,\zeta)\lambda^{-1}d^2\zeta+\frac{1}{\pi}\int_{\mathbb{U}}\breve{\mu}(\zeta)R(z,\zeta)d^2\zeta\\
=& -\frac{1}{\pi}\int_{\mathbb{U}} \breve{\mu}(\zeta)R(\lambda z,\lambda\zeta)\lambda d^2\zeta+\frac{1}{\pi}\int_{\mathbb{U}}\breve{\mu}(\zeta)R(z,\zeta)d^2\zeta\\
=&-\frac{1}{\pi}\int_{\mathbb{U}}\breve{\mu}(\zeta)\left\{\frac{\lambda z-z}{\zeta}-\frac{\lambda^2 z}{\lambda \zeta-1}+\frac{z}{\zeta-1}\right\}d^2\zeta.
 \end{align*}Let
 \begin{align}\label{eq2_14_1}S_{\lambda}=\left\{x+iy\,|\,-\infty<x<\infty, 1<y<\lambda\right\}.\end{align} Then
 \begin{align}\label{eq2_23_7}\mathbb{U}=\bigcup_{k=-\infty}^{\infty}\lambda^k(S_{\lambda}).\end{align}Hence,
 \begin{align*}
\frac{\dot{\lambda}}{\lambda}=&-\frac{1}{\pi}\int_{\mathbb{U}}\breve{\mu}(\zeta)\left\{\frac{\lambda  -1}{\zeta}-\frac{\lambda^2  }{\lambda \zeta-1}+\frac{1}{\zeta-1}\right\}d^2\zeta\\
=&-\frac{1}{\pi}\lim_{M\rightarrow\infty}\sum_{m=-M}^M\int_{S_{\lambda}}\breve{\mu}(\zeta)\left\{ \frac{\lambda  -1}{\lambda^m\zeta}-\frac{\lambda^2  }{\lambda^{m+1} \zeta-1}+\frac{1}{\lambda^{m}\zeta -1}\right\}\lambda^{2m}d^2\zeta\\
=&-\frac{1}{\pi}\lim_{M\rightarrow\infty}\sum_{m=-M}^M\int_{S_{\lambda}}\breve{\mu}(\zeta)\left\{ \frac{\lambda^{m+1}  -\lambda^m}{ \zeta}-\frac{\lambda^{2m+2}  }{\lambda^{m+1} \zeta-1}+\frac{\lambda^{2m}}{\lambda^{m}\zeta-1}\right\} d^2\zeta\\
=&-\frac{1}{\pi}\lim_{M\rightarrow\infty} \int_{S_{\lambda}}\breve{\mu}(\zeta)\left\{ \frac{\lambda^{M+1}  -\lambda^{-M}}{ \zeta}-\frac{\lambda^{2M+2}  }{\lambda^{M+1} \zeta-1}+\frac{\lambda^{-2M}}{\lambda^{-M}\zeta-1}\right\} d^2\zeta\\
=&-\frac{1}{\pi}\lim_{M\rightarrow\infty} \int_{S_{\lambda}}\breve{\mu}(\zeta)\left\{ \frac{\lambda^{M+1}  }{ \zeta}-\frac{\lambda^{2M+2}  }{\lambda^{M+1} \zeta-1} \right\} d^2\zeta\\
=&\frac{1}{\pi}\int_{S_{\lambda}}\frac{\breve{\mu}(\zeta)}{\zeta^2}d^2\zeta.
 \end{align*}
 This proves the result of Gardiner \cite{Gardiner}.
\begin{align}\label{eq2_23_8}
\frac{\dot{\lambda}}{\lambda}=&\frac{1}{\pi}\int_{S_{\lambda}}\frac{\breve{\mu}(\zeta)}{\zeta^2}d^2\zeta.
\end{align} 

 Our next goal is to prove the following formula for $\pa_{\mu}\log\Delta_n$. 

\begin{theorem}\label{thm2_23_1}When $n\geq 2$, 
\begin{equation}\label{eq2_17_5}\begin{split}
 \pa_{\mu}\log \det \Delta_n 
=&\frac{1}{ \pi}\int_X \sum_{\substack{\gamma\in\Gamma\\\gamma\;\text{hyperbolic}}}\frac{\gamma'(z)^n}{(\gamma z-z)^{2}}\left(\frac{z-\bar{z}}{\gamma z-\bar{z}}\right)^{2n-2}\mu(z)d^2z.
\end{split}\end{equation}
\end{theorem}This is more or less Lemma 4 in \cite{TZ_index_2}, except that we only take the summation over hyperbolic elements in $\Gamma$ instead of over all non-identity elements in $\Gamma$.  By \eqref{eq2_17_7}, 
$$\frac{1}{\pi} \frac{\gamma'(z)^n}{(\gamma z-z)^{2}}\left(\frac{z-\bar{z}}{\gamma z-\bar{z}}\right)^{2n-2}$$ is equal to $$-4  \left.\frac{\pa}{\pa z} \frac{\pa}{\pa w}\rho(w)^{1-n}\mathcal{G}_{n-1,1}(w,z)\right|_{w=\gamma z}\gamma'(z)^n$$ or $$ 4\rho(z)^{1-n}\left.\frac{\pa}{\pa z}\rho(z)^{-1}\frac{\pa}{\pa z}\mathcal{G}_{n, 0}(w,z)\right|_{w=\gamma z}\gamma'(z)^n.$$In \cite{Teo_Sigma}, we have shown that 
\begin{align}\label{eq2_23_2}\frac{1}{2s+2n-3}\frac{d}{ds}\log Z(s+n-1)=\int_X \sum_{\substack{\gamma\in\Gamma\\\gamma\;\text{hyperbolic}}}\mathcal{G}_{n-1,s}(\gamma z,z)\gamma'(z)^{n-1}\rho(z)^{2-n}d^2z.\end{align}One can use this as starting point to prove \eqref{eq2_17_5} as was shown in \cite{TZ_index_2}, with $\mathcal{G}_{0,s}(z,w)$ replaced by $\mathcal{G}_{n-1,s}(z,w)$. However, one then needs to get rid of the derivative with respect to $s$ carefully. In the following, we will give a direct proof similar to the manipulation given in \cite{Mcintyre}.

First we start with a lemma.
\begin{lemma}
\label{lemma2_23_1}
Given a hyperbolic element $\gamma$ in $\Gamma$ with multiplier $\lambda $, let $\dot{\lambda}$ be defined by \eqref{eq2_23_3}, and let $\gamma_0$ be a primitive hyperbolic element in $\Gamma$ so that $\lambda=\lambda_0^{\ell}$ for some positive integer $\ell$. If $\text{Re}\,s>1$ and   $\mu\in\Omega_{-1,1}(X)$, we have
\begin{align}\label{eq3_3_3}
L_{\mu}\log Z(s)=\sum_{[\gamma]\in \Pi}\left[\frac{\lambda^{2-s}}{(\lambda-1)^2}+(s-1)\frac{\lambda^{1-s}}{ \lambda -1}\right] \frac{\dot{\lambda}_0}{\lambda_0}.
\end{align}
\end{lemma}
\begin{proof}
Since   the infinite product in \eqref{eq2_23_4} that defines $Z(s)$ converges absolutely when $\text{Re}\,s>1$, we can differentiate term by term to obtain
\begin{align*}
L_{\mu}\log Z(s)=&\sum_{[\gamma_0]\in P}\sum_{k=0}^{\infty}\frac{(s+k)\lambda_0^{-s-k-1}}{1-\lambda_0^{-s-k}}\dot{\lambda}_0\\
=&\sum_{[\gamma_0]\in P}\sum_{k=0}^{\infty}\sum_{\ell=1}^{\infty} (s+k)\lambda_0^{-\ell(s+k) } \frac{\dot{\lambda}_0}{\lambda_0}\end{align*}
Using the fact that each $[\gamma]\in \Pi$ is equal to $[\gamma_0^{\ell}]$ for a unique $[\gamma_0]\in P$ and a unique positive integer $\ell$, we find that 
\begin{align*}
L_{\mu}\log Z(s)=&\sum_{[\gamma]\in \Pi}\sum_{k=0}^{\infty}  (s+k)\lambda^{-s-k } \frac{\dot{\lambda}_0}{\lambda_0}\\
=&\sum_{[\gamma]\in \Pi}\lambda^{-s}\left[\frac{1}{(1-\lambda^{-1})^2}+\frac{s-1}{1-\lambda^{-1}}\right] \frac{\dot{\lambda}_0}{\lambda_0}.
\end{align*}The assertion follows.
\end{proof}

\begin{proof}[Proof of Theorem \ref{thm2_23_1}]
 We will prove the Theorem for $n\geq 2$. For the $n=1$ case, extra care need to be taken since the infinite product \eqref{eq2_23_4} is divergent when $s=1$. We will discuss this case in Appendix \ref{n1}.

Call the right-hand side of \eqref{eq2_17_5} as $T_n$. Dividing the hyperbolic elements into conjugacy classes, we have
\begin{align}\label{eq2_17_10}T_n =\sum_{[\gamma]\in\Pi}\frac{1}{ \pi}\int_X\sum_{\alpha\in\Gamma_{\gamma}\backslash\Gamma}\frac{(\alpha\gamma\alpha^{-1})'(z)^n}{(\alpha\gamma\alpha^{-1}z-z)^{2}}\left(\frac{z-\bar{z}}{\alpha\gamma\alpha^{-1} z-\bar{z}}\right)^{2n-2}\mu(z)d^2z.\end{align}
Now it is sufficient to simplify each term in the summation. Using $\sigma$ that conjugates $\gamma$ to a diagonal element, we find that
\begin{align*}
T_n[\gamma]=&\frac{1}{\pi}\int_X\sum_{\alpha\in\Gamma_{\gamma}\backslash\Gamma}\frac{(\alpha\gamma\alpha^{-1})'(z)^n}{(\alpha\gamma\alpha^{-1}z-z)^{2}}\left(\frac{z-\bar{z}}{\alpha\gamma\alpha^{-1} z-\bar{z}}\right)^{2n-2}\mu(z)d^2z\\
=&\frac{1}{\pi} \int_{S_{\lambda_0}} \frac{\lambda^n}{(\lambda-1)^2 z^2}\left(\frac{z-\bar{z}}{\lambda z-\bar{z}}\right)^{2n-2}\breve{\mu}(z)d^2z,
\end{align*}where $\breve{\mu}=\breve{\mu}[\gamma]$ is defined by \eqref{eq2_23_5}, with $\sigma$ an element of $\text{PSL}\,(2,\mathbb{R})$ that conjugates $\gamma$ to $$\begin{pmatrix}\lambda^{1/2}  & 0\\0 & \lambda^{-1/2}\end{pmatrix}. $$ Notice that since $\Gamma_{\gamma}=\Gamma_{\gamma_0}$, after unravelling the summation over $\alpha\in \Gamma_{\gamma}$, the integration is over $S_{\lambda_0}$, not $S_{\lambda}$. Now we use \eqref{eq2_6_6} to write $T_n[\gamma]$ as
\begin{align*}
T_n[\gamma]=&-\frac{3}{\pi^2} \int_{\mathbb{U}}\int_{S_{\lambda_0}}\breve{\mu}(\zeta) \frac{\lambda^n}{(\lambda-1)^2 z^2} \frac{(z-\bar{z})^{2n}}{(\lambda z-\bar{z})^{2n-2}}
\frac{1}{(\zeta-\bar{z})^4}d^2z d^2\zeta.
\end{align*}Using the fact that 
$$\breve{\mu}(\lambda_0z)=\breve{\mu}(z),$$ as well as
\eqref{eq2_23_7}, we can fold up the integral over $\zeta$ and unfold the integral over $z$ to get\begin{align*}
T_n[\gamma]=&-\frac{3}{\pi^2} \int_{S_{\lambda_0}}\breve{\mu}(\zeta) \int_{\mathbb{U}}\frac{\lambda^n}{(\lambda-1)^2 z^2} \frac{(z-\bar{z})^{2n}}{(\lambda z-\bar{z})^{2n-2}}
\frac{1}{(\zeta-\bar{z})^4}d^2z d^2\zeta.
\end{align*}Now let us concentrate on the integral
\begin{align*}
Q(\zeta,\lambda)=&\frac{3}{\pi}  \int_{\mathbb{U}}\frac{\lambda^n}{(\lambda-1)^2 z^2} \frac{(z-\bar{z})^{2n}}{(\lambda z-\bar{z})^{2n-2}}
\frac{1}{(\zeta-\bar{z})^4}d^2z.
\end{align*}
Applying integration by parts with respect to $\bar{z}$ three times, we have
\begin{align*}
Q(\zeta,\lambda)=&-\frac{1}{2\pi}  \int_{\mathbb{U}}\frac{\lambda^n}{(\lambda-1)^2 z^2}\frac{\pa^3}{\pa\bar{z}^3}\left[ \frac{(z-\bar{z})^{2n}}{(\lambda z-\bar{z})^{2n-2}}\right]
\frac{1}{(\zeta-\bar{z}) }d^2z\\
=&\frac{2n(n-1)(2n-1)}{ \pi} \lambda^n(\lambda-1) \int_{\mathbb{U}} \frac{z}{(\zeta-\bar{z}) }  \frac{(z-\bar{z})^{2n-3}}{(\lambda z-\bar{z})^{2n+1}} 
d^2z.
\end{align*}By inspecting the behaviors of the integrands when $y\rightarrow 0^+$ and $y\rightarrow\infty$, one can see that no boundary terms appear during the integration by parts.

Obviously, $Q(\zeta, \lambda)$ is a holomorphic function of $\zeta\in \mathbb{U}$. Hence, it is sufficient to compute $Q(\zeta,\lambda)$ when $\zeta=iv$, $v>0$. We have
\begin{align*}
Q(iv, \lambda)=&(-1)^{n} 2^{2n-1}n(n-1)(2n-1)\lambda^n (\lambda-1) \\&\times\frac{1}{2\pi i} \int_0^{\infty} \int_{-\infty}^{\infty}\frac{x+iy}{(x-i(y+v))}
\frac{y^{2n-3}}{ ((\lambda-1)x+i(\lambda+1)y)^{2n+1}}dxdy.
\end{align*}Now making a change of variables $x\mapsto vx, y\mapsto vy$, we find that $$Q(iv, \lambda)=\frac{1}{v^2}Q(i, \lambda).$$ This implies that
$$Q(\zeta, \lambda)=-\frac{1}{\zeta^2}Q(i, \lambda),$$ and hence, by \eqref{eq2_23_8}, we have
$$T_n[\gamma]=\frac{1}{\pi}\int_{S_{\lambda_0}}\frac{\breve{\mu}(\zeta)}{\zeta^2}d^2\zeta\times Q(i, \lambda)=Q(i,\lambda)\frac{\dot{\lambda}_0}{\lambda_0}.$$
The computation of 
\begin{align*}
Q(i, \lambda)=&  (-1)^{n}2^{2n-1}n(n-1)(2n-1)\lambda^n(\lambda-1) \\&\hspace{2cm}\times\frac{1 }{2\pi i}\int_0^{\infty}\int_{-\infty}^{\infty}\frac{x+iy}{(x-i(y+1))}\frac{y^{2n-3}}{\left((\lambda-1)x+i(\lambda+1)y\right)^{2n+1}}dxdy
\end{align*}is elementary but tedious. For fixed $y>0$, 
$$\int_{-\infty}^{\infty}\frac{x+iy}{(x-i(y+1))}\frac{1}{\left((\lambda-1)x+i(\lambda+1)y\right)^{2n+1}}dx$$ can be considered as an integral of the analytic function
$$g(w)=\frac{w+iy}{(w-i(y+1))}\frac{1}{\left((\lambda-1)w+i(\lambda+1)y\right)^{2n+1}}$$ of $w$ over the real-line. Using contour integration technique,  the integral is equal to $2\pi i$ times the residue of $g(w)$ at the pole $w=i(y+1)$ on the upper half-plane. Therefore,
\begin{align*}
Q(i,\lambda)=&2^{2n-1}n(n-1)(2n-1)\lambda^n(\lambda-1)   \int_0^{\infty}\frac{y^{2n-3}(2y+1)}{\left(2\lambda y+\lambda-1\right)^{2n+1}}dy.
\end{align*}Let $$a=\frac{\lambda-1}{2\lambda}.$$
Then
\begin{align*}
Q(i,\lambda)=&\frac{n(n-1)(2n-1)}{4}\lambda^{-n-1}(\lambda-1)   \int_0^{\infty}\frac{y^{2n-3}(2y+1)}{\left(  y+a\right)^{2n+1}}dy.
\end{align*}Using the fact that if $k$ and $r$ are integers such that $r\geq 1$, $k\geq r+2$, then
$$\int_0^{\infty}\frac{y^r}{(y+a)^k}dy=\frac{1}{a^{k-r-1}}\frac{r!(k-r-2)!}{(k-1)!},$$ we finally obtain
\begin{align*}
Q(i,\lambda)=&\frac{n(n-1)(2n-1)}{4}\lambda^{-n-1}(\lambda-1) \left\{\frac{1}{a^2}\frac{1}{n(2n-1)}+\frac{1}{a^3}\frac{1}{2n(2n-1)(n-1)}\right\}\\
=&(n-1)\frac{\lambda^{1-n}}{\lambda-1}+\frac{\lambda^{2-n}}{(\lambda-1)^2}.
\end{align*}Therefore, we have shown that
$$T_n[\gamma]=\left\{(n-1)\frac{\lambda^{1-n}}{ \lambda-1}+\frac{\lambda^{2-n}}{(\lambda-1)^2}\right\}\frac{\dot{\lambda}_0}{\lambda_0}.$$ By summing over conjugacy classes of hyperbolic elements and using Lemma \ref{lemma2_23_1}, we prove the formula \eqref{eq2_17_5}. 
\end{proof}

\begin{remark}\label{remark3_9_2}
In fact, in the proof of Theorem \ref{thm2_23_1}, we have shown that if $\gamma$ is a hyperbolic element of $\Gamma$ and $\gamma=\gamma_0^{\ell}$, where $\gamma_0$ is primitive and $\ell$ is a positive integer, then
\begin{equation}\label{eq3_3_4}\begin{split}T_n[\gamma]=&\frac{1}{\pi}\int_X\sum_{\alpha\in\Gamma_{\gamma}\backslash\Gamma}\frac{(\alpha\gamma\alpha^{-1})'(z)^n}{(\alpha\gamma\alpha^{-1}z-z)^{2}}\left(\frac{z-\bar{z}}{\alpha\gamma\alpha^{-1} z-\bar{z}}\right)^{2n-2}\mu(z)d^2z\\=&\left\{(n-1)\frac{\lambda^{1-n}}{\lambda-1}+\frac{\lambda^{2-n}}{(\lambda-1)^2}\right\}\frac{\dot{\lambda}_0}{\lambda_0}. \end{split}\end{equation}Fixed a primitive hyperbolic element $\gamma_0$ and summing up the contribution from all $\gamma_0^{\ell}$ for $\ell\geq 1$, the derivation of \eqref{eq3_3_3} and \eqref{eq3_3_4} show  that
\begin{equation}\label{eq2_23_6}\begin{split}
L_{\mu}\log\left[ \prod_{k=0}^{\infty}\left(1-\lambda_0^{-k-n}\right)\right]=&\sum_{\ell=1}^{\infty}\left\{(n-1)\frac{\lambda_0^{\ell(1-n)}}{\lambda_0^{\ell}-1}+\frac{\lambda_0^{\ell(2-n)}}{(\lambda_0^{\ell}-1)^2}\right\}\frac{\dot{\lambda}_0}{\lambda_0}\\
=&\frac{1}{\pi}\int_X\sum_{\gamma\in \tilde{\Gamma}[\gamma_0]} \frac{\gamma'(z)^n}{(\gamma z-z)^{2}}\left(\frac{z-\bar{z}}{\gamma z-\bar{z}}\right)^{2n-2}\mu(z)d^2z,
\end{split}\end{equation}where $$\tilde{\Gamma}[\gamma_0]=\left\{\alpha\gamma_0^{\ell}\alpha^{-1}\,|\,\alpha\in \Gamma_{\gamma_0}\backslash\Gamma, \ell\in\mathbb{Z}^+\right\}$$contains all elements in $\Gamma$ that are conjugate to $\gamma_0^{\ell}$ for some positive integer $\ell$. 

Eq. \eqref{eq2_23_6} can be considered as the variational formula of the local zeta function
$$\prod_{k=0}^{\infty}\left(1-\lambda_0^{-s-k}\right)$$at $s=n$. This formula still holds for $n=1$. In fact, using \eqref{eq2_23_8}, the $n=1$ case is almost a tautology.
\end{remark}

Now we proceed to consider the second variation of $\log\det\Delta_n$.
\begin{theorem}\label{thm2_22_6}When $n\geq 2$, 
\begin{equation}\label{eq2_22_8}\begin{split}
&\pa_{\mu}\pa_{\bar{\nu}}\log \det \Delta_n\\=&- \frac{n}{\pi^2}\sum_{\substack{\gamma\in\Gamma\\\gamma\;\text{hyperbolic}}}\int_X \int_{\mathbb{U}} \frac{\mu(z)\overline{\nu( \eta)}}{(z-\bar{\eta})^2(\gamma z-\bar{\eta})^2}  \frac{(  z-\bar{ z})^{2n-2} }{(\gamma z-\bar{z} )^{2n-2}}    \gamma'(z) ^n  d^2\eta  d^2z \\
& - \frac{2(n-1)}{\pi^2} \sum_{\substack{\gamma\in\Gamma\\\gamma\;\text{hyperbolic}}}\int_X  \int_{\mathbb{U}} \frac{\mu(z)\overline{\nu( \eta)}}{(z-\bar{\eta})^3(\gamma z-\bar{\eta}) }  \frac{(  z-\bar{ z})^{2n-1} }{(\gamma z-\bar{z} )^{2n-1}}    \gamma'(z) ^n  d^2\eta  \\
&+\frac{(n-1)(2n-1)}{ \pi} \sum_{\substack{\gamma\in\Gamma\\\gamma\;\text{hyperbolic}}}\int_X   \int_{ \mathbb{U}} \frac{(z-\bar{z})^{2n-2}}{(\gamma z-\bar{z} )^{2n}} \gamma'(z)^n \mathcal{G}(z,w) \mu(w)\overline{\nu(w)}\rho(w)d^2wd^2z.
\end{split}\end{equation}
\end{theorem}
\begin{proof}
Applying $\pa_{\bar{\nu}}$ to \eqref{eq2_17_5} using Proposition \ref{prop_vary} and Proposition \ref{varymetric}, we have
\begin{align*}
\pa_{\mu}\pa_{\bar{\nu}}\log \det \Delta_n=&\mathfrak{D}_1+\mathfrak{D}_2+\mathfrak{D}_3,
\end{align*}where
\begin{align*}
\mathfrak{D}_1=&-\frac{1}{ \pi^2} \sum_{\substack{\gamma\in\Gamma\\\gamma\;\text{hyperbolic}}}\int_X\int_{\mathbb{U}}\frac{\mu(z)\overline{\nu( \eta)}}{(z-\bar{\eta})^2(\gamma z-\bar{\eta})^2}  \frac{(  z-\bar{ z})^{2n-2} }{(\gamma z-\bar{z} )^{2n-2}}    \gamma'(z) ^n  d^2\eta  d^2z \\
\mathfrak{D}_2=&-\frac{(n-1)}{\pi^2}\sum_{\substack{\gamma\in\Gamma\\\gamma\;\text{hyperbolic}}}\int_X\int_{\mathbb{U}}\frac{\mu(z)\overline{\nu( \eta)}}{(\gamma z -z)^2(\bar{z}-\bar{\eta})^2(\gamma z-\bar{\eta})^2}  \frac{(  z-\bar{ z})^{2n-2} }{(\gamma z-\bar{z} )^{2n-4}}    \gamma'(z) ^n  d^2\eta  d^2z,\\
\mathfrak{D}_3=&\frac{1}{ \pi} \sum_{\substack{\gamma\in\Gamma\\\gamma\;\text{hyperbolic}}}\int_X\frac{\gamma'(z)^n}{(\gamma(z)-z)^{2}}\left(\frac{z-\bar{z}}{\gamma( z)-\bar{z}}\right)^{2n-2}(L_{\bar{\nu}}\mu)(z)d^2z.
\end{align*}By the complex conjugate version of Proposition \ref{prop2_2_1},
\begin{align}
\mathfrak{D}_2=&-\frac{(n-1)}{\pi^2}\sum_{\substack{\gamma\in\Gamma\\\gamma\;\text{hyperbolic}}}\int_X\int_{\mathbb{U}}\frac{\mu(z)\overline{\nu( \eta)}}{ (z-\bar{\eta})^2(\gamma z-\bar{\eta})^2}  \frac{(  z-\bar{ z})^{2n-2} }{(\gamma z-\bar{z} )^{2n-2}}    \gamma'(z) ^n  d^2\eta  d^2z\label{eq2_22_9}\\
&-\frac{2(n-1)}{\pi^2}\sum_{\substack{\gamma\in\Gamma\\\gamma\;\text{hyperbolic}}}\int_X\int_{\mathbb{U}}\frac{\mu(z)\overline{\nu( \eta)}}{ (z-\bar{\eta})^3(\gamma z-\bar{\eta}) }  \frac{(  z-\bar{ z})^{2n-1} }{(\gamma z-\bar{z} )^{2n-1}}    \gamma'(z) ^n  d^2\eta  d^2z.\nonumber
\end{align}Notice that \eqref{eq2_22_9} is equal to $(n-1)\mathfrak{D}_1$. 
Hence, $\mathfrak{D}_1+\mathfrak{D}_2$ gives the first two terms in \eqref{eq2_22_8}.

Now for $\mathfrak{D}_3$, Proposition \ref{vary_differential} gives
\begin{align*}
\mathfrak{D}_3=&-\frac{4}{ \pi} \sum_{\substack{\gamma\in\Gamma\\\gamma\;\text{hyperbolic}}}\int_X\frac{\gamma'(z)^n}{(\gamma(z)-z)^{2}}\left(\frac{z-\bar{z}}{\gamma z-\bar{z}}\right)^{2n-2}\frac{\pa}{\pa\bar{z}}\rho(z)^{-1}\frac{\pa}{\pa\bar{z}}(\Delta_0+2)^{-1}(\mu\bar{\nu})(z)d^2z.
\end{align*}Integration by parts twice, we find that
\begin{align*}
\mathfrak{D}_3=& \frac{4}{ \pi}\sum_{\substack{\gamma\in\Gamma\\\gamma\;\text{hyperbolic}}}\int_X \frac{\pa}{\pa\bar{z}}\left\{\frac{\gamma'(z)^n}{(\gamma z-z)^{2}}\left(\frac{z-\bar{z}}{\gamma z-\bar{z}}\right)^{2n-2}\right\}\rho(z)^{-1}\frac{\pa}{\pa\bar{z}}(\Delta_0+2)^{-1}(\mu\bar{\nu})(z)d^2z\\
=&  \frac{2(n-1)}{ \pi} \sum_{\substack{\gamma\in\Gamma\\\gamma\;\text{hyperbolic}}} \int_X\frac{\gamma'(z)^n}{(\gamma z-z) }\left( \frac{z-\bar{z}}{\gamma z-\bar{z}}\right)^{2n-1}  \frac{\pa}{\pa\bar{z}}(\Delta_0+2)^{-1}(\mu\bar{\nu})(z)d^2z\\
=& - \frac{2(n-1)}{ \pi}  \sum_{\substack{\gamma\in\Gamma\\\gamma\;\text{hyperbolic}}}\int_X\frac{\pa}{\pa\bar{z}}\left\{ \frac{\gamma'(z)^n}{(\gamma z-z) }\left( \frac{z-\bar{z}}{\gamma z-\bar{z}}\right)^{2n-1}  \right\}(\Delta_0+2)^{-1}(\mu\bar{\nu})(z)d^2z\\
=&\frac{(n-1)(2n-1)}{ \pi}  \sum_{\substack{\gamma\in\Gamma\\\gamma\;\text{hyperbolic}}} \int_X  \frac{(z-\bar{z})^{2n-2}}{(\gamma z-\bar{z} )^{2n}} \gamma'(z)^n  \int_{ \mathbb{U}}\mathcal{G}(z,w) \mu(w)\overline{\nu(w)}\rho(w)d^2wd^2z.
\end{align*}This completes the proof of the theorem.

When we perform integration by parts in the computations of $\mathfrak{D}_3$,  there is no boundary terms that appear. One can verify this by writing the sum over all hyperbolic elements into a double sum, the outer sum is over the conjugacy classes, and the inner sum is over all hyperbolic elements in that class. For a fixed conjugacy class, conjugate a hyperbolic representative to a diagonal element. Then the sum of integrals over all hyperbolic elements in that conjugacy class can be combined  into a single integral over a domain of the form $S_{\lambda}$ \eqref{eq2_14_1}. When we perform integration by parts, the invariant properties of the integrands  clearly show that no boundary terms appear.

\end{proof}

Compare Theorem \ref{thm2_22_6} to Theorem \ref{thm2_22_3}, we immediately obtain
\begin{theorem}\label{thm2_22_7} When $n\geq 2$,
 $\pa_{\mu}\pa_{\bar{\nu}}\log\Delta_n$ is  equal to the hyperbolic contribution $\mathscr{E}_H$ to $\pa_{\mu}\pa_{\bar{\nu}}\log N_n$. 
\end{theorem}

 \smallskip
\section{Parabolic Contribution}\label{a3}
In this section, we compute the parabolic contribution $\mathscr{E}_P$ to $\pa_{\mu}\pa_{\bar{\nu}}\log\det N_n$ when the surface $X$ has cusps. The result is
\begin{theorem}\label{thm2_24_1}
Let $n\geq 2$ and $q\geq 1$. Given $\mu,\nu\in\Omega_{-1,1}(X)$, the parabolic  contribution to $\pa_{\mu}\pa_{\bar{\nu}}\log\det N_n$ is given by
\begin{align} \mathscr{E}_P=\frac{\pi}{9}\langle \mu, \nu\rangle_{\text{TZ}}^{\text{cusp}}.\label{main_result_P}\end{align}
\end{theorem}

In the following, we give the details of the proof of this theorem.

The Riemann surface $X$ has $q$ cusps corresponding to the $q$ parabolic elements $\kappa_1, \ldots, \kappa_q$. For each $1\leq j\leq q$, there is a $\sigma_j\in \text{PSL}\,(2,\mathbb{R})$ such that
\begin{align}\label{eq2_24_4}\sigma_j^{-1}\kappa_j\sigma_j=\begin{pmatrix} 1 & \pm 1\\0 & 1\end{pmatrix}.\end{align}
If $\gamma$ is a parabolic element of $\Gamma$, there is an $\alpha\in \Gamma$, $1\leq j\leq q$ and  a nonzero integer $\ell$ such that
$$\alpha^{-1}\gamma\alpha=\kappa_j^{\ell}=\sigma_jT_{\pm\ell}\sigma_j^{-1},$$
where $$T_{\ell}=\begin{pmatrix} 1 & \ell\\0 & 1\end{pmatrix}.$$
For $1\leq j\leq q$, the stabilizer of $\kappa_j$ is given by
$$\Gamma_{\kappa_j}=\sigma_jB\sigma_j^{-1},$$ where $B$ is the parabolic subgroup of $\text{PSL}\,(2,\mathbb{R})$ given by
\begin{align}\label{eq2_25_1}B=\left\{T_{\ell}\,|\, \ell\in\mathbb{Z}\right\}.\end{align}

Given $\mu\in\Omega_{-1,1}(X)$, $1\leq j\leq q$, let
\begin{align}\label{eq2_14_7}\hat{\mu}_j(z)=\mu\circ\sigma_j (z)\frac{\overline{\sigma_j '(z)}}{\sigma_j'(z)}.\end{align}Then $\hat{\mu}_j$ is an automorphic $(-1,1)$-form of the group $\sigma_j^{-1}\Gamma\sigma_j$. Since $B$ is a subgroup of  $\sigma_j^{-1}\Gamma\sigma_j$, we find that
\begin{align}\label{eq2_4_1}\hat{\mu}_j(z+1)=\hat{\mu}_j(z)\hspace{1cm}\text{for all}\;z\in\mathbb{U}.\end{align} By definition of harmonic Beltrami differentials, 
$y^{-2}\hat{\mu}_j(z)$ is antiholomorphic. Together with the periodicity \eqref{eq2_4_1}, we find that $\hat{\mu}_j(z)$ can be written as
\begin{align}\label{eq2_18_8}\hat{\mu}_j(z)=y^2\sum_{k=1}^{\infty} \beta_k^{(j)}e^{-2\pi i k\bar{z}}.\end{align}
We can choose a fundamental domain $F$   of $X$ on $\mathbb{U}$ such that
$$\mathfrak{S}=\bigcup_{\alpha\in \Gamma_{\kappa_j}\backslash\Gamma}\left(\sigma_j^{-1}\circ \alpha^{-1}\right) (\overline{F})=\left\{x+iy\,|\,0\leq x\leq 1, y>0\right\}.$$
Then
\begin{align*}
\langle\mu, \mu\rangle_{\text{TZ}, j}^{\text{cusp}}=&\iint\limits_F \mu(z)\overline{\mu(z)}\sum_{\alpha\in \Gamma_{\kappa_j}\backslash\Gamma}\left[\text{Im}\,\left(\sigma_j^{-1}\alpha z\right)\right]^{2}\frac{dxdy}{y^2}\\
=&\int_0^{\infty}\int_0^1\hat{\mu}(z)\overline{\hat{\mu}(z)}dxdy\\
=&\int_0^{\infty}\sum_{k=1}^{\infty}|\beta_k^{(j})|^2e^{-4\pi k y}y^4dy\\
=&\frac{3}{128 \pi^5}\sum_{k=1}^{\infty}\frac{|\beta_k^{(j})|^2}{k^5}.
\end{align*}

This is essentially eq.\ (2.10) in \cite{TZ_index_2}. Let us state it as a proposition.
\begin{proposition} \label{prop2_24_2}
Let $\kappa_j$ be a parabolic generator of the group $\Gamma$ and let $\sigma_j$ be an element of $\text{PSL}\,(2, \mathbb{R})$ so that \eqref{eq2_24_4} holds. Given $\mu\in\Omega_{-1,1}(X)$, let
$$ \mu\circ\sigma_j (z)\frac{\overline{\sigma_j '(z)}}{\sigma_j'(z)}=y^2\sum_{k=1}^{\infty} \beta_k^{(j)}e^{-2\pi i k\bar{z}}.$$Then
\begin{align}\label{eq2_24_5}
\langle\mu, \mu\rangle_{\text{TZ}, j}^{\text{cusp}}=&\frac{3}{128 \pi^5}\sum_{k=1}^{\infty}\frac{|\beta_k^{(j})|^2}{k^5}.
\end{align}
\end{proposition}

Next, we give a proof of the following proposition which is a more precise version of Lemma 2 in \cite{TZ_index_2}.
\begin{proposition}\label{prop2_24_1}
Given $\mu\in\Omega_{-1,1}(X)$,  let $\hat{\mu}_j$ be defined by \eqref{eq2_14_7}, with Fourier expansion given by \eqref{eq2_18_8}. Then 
$2 (\Delta_0+2)^{-1}(\mu\bar{\mu})(\sigma_j z)=2 (\Delta_0+2)^{-1}(\hat{\mu}_j\bar{\hat{\mu}}_j)(z)$ has a Fourier expansion of the form
\begin{align}\label{eq2_24_12}
2 (\Delta_0+2)^{-1}(\mu\bar{\mu})(\sigma_j z)=\sum_{r=-\infty}^{\infty}c_r(y)e^{2\pi i rx}.
\end{align}The term  $c_0(y)$ can be written as a sum of two terms $\alpha(y)$ and $\beta(y)$, where 
$\alpha(y)$ is a multiple of $y^{-1}$ given by
\begin{align}\label{eq2_25_3}
\alpha(y)=\frac{2}{3y}\langle\mu, \mu\rangle_{\text{TZ}, j}^{\text{cusp}},
\end{align}and $\beta(y)$ is an exponentially decaying term  given by
\begin{align}\label{eq2_25_2}
\beta(y)=-  \frac{1}{64\pi^5y}\sum_{k=1}^{\infty}\frac{1}{k^5}\left|\beta_k^{(j)}\right|^2\left(1 + 4\pi ky + 8\pi^2k^2y^2 + 8\pi^3k^3y^3\right)  e^{ -4\pi k y}.
\end{align}
\end{proposition}
\begin{proof}
Let $$g(z)=2(\Delta_0+2)^{-1}(\mu\bar{\mu})(\sigma_j z).$$
Since the function $g(z)$ is an automorphic function of the group $\sigma_j^{-1}\Gamma\sigma_j$, which contains $B$ \eqref{eq2_25_1} as a subgroup, we find that 
$$g(z+1)=g(z)\hspace{1cm}\text{for all}\;z\in\mathbb{U}.$$ This implies that $g(z)$ has a Fourier expansion of the form \eqref{eq2_24_12}. By the theory of Fourier series, the term $c_0(y)$ is given by
\begin{align*}
c_0(y)=& \int_0^1 g(x+iy)dx.
\end{align*}
By Proposition \ref{integral2_24_1}, 
\begin{align*}
g(z)=&-\frac{8y^2}{\pi^2}\int_{\mathbb{U}}\int_{\mathbb{U}}\hat{\mu}_j(\zeta)\overline{\hat{\mu}_j(\eta)}\left\{  \frac{1}{( \zeta -\bar{z})^2(\zeta-\bar{\eta})^2(z-\bar{\eta})^2}
+ \frac{2iy }{( \zeta -\bar{z})^3(\zeta-\bar{\eta}) (z-\bar{\eta})^3}\right\}d^2\zeta d^2\eta.
\end{align*}
 By invariance properties,
\begin{align*}
c_0(y)=&-\frac{8y^2}{\pi^2}\int_{-\infty}^{\infty}\int_{\mathfrak{S}}\int_{\mathbb{U}}\hat{\mu}_j(\zeta)\overline{\hat{\mu}_j(\eta)}\left\{  \frac{1}{( \zeta -x+iy)^2(\zeta-\bar{\eta})^2(x+iy-\bar{\eta})^2}
\right.\\&\hspace{3cm}\left.+ \frac{2iy }{( \zeta -x+iy)^3(\zeta-\bar{\eta}) (x+iy-\bar{\eta})^3}\right\}d^2\zeta d^2\eta dx\\
=&c_{0,1}(y)+c_{0,2}(y).
\end{align*}
Let us integrate over $x$ first. This can be done using contour integration technique. In this case, we can consider the real line as the boundary of the upper half plane or the lower half plane. Let us consider it as the boundary of the lower half plane. As a function of the complex variable $w$,
$$ \frac{1}{( \zeta -w+iy)^2(\zeta-\bar{\eta})^2(w+iy-\bar{\eta})^2}$$has a double pole at $w=\bar{\eta}-iy$ on the lower half plane. Therefore,
\begin{align*}
\int_{-\infty}^{\infty} \frac{1}{( \zeta -x+iy)^2(\zeta-\bar{\eta})^2(x+iy-\bar{\eta})^2}dx=&-2\pi i \left.\frac{\pa}{\pa w}\right|_{w=\bar{\eta}-iy} \frac{1}{( \zeta -w+iy)^2(\zeta-\bar{\eta})^2}\\
=&-\frac{4\pi i}{(\zeta-\bar{\eta}+2iy)^3(\zeta-\bar{\eta})^2}.
\end{align*}Similarly,
\begin{align*}
\int_{-\infty}^{\infty} \frac{2iy}{( \zeta -x+iy)^3(\zeta-\bar{\eta})(x+iy-\bar{\eta})^3}dx=&-\frac{2\pi i}{2} \times 2iy\left.\frac{\pa^2}{\pa w^2}\right|_{w=\bar{\eta}-iy} \frac{1}{( \zeta -w+iy)^3(\zeta-\bar{\eta})}\\
=&\frac{24\pi y}{(\zeta-\bar{\eta}+2iy)^5(\zeta-\bar{\eta})}.
\end{align*}
Now we compute the integrals with respect to $\zeta$ and $\eta$. Writing $\zeta=a+ib$, $\eta=u+iv$, we have
\begin{align*}
c_{0,1}(y)=&\frac{32 i y^2}{\pi} \int_{0}^{\infty} \int_0^{\infty} \int_0^1 \int_{-\infty}^{\infty}b^2v^2\sum_{k=1}^{\infty} \beta_k^{(j)}e^{-2\pi i k(a-ib)}\sum_{m=1}^{\infty}
\overline{\beta_m^{(j)}}e^{2\pi i m(u+iv)}\\&\hspace{3cm}\times\frac{1}{(a-u+i(b+v+2y))^3(a-u+i(b+v))^2}dadudbdv.
\end{align*}For fixed $u$, make a change of variables $a\mapsto a+u$. Then integrate over $u$ first, we find that 
\begin{align*}
c_{0,1}(y)=&\frac{32 i y^2}{\pi} \int_{0}^{\infty}\int_0^{\infty}    \int_{-\infty}^{\infty}b^2v^2\sum_{k=1}^{\infty} \left|\beta_k^{(j)}\right|^2e^{-2\pi i ka}e^{-2\pi k(b+v)}\\&\hspace{3cm}\times \frac{1}{(a+i(b+v+2y))^3(a+i(b+v))^2}da dbdv.
\end{align*}Making another change of variables $b\mapsto b-v$ and integrate over $v$ first, we have
\begin{align*}
c_{0,1}(y)=&\frac{32 i y^2}{\pi} \int_{0}^{\infty}\int_0^{b}    \int_{-\infty}^{\infty}(b-v)^2v^2\sum_{k=1}^{\infty} \left|\beta_k^{(j)}\right|^2  \frac{e^{-2\pi i ka}e^{-2\pi kb}}{(a+i(b +2y))^3(a+ib)^2}da dvdb \\
=& \frac{16  i y^2}{15\pi}\int_{0}^{\infty}b^5\int_{-\infty}^{\infty} \sum_{k=1}^{\infty} \left|\beta_k^{(j)}\right|^2 \frac{e^{-2\pi i ka}e^{-2\pi kb}}{(a+i(b +2y))^3(a+ib)^2}da  db.
\end{align*}In the same way, we find that
\begin{align*}
c_{0,2}(y)
=&-\frac{32 y^3}{5\pi}\int_{0}^{\infty}b^5\int_{-\infty}^{\infty} \sum_{k=1}^{\infty} \left|\beta_k^{(j)}\right|^2 \frac{e^{-2\pi i ka}e^{-2\pi kb}}{(a+i(b +2y))^5(a+ib)}da  db.
\end{align*}Now to integrate over $a$ using contour integration technique, we need to consider the real line as boundary of the upper lower plane because of the factor $e^{-2\pi i ka}$. By carefully evaluating the residues over the two poles $-ib$ and $-i(b+2y)$, and summing up the results, we obtain
\begin{align*}
c_0(y)=&\alpha(y)+\beta(y),
\end{align*}
where
\begin{align*}
\alpha(y)=&\frac{8\pi}{15y}\int_{0}^{\infty}b^5  \sum_{k=1}^{\infty} k\left|\beta_k^{(j)}\right|^2  e^{-4\pi kb}  db\\
=&\frac{1}{64\pi^5y}\sum_{k=1}^{\infty}\frac{1}{k^5}\left|\beta_k^{(j)}\right|^2\\
=&\frac{2}{3y}\langle\mu, \mu\rangle_{\text{TZ}, j}^{\text{cusp}},
\end{align*}and
\begin{align*}
\beta(y)=&-\frac{8\pi}{15y}\int_{0}^{\infty}b^5  \sum_{k=1}^{\infty} k\left|\beta_k^{(j)}\right|^2\left(1 + 4\pi ky + 8\pi^2k^2y^2 + 8\pi^3k^3y^3\right)  e^{-4\pi kb-4\pi k y}  db\\=&-\frac{1}{64\pi^5y}\sum_{k=1}^{\infty}\frac{1}{k^5}\left|\beta_k^{(j)}\right|^2\left(1 + 4\pi ky + 8\pi^2k^2y^2 + 8\pi^3k^3y^3\right)  e^{ -4\pi k y}.
\end{align*}From the expressions, we find that $\alpha(y)$ is proportional to $y^{-1}$ and is related to the TZ cusp metric, and $\beta(y)$ is an exponentially decaying term.
\end{proof}

\begin{proof}[Proof of Theorem \ref{thm2_24_1}]
It is sufficient to consider the case where $\nu=\mu$. 

The set of parabolic elements in $\Gamma$ can be divided into $q$ disjoint subsets $\mathscr{P}_1, \ldots, \mathscr{P}_q$,  where $\mathscr{P}_j$ contains those parabolic elements that are conjugate to $\kappa_j^{\ell}$ for some nonzero integer $\ell$. Call the corresponding contribution to $\mathscr{E}_P$ as $\mathscr{E}_{P,j}$. Corresponding to \eqref{eq2_24_3}, it can be written as the sum of $\mathscr{X}_{P,j}$, $\mathscr{Y}_{P,j}$ and $\mathscr{Z}_{P,j}$.

Let us first concentrate on $\mathscr{X}_{P,j}$. 
\begin{equation}\label{eq3_4_9}\begin{split}
\mathscr{X}_{P,j} =&-\frac{n}{\pi^2} \sum_{\ell\neq 0}\sum_{\alpha\in\Gamma_{\kappa_j}\backslash\Gamma} \int_X\int_{\mathbb{U}}   \frac{\mu(z)\overline{\mu( w)}}{(z-\bar{w})^2(\alpha\kappa_j^{\ell}\alpha^{-1} z-\bar{w})^2} \\&\hspace{4cm}\times \frac{(  z-\bar{ z})^{2n-2} }{(\alpha\kappa_j^{\ell}\alpha^{-1}  z-\bar{z} )^{2n-2}}   [(\alpha\kappa_j^{\ell}\alpha^{-1}) '(z) ]^n  d^2w  d^2z\\
=&-\frac{n}{\pi^2} \sum_{\ell\neq 0} \int_{\mathfrak{S}}\int_{\mathbb{U}} \frac{\hat{\mu}_j(z)\overline{\hat{\mu}_j( w)}}{(z-\bar{w})^2(  z+\ell-\bar{w})^2}   \frac{(  z-\bar{ z})^{2n-2} }{(   z+\ell-\bar{z} )^{2n-2}}     d^2w  d^2z.
\end{split}\end{equation}
Let $z=x+iy$ and $w=u+iv$, and let $\hat{\mu}_j$ be given by \eqref{eq2_18_8}. Then
\begin{align*}
\mathscr{X}_{P,j}=&(-1)^n\frac{2^{2n-2}n}{\pi^2}\int_0^{\infty}\int_0^{\infty}\int_0^1 \int_{-\infty}^{\infty}\sum_{\ell\neq 0}v^2y^2\sum_{k=1}^{\infty}\beta_k^{(j)}e^{-2\pi i k(x-iy)}
\sum_{m=1}^{\infty}\overline{\beta_m^{(j)}}e^{2\pi i m(u+iv)}\\& \hspace{3cm}\times\frac{1}{( u-x-\ell-i(y+v))^2(u-x-i(y+v))^2}\frac{y^{2n-2}}{(\ell+2iy)^{2n-2}}du dxdv dy.
\end{align*}For fixed $x$, replace $u$ by $u+x$, interchange the order of integration and integrate over $x$ first, we find that
\begin{align}
\mathscr{X}_{P,j}
=&(-1)^n\frac{2^{2n-2}n}{\pi^2}\int_0^{\infty}\int_0^{\infty}  \int_{-\infty}^{\infty}\sum_{\ell\neq 0} v^2\sum_{k=1}^{\infty}\left|\beta_k^{(j)}\right|^2e^{-2\pi k(y+v)}
 e^{2\pi i ku}\nonumber\\& \hspace{3cm}\times\frac{1}{( u-\ell-i(y+v))^2(u-i(y+v))^2}\frac{y^{2n}}{(\ell+2iy)^{2n -2}}du  dv dy.\label{eq2_24_6}
\end{align}
It is straightforward to  verify that for any complex number $w$,
\begin{align}\label{eq2_24_7}
\frac{1}{w^2( w-\ell )^2 }=\frac{1}{\ell^2}\left\{\frac{1}{w^2}+\frac{1}{((w-\ell)^2}\right\}+\frac{2}{\ell^3}\left\{\frac{1}{w }-\frac{1}{w-\ell}\right\}.
\end{align}
For fixed $\ell\neq 0$, use \eqref{eq2_24_7} to split the integral \eqref{eq2_24_6} into four terms and apply change of variables $u\mapsto u+\ell$ to the term that has $u-\ell$, we find that
$$\int_{-\infty}^{\infty} 
  \frac{e^{2\pi i ku}}{( u-\ell-i(y+v))^2(u-i(y+v))^2}du =\frac{2}{\ell^2}\int_{-\infty}^{\infty} 
  \frac{e^{2\pi i ku}}{ (u-i(y+v))^2}du.$$
Using contour integration technique,  we have
  $$\int_{-\infty}^{\infty} 
  \frac{e^{2\pi i ku}}{ (u-i(y+v))^2}du=-4\pi^2 ke^{-2\pi k (y+v)}.$$
Therefore,
\begin{align}
\mathscr{X}_{P,j}=&(-1)^{n-1}2^{2n+1}n\sum_{\ell\neq 0}\frac{1}{\ell^2}\int_0^{\infty}\int_0^{\infty}   v^2 \sum_{k=1}^{\infty}k\left|\beta_k^{(j)}\right|^2e^{-4\pi k(y+v)}
  \frac{y^{2n}}{(\ell+2iy)^{2n-2}}   dv dy\nonumber\\
=&(-1)^{n-1}\frac{2^{2n-2}n}{4\pi^3}\sum_{\ell\neq 0}\frac{1}{\ell^2}\sum_{k=1}^{\infty}\frac{1}{k^2}\left|\beta_k^{(j)}\right|^2\int_0^{\infty} e^{-4\pi k y}
  \frac{y^{2n}}{(\ell+2iy)^{2n-2}}    dy.\label{eq3_4_7}
\end{align}

For the    term $\mathscr{Y}_{P,j}$, follow the same steps as that for $\mathscr{X}_{P,j}$. Instead of \eqref{eq2_24_7}, we have
\begin{align}\label{eq2_25_5}\frac{1}{w^3(w-\ell)}=\frac{1}{\ell^3}\left\{\frac{1}{w-\ell}-\frac{1}{w}\right\}-\frac{1}{\ell^2}\frac{1}{w^2}-\frac{1}{\ell}\frac{1}{w^3}.\end{align}
After a change of variables $u\mapsto u+\ell$ for the first term, we find that the contribution from the first two terms in \eqref{eq2_25_5} cancel.  After integration with respect to $v$,  we obtain
 \begin{align}
\mathscr{Y}_{P,j}
=& (-1)^{n}\frac{2^{2n-2}(n-1)}{4\pi^3}\sum_{\ell\neq 0}\frac{1}{\ell^2}\sum_{k=1}^{\infty}\frac{1}{k^2}\left|\beta_k^{(j)}\right|^2\int_0^{\infty} e^{-4\pi k y}
  \frac{2iy^{2n+1}}{(\ell+2iy)^{2n-1}}    dy\label{eq2_25_9}\\
&+(-1)^{n-1}\frac{2^{2n-2}(n-1)}{2\pi^2}\sum_{\ell\neq 0}\frac{1}{\ell }\sum_{k=1}^{\infty}\frac{1}{k }\left|\beta_k^{(j)}\right|^2\int_0^{\infty} e^{-4\pi k y}
  \frac{y^{2n+1}}{(\ell+2iy)^{2n-1}}    dy.\nonumber
\end{align}
For the term \eqref{eq2_25_9}, we split out the part $2iy$ and  rewrite it as $2iy+\ell -\ell$. This gives
\begin{align}
\mathscr{Y}_{P,j}
=& (-1)^{n}\frac{2^{2n-2}(n-1)}{4\pi^3}\sum_{\ell\neq 0}\frac{1}{\ell^2}\sum_{k=1}^{\infty}\frac{1}{k^2}\left|\beta_k^{(j)}\right|^2\int_0^{\infty} e^{-4\pi k y}
  \frac{ y^{2n }}{(\ell+2iy)^{2n-2}}    dy \nonumber\\
&+(-1)^{n-1}\frac{2^{2n-2}(n-1)}{4\pi^3}\sum_{\ell\neq 0}\frac{1}{\ell }\sum_{k=1}^{\infty}\frac{1}{k^2}\left|\beta_k^{(j)}\right|^2\int_0^{\infty} e^{-4\pi k y}
  (1+2\pi k y)\frac{y^{2n }}{(\ell+2iy)^{2n-1}}    dy \label{eq2_25_7}.
\end{align}
Using the fact that
\begin{align}\label{eq2_25_8}
(2n-2)\frac{y^{2n-3}}{(\ell+2iy)^{2n-1 }} =\frac{1}{\ell} \frac{d}{dy} \frac{y^{2n-2}}{(\ell  +2iy)^{2n-2}},
\end{align}integration by parts of the second term \eqref{eq2_25_7} in $\mathscr{Y}_{P,j}$ gives
\begin{align*}
&(-1)^{n-1}\frac{2^{2n-2}(n-1)}{4\pi^3}\sum_{\ell\neq 0}\frac{1}{\ell }\sum_{k=1}^{\infty}\frac{1}{k^2}\left|\beta_k^{(j)}\right|^2\int_0^{\infty} e^{-4\pi k y}
  (1+2\pi k y)\frac{y^{2n }}{(\ell+2iy)^{2n-1}}    dy\\
=&(-1)^{n }\frac{2^{2n-2}   }{8\pi^3}\sum_{\ell\neq 0}\frac{1}{\ell^2 }\sum_{k=1}^{\infty}\frac{1}{k^2}\left|\beta_k^{(j)}\right|^2\int_0^{\infty} e^{-4\pi k y}(3+4\pi k y-8\pi^2k^2y^2)
  \frac{y^{2n}}{(\ell+2iy)^{2n-2}}    dy.
\end{align*}Thus,
\begin{align*}
&\mathscr{X}_{P,j}+\mathscr{Y}_{P,j}\\=& (-1)^{n }\frac{2^{2n-2}   }{8\pi^3}\sum_{\ell\neq 0}\frac{1}{\ell^2 }\sum_{k=1}^{\infty}\frac{1}{k^2}\left|\beta_k^{(j)}\right|^2\int_0^{\infty} e^{-4\pi k y}(1+4\pi k y-8\pi^2k^2y^2)
  \frac{y^{2n}}{(\ell+2iy)^{2n-2}}    dy.
\end{align*}
For the term $\mathscr{Z}_{P,j}$, using the same reasoning as that for $\mathscr{X}_{P,j}$, we find that
\begin{align*}
\mathscr{Z}_{P,j}=(-1)^{n-1}\frac{2^{2n-2}(n-1)(2n-1)}{\pi} \sum_{\ell\neq 0}\int_0^{\infty}\int_0^1 \frac{y^{2n-2}}{(\ell+2iy)^{2n}} 2(\Delta_0+2)^{-1}\left(\hat{\mu}_j\bar{\hat{\mu}}_j\right)(z)dxdy.
\end{align*}The Fourier expansion of $ 2(\Delta_0+2)^{-1}\left(\hat{\mu}_j\bar{\hat{\mu}}_j\right)(z)$ is given by \eqref{eq2_24_12}. After integration  over $x$, we find that only the $c_0(y)=\alpha(y)+\beta(y)$ term \eqref{eq2_25_3} and \eqref{eq2_25_2} contribute. Hence,
\begin{align*}
\mathscr{Z}_{P,j}=(-1)^{n-1}\frac{2^{2n-2}(n-1)(2n-1)}{\pi} \sum_{\ell\neq 0}\int_0^{\infty}  \frac{y^{2n-2}}{(\ell+2iy)^{2n}} c_0(y) dy
=\mathfrak{Z}_1+\mathfrak{Z}_2,
\end{align*}where $\mathfrak{Z}_1$ is obtained from $\mathscr{Z}_{P,j}$ by replacing $c_0(y)$ by $\alpha(y)$, and $\mathfrak{Z}_2$ is by $\beta(y)$.
A direct integration by parts $(2n-3)$ times give
\begin{align*}
\mathfrak{Z}_1=&(-1)^{n-1}\frac{2^{2n-1}(n-1)(2n-1)}{3\pi} \langle\mu, \mu\rangle_{\text{TZ}, j}^{\text{cusp}}\sum_{\ell\neq 0}\int_0^{\infty}  \frac{y^{2n-3}}{(\ell+2iy)^{2n}}   dy\\
=&\frac{(-1)^{n-1}}{(2i)^{2n-3}}\frac{2^{2n-1} }{3\pi} \langle\mu, \mu\rangle_{\text{TZ}, j}^{\text{cusp}}\sum_{\ell\neq 0}\int_0^{\infty}  \frac{1}{(\ell+2iy)^{3}}   dy.
\end{align*}Finally, an integration and the summation formula
\begin{align}\label{eq3_4_8}\sum_{\ell\neq 0}\frac{1}{\ell^2}=\frac{\pi^2}{3}\end{align}produce
\begin{align*}
\mathfrak{Z}_1=&\frac{(-1)^{n-1}}{(2i)^{2n-2}}\frac{2^{2n-2} }{3\pi} \langle\mu, \mu\rangle_{\text{TZ}, j}^{\text{cusp}}\sum_{\ell\neq 0}\frac{1}{\ell^2}\\
=&\frac{\pi}{9} \langle\mu, \mu\rangle_{\text{TZ}, j}^{\text{cusp}}.
\end{align*}This is the desired term for the parabolic contribution $\mathscr{E}_{P,j}$. To complete the proof of the theorem, we need to show that the sum of $\mathscr{X}_{P,j}$, $\mathscr{Y}_{P,j}$ and $\mathfrak{Z}_2$ is zero.

For the $\mathfrak{Z}_2$ term, \eqref{eq2_25_2} gives
\begin{align*}
\mathfrak{Z}_2=&(-1)^{n}\frac{2^{2n-2}(n-1)(2n-1)}{64\pi^6} \sum_{\ell\neq 0} \sum_{k=1}^{\infty}\frac{1}{k^5}\left|\beta_k^{(j)}\right|^2\\&\times \int_0^{\infty}  \frac{y^{2n-3}}{(\ell+2iy)^{2n}}  \left(1 + 4\pi ky + 8\pi^2k^2y^2 + 8\pi^3k^3y^3\right)  e^{ -4\pi k y} dy.
\end{align*}Since
\begin{align*}
(2n-1)\frac{y^{2n-2}}{(\ell+2iy)^{2n }} =\frac{1}{\ell} \frac{d}{dy} \frac{y^{2n-1}}{(\ell  +2iy)^{2n-1}},
\end{align*}we find that
\begin{align*}
\mathfrak{Z}_2=&(-1)^{n-1}\frac{2^{2n-2}(n-1) }{64\pi^6} \sum_{\ell\neq 0}\frac{1}{\ell} \sum_{k=1}^{\infty}\frac{1}{k^5}\left|\beta_k^{(j)}\right|^2\\&\times \int_0^{\infty}  \frac{y^{2n-1}}{(\ell+2iy)^{2n-1}}\frac{d}{dy}\left\{y^{-1}  \left(1 + 4\pi ky + 8\pi^2k^2y^2 + 8\pi^3k^3y^3\right)  e^{ -4\pi k y} \right\}dy\\
=&(-1)^{n}\frac{2^{2n-2}(n-1) }{64\pi^6} \sum_{\ell\neq 0}\frac{1}{\ell} \sum_{k=1}^{\infty}\frac{1}{k^5}\left|\beta_k^{(j)}\right|^2\\&\times \int_0^{\infty}  \frac{y^{2n-3}}{(\ell+2iy)^{2n-1}}  \left(1 + 4\pi ky + 8\pi^2k^2y^2 + 16\pi^3k^3y^3+ 32\pi^4k^4y^4\right)  e^{ -4\pi k y}  dy.
\end{align*}Now using \eqref{eq2_25_8}, another integration by parts give
\begin{align*}
\mathfrak{Z}_2=&(-1)^{n-1}\frac{2^{2n-2}  }{8\pi^3} \sum_{\ell\neq 0}\frac{1}{\ell^2} \sum_{k=1}^{\infty}\frac{1}{k^2}\left|\beta_k^{(j)}\right|^2 \int_0^{\infty}  \frac{y^{2n}}{(\ell+2iy)^{2n-2}}  \left(1 + 4\pi ky - 8\pi^2k^2y^2  \right)  e^{ -4\pi k y}  dy.
\end{align*}This proves that
$$\mathscr{X}_{P,j}+\mathscr{Y}_{P,j}+\mathfrak{Z}_2=0,$$ and thus completes the proof that the 
 parabolic contribution $\mathscr{E}_P$ is indeed given by \eqref{main_result_P}.
\end{proof}

Setting $n=1$ in $\mathscr{X}_{P,j}$ \eqref{eq3_4_7}, the integration over $y$ can be performed directly and the summation over $\ell$ is simple. Using \eqref{eq3_4_8},
we have
\begin{equation}\label{eq3_9_19}\begin{split}
\mathscr{X}_{P,j}=&\frac{1}{12\pi}  \sum_{k=1}^{\infty}\frac{1}{k^2}\left|\beta_k^{(j)}\right|^2\int_0^{\infty}y^2e^{-4\pi k y}dy\\
=&  \frac{1}{384\pi^4}\sum_{k=1}^{\infty}\frac{|\beta_k^{(j})|^2}{k^5}\\&=\frac{\pi}{9}\langle\mu, \mu\rangle_{\text{TZ}, j}^{\text{cusp}}.
\end{split}\end{equation}
  From this and the definition of $\mathscr{X}_{P,j}$ given by \eqref{eq3_4_9},
we   obtain an explicit integral formula for the parabolic TZ metric.
\begin{corollary}\label{cor3_3_1}
Let $\mu, \nu\in\Omega_{-1,1}(X)$ and let $\kappa_j$ be an elliptic generator of $\Gamma$. Then 
\begin{align}\label{eq3_3_5}
\langle\mu, \nu\rangle_{\text{TZ},j}^{\text{cusp}}=& -\frac{9}{\pi^3}\sum_{\gamma\in \tilde{\Gamma}[\kappa_j]}\int_X\int_{\mathbb{U}}\frac{\mu(z)\overline{\nu(w)}}{(z-\bar{w})^2(\gamma z-\bar{w})^2}\gamma'(z)d^2wd^2z,
\end{align}where
\begin{align*}
\tilde{\Gamma}[\kappa_j]=\left\{\alpha\kappa_j^{\ell}\alpha^{-1}\,|\,\alpha\in \Gamma_{\kappa_j}\backslash\Gamma, \ell\in\mathbb{Z}, \ell\neq 0\right\}.
\end{align*}
\end{corollary}

\begin{remark}
The formula \eqref{eq3_3_5} can be compared to the following integral formula for the Weil-Petersson metric.
\begin{align*}
\langle \mu, \nu\rangle_{\text{WP}}= \frac{12}{\pi}\int_X\int_{\mathbb{U}}\frac{\mu(z)\overline{\nu(w)}}{(z-\bar{w})^4}d^2wd^2z.
\end{align*}
\end{remark}
\smallskip
\section{Ellitpic Contribution}\label{a4}

In this section, we compute the elliptic contribution $\mathscr{E}_E$ to $\pa_{\mu}\pa_{\bar{\nu}}\log\det N_n$ when the surface $X$ has ramification points. The result is
\begin{theorem}\label{thm2_24_2}Let $n\geq 2$ and $v\geq 1$.
Given $\mu,\nu\in\Omega_{-1,1}(X)$, the elliptic  contribution to $\pa_{\mu}\pa_{\bar{\nu}}\log\det N_n$ is given by
\begin{align}\mathscr{E}_E=& \sum_{j=1}^v\mathfrak{B}(m_j,n)\langle \mu, \nu\rangle_{\text{TZ}, j}^{\text{ell}}.\label{main_result_E}\end{align}where $\mathfrak{B}(m,n)$ is the constant  
\begin{align*}
\mathfrak{B}(m,n)=\frac{m}{4}\left[B_2\left(\left\{\frac{n-1}{m}\right\}\right)-\frac{1}{6m^2}\right]-\frac{n-1}{2}\left[B_1\left(\left\{\frac{n-1}{m}\right\}\right)+\frac{1}{2m}\right].
\end{align*}Here $B_1(x)=x-\di\frac{1}{2}$ and $\di B_2(x)=x^2-x+\frac{1}{6}$ are the first and second Bernoulli polynomials, and $\{x\}$ is the fractional part of $x$. 
\end{theorem}

The Riemann surface has $v$ ramification points corresponding to the $v$ elliptic elements $\tau_1, \ldots, \tau_v$ of orders $m_1, \ldots, m_v$ respectively. For each $1\leq j\leq v$, there is a fractional linear transformation $h_j$ mapping $\mathbb{D}$ biholomorphically onto $\mathbb{U}$, and such that
$$h_j^{-1}\circ \tau_j \circ h_j =R_{\theta_j},$$ and $h_j$ maps $w_j$, the fixed point of $\tau_j$ in $\mathbb{U}$, to the origin in $\mathbb{D}$. Here
$$ R_{\theta}=\begin{pmatrix} e^{i\theta} & 0\\0 & e^{-i\theta}\end{pmatrix},\hspace{1cm}\theta_j=\di \frac{\pi}{m_j}.$$

For a fixed positive integer $m\geq 2$, let $U_m$ be the subgroup of $\text{PSU}\,(1,1)$ generated  by $R_{\frac{\pi}{m}}$. Then the subgroup of $\Gamma$ generated by $\tau_j$ is
$h_jU_mh_j^{-1}$.

Given $\mu\in\Omega_{-1,1}(X)$, $1\leq j\leq v$, let
\begin{align}\label{eq2_6_4}\check{\mu}_j(z)=\mu\circ h_j (z)\frac{\overline{h_j '(z)}}{h_j'(z)}.\end{align}
Then $\check{\mu}_j$ is an automorphic $(-1,1)$-form of the group $h_j^{-1}\Gamma h_j$. Since $U_m$ is a subgroup of  $h_j^{-1}\Gamma h_j$, we find that
\begin{align}\label{eq2_4_2}\check{\mu}_j(e^{2i\theta_j}z)e^{-4i\theta_j}=\check{\mu}_j(z)\hspace{1cm}\text{for all}\;z\in\mathbb{U}.\end{align} By definition of harmonic Beltrami differentials, 
$(1-|z|^2)^2\check{\mu}_j(z)$ is antiholomorphic. Together with  \eqref{eq2_4_2}, we find that $\check{\mu}_j(z)$ can be written as
\begin{align}\label{eq2_26_1}\check{\mu}_j(z)=\frac{(1-|z|^2)^2}{4}\sum_{\substack{k\geq 2\\k\equiv 0\,\text{mod}\,m_j}}(k^3-k) \chi_k^{(j)}\bar{z}^{k-2},\end{align}where only the terms for which $k\equiv 0\;\text{mod}\;m_j$ are present. For simplicity, we will write this as
$$\check{\mu}_j(z)=\frac{(1-|z|^2)^2}{4}\sum_{k=2}^{\infty}(k^3-k) \chi_k^{(j)}\bar{z}^{k-2},$$ understanding that $\chi_k^{(j)}=0$ if $k\not\equiv 0\;\text{mod}\;m_j$.
By definition of the elliptic TZ metric \eqref{eq3_4_2}, we find that 
\begin{align*}
\langle\mu, \mu\rangle_{\text{TZ}, j}^{\text{ell}}=&2\int_{\mathbb{U}}  \mathcal{G}(w_j, z)\mu(z)\overline{\mu(z)} \rho(z)d^2z\\
=&8\int_{\mathbb{D}}\check{\mathcal{G}}(0, z)\check{\mu}(z)\overline{\check{\mu}(z)}\frac{d^2z}{(1-|z|^2)^2},
\end{align*}
 Here $\check{\mathcal{G}}(w, z)$ is given by \eqref{eq2_7_3}.

  Proposition \ref{integral2_24_1} gives the following.

\begin{lemma}\label{lemma2_6_1} If $\check{\mu}$ and $\check{\nu}$ are harmonic Beltrami differentials on $\mathbb{D}$, then
\begin{equation}\label{eq2_6_1}
\begin{split}
2(\Delta_0+2)^{-1}(\check{\mu}\overline{\check{\nu}})(z)=&\int_{\mathbb{D}}\check{\mathcal{G}}(z,w)\check{\mu}(w)\overline{\check{\nu}(w)}\rho(w)d^2w\\
=&\frac{2}{\pi^2}\int_{\mathbb{D}}\int_{\mathbb{D}}\check{\mu}(\zeta)\overline{\check{\nu}(\eta)}\left\{\frac{(1-|z|^2)^2}{(1-\zeta \bar{z})^2(1-\zeta\bar{\eta})^2(1-z\bar{\eta})^2}\right.\\
&\hspace{3.5cm}+\left.\frac{ (1-|z|^2)^3}{(1-\zeta \bar{z})^3(1-\zeta\bar{\eta})(1-z\bar{\eta})^3}\right\}d^2\zeta d^2\eta.
\end{split}
\end{equation}
\end{lemma}
Since $h_j$ maps the elliptic fixed point $w_j$ to the origin, Lemma \ref{lemma2_6_1} can be used to express the elliptic Takhtajan-Zograf metric in terms of Fourier coefficients.
\begin{proposition}
Given $\mu, \nu \in \Omega_{-1,1}(X)$, $1\leq j\leq v$, let $\check{\mu}_j, \check{\nu}_j$ be defined by \eqref{eq2_6_4}. Then 
\begin{align*}
\langle\mu, \nu\rangle_{\text{TZ}, j}^{\text{ell}}
=&\frac{4}{\pi^2}\int_{\mathbb{D}} \int_{\mathbb{D}}\check{\mu}(\zeta)\overline{\check{\nu}(\eta)}\mathscr{H}(\zeta,\eta) d^2\zeta d^2\eta,
\end{align*}where
$$\mathscr{H}(\zeta, \eta)=\frac{1}{(1-\zeta\bar{\eta})^2}+\frac{1}{1-\zeta\bar{\eta}}=\sum_{k=2}^{\infty}k(\zeta\bar{\eta})^{k-2}.$$
If
$$\check{\mu}_j(z)=\mu\circ h_j (z)\frac{\overline{h_j '(z)}}{h_j'(z)}=\frac{(1-|z|^2)^2}{4}\sum_{k=2}^{\infty}(k^3-k) \chi_k^{(j)}\bar{z}^{k-2},$$ then
\begin{align}\label{eq2_26_8}
\langle\mu, \mu\rangle_{\text{TZ}, j}^{\text{ell}}=& \sum_{k=2}^{\infty}k\left| \chi_k^{(j)}\right|^2.
\end{align}
\end{proposition}

If $\gamma$ is an elliptic element of $\Gamma$, there is an $\alpha\in \Gamma$, $1\leq j\leq v$ and  an integer $1\leq\ell\leq m_j-1$ such that
$$\alpha^{-1}\gamma\alpha=\tau_j^{\ell}=h_jR_{l\theta_j}h_j^{-1}.$$
For $1\leq j\leq v$, the stabilizer of $\tau_j$ is 
$$\Gamma_{\tau_j}=h_jU_mh_j^{-1},$$which is precisely the subgroup generated by $\tau_j$.

We can choose a fundamental domain $F$   of $X$ on $\mathbb{U}$ such that
$$\mathscr{D}_j=\bigcup_{\alpha\in \Gamma_{\tau_j}\backslash\Gamma}\left(h_j^{-1}\circ \alpha^{-1}\right) (\overline{F})=\left\{(r\cos\theta, r\sin\theta)\,|\, 0<r<1, 0\leq \theta\leq 2\theta_j\right\}$$is the sector of the unit disc with central angle $2\theta_j$. It is $1/m_j$ of the unit disc.

\begin{remark}\label{remark2_26_1}
Using the notation in Remark \ref{remark2_2_1}, the elliptic contribution $\mathscr{E}_E$ to $\pa_{\mu}\pa_{\bar{\nu}}\log\det N_n$ can be written as
\begin{align*}
\mathscr{E}_E=\int\limits_F \sum_{j=1}^v\sum_{\ell=1}^{m_j-1}\sum_{\alpha\in \Gamma_{\tau_j}\backslash\Gamma}\mathscr{A}(\alpha\tau_j^{\ell}\alpha^{-1} z, z, \mu, \nu)\left[(\alpha\tau_j^{\ell}\alpha^{-1})'(z) \right]^n\rho(z)^{1-n}d^2z
=&\sum_{j=1}^v\mathscr{E}_{E,j}.
\end{align*}By making a change of variables using the map $h_j$, we find that
$$\mathscr{E}_{E,j}=\sum_{\ell=1}^{m_j-1}\int_{\mathscr{D}_j}\check{ \mathscr{A}}_j (e^{2i\ell\theta_j}z, z, \check{\mu}_j, \check{\nu}_j)  \frac{(1-|z|^2)^{2n-2}}{2^{2n-2}}d^2z,$$ 
where
\begin{equation}\label{eq2_5_1}\begin{split}
 \check{\mathscr{A}}_j(z',z, \check{\mu}, \check{\nu})
=&\mathscr{A}(h_j(z'),h_j(z),\mu,\nu)h_j'(z')^n\overline{h_j'(z)}^n.
\end{split}\end{equation}

Given a positive integer $k$, let $e^{2ik\theta_j}\mathscr{D}_j$ be the sector $\mathscr{D}_j$ rotated by an angle of $2k\theta_j$. It is easy to check that for any positive integer $k$ and $\ell$,
$$\int_{e^{2ik\theta_j}\mathscr{D}_j}\check{ \mathscr{A}}_j(e^{2i\ell\theta_j}z, z, \check{\mu}_j, \check{\nu}_j)  \frac{(1-|z|^2)^{2n-2}}{2^{2n-2}}d^2z=\int_{\mathscr{D}_j}\check{ \mathscr{A}}_j (e^{2i\ell\theta_j}z, z, \check{\mu}_j, \check{\nu}_j)  \frac{(1-|z|^2)^{2n-2}}{2^{2n-2}}d^2z.$$This means that instead of integrating over the sector $\mathscr{D}_j$, one can integrate over any sectors of the unit disc with central angle $2\theta_j$. 

Since
$$\bigcup_{k=1}^{m_j} e^{2ik\theta_j}\mathscr{D}_j=\mathbb{D},$$ and integrating over each of the sectors $e^{2ik\theta_j}\mathscr{D}_j$ produce the same result, 
one can replace the integral over $\mathscr{D}_j$ by the average of the integrals over the $m_j$ sectors. This gives
$$\mathscr{E}_{E,j}=\frac{1}{m_j}\sum_{\ell=1}^{m_j-1}\int_{\mathbb{D}}\check{ \mathscr{A}} (e^{2i\ell\theta_j}z, z, \check{\mu}_j, \check{\nu}_j)  \frac{(1-|z|^2)^{2n-2}}{2^{2n-2}}d^2z.$$ 
\end{remark}
Let us now   give a more precise version of Lemma 1 in \cite{TZ_index_3}, a counterpart of Proposition \ref{prop2_24_1} for elliptic generators.

\begin{proposition}\label{prop2_26_1}
Given $\mu\in\Omega_{-1,1}(X)$,  let $\check{\mu}_j$ be defined by \eqref{eq2_6_4}, with Fourier expansion given by \eqref{eq2_26_1}. Then 
$ 2 (\Delta_0+2)^{-1}(\check{\mu}_j\bar{\check{\mu}}_j)(z)$ has a Fourier expansion of the form
\begin{align}\label{eq2_26_2}
2 (\Delta_0+2)^{-1}(\check{\mu}_j\bar{\check{\mu}}_j)(re^{i\theta})=\sum_{k=-\infty}^{\infty}a_k(r)e^{2 i k\theta}.
\end{align}The term  $a_0(r)$ is given by
 \begin{align}\label{eq2_26_3}
 a_0(r)=\frac{(1-r^2)^2}{2}\sum_{k=2}^{\infty}k \left|\chi_k^{(j)}\right|^2\left\{\sum_{s =1}^{k-1} s^2 r^{2s-2}-\frac{k (k-1)^2}{4}r^{2k-2}\right\}.
\end{align}In particular, $a_0(r)$ is a function of $r^2$.
\end{proposition}
\begin{proof}  Since $2 (\Delta_0+2)^{-1}(\check{\mu}_j\bar{\check{\mu}}_j)(re^{i\theta})$  is a smooth function on the disc $\mathbb{D}$, it is periodic of period $2\pi$ in $\theta$. Hence, it has a Fourier expansion of the form \eqref{eq2_26_2}.
To find $a_0(r)$,  we use the formula  of $2 (\Delta_0+2)^{-1}(\check{\mu}_j\bar{\check{\mu}}_j)(z)$ given by \eqref{eq2_6_1}. 
From the expansion
\begin{align*}
\frac{1}{(1-\zeta\bar{z})^2(1-\zeta\bar{\eta})^2(1-z\bar{\eta})^2}=&\sum_{s_1=1}^{\infty}s_1(\zeta\bar{z})^{s_1-1}\sum_{s_2=1}^{\infty}s_2(\zeta\bar{\eta})^{s_2-1}
\sum_{s_3=1}^{\infty}s_3(z\bar{\eta})^{s_3-1},\\
\frac{1}{(1-\zeta\bar{z})^3(1-\zeta\bar{\eta}) (1-z\bar{\eta})^3}=&\frac{1}{4}\sum_{s_1=2}^{\infty}s_1(s_1 -1)(\zeta\bar{z})^{s_1-2}\sum_{s_2=0}^{\infty} (\zeta\bar{\eta})^{s_2 }
\sum_{s_3=2}^{\infty}s_3(s_3-1)(z\bar{\eta})^{s_3-2},
\end{align*}we find that for the $a_0(r)$ term, we only need those terms with $s_1=s_3$.

Substituting the Fourier expansion of $\check{\mu}_j$, eq. \eqref{eq2_6_1} gives
\begin{align*}
&2(\Delta_0+2)^{-1}(\check{\mu}\overline{\check{\mu}})(re^{i\theta}) \\
=&\frac{(1-|z|^2)^2}{8\pi^2}\int_{\mathbb{D}}\int_{\mathbb{D}}(1-|\zeta|^2)^2 (1-|\eta|^2)^2\sum_{k=2}^{\infty}(k^3-k) \chi_k^{(j)}\bar{\zeta}^{k-2} 
\sum_{m=2}^{\infty}(m^3-m) \overline{\chi_m^{(j)}}\bar{\eta}^{m-2}\\
&\times 
\left\{\frac{1}{(1-\zeta \bar{z})^2(1-\zeta\bar{\eta})^2(1-z\bar{\eta})^2}  +\frac{ (1-|z|^2) }{(1-\zeta \bar{z})^3(1-\zeta\bar{\eta})(1-z\bar{\eta})^3}\right\}d^2\zeta d^2\eta.
\end{align*} This shows that for the $a_0(r)$ term, we must have
$$k=s_1+s_2=s_3+s_2=m.$$ 
Let $s=s_1$. Then $s_2=k-s$. We find that
\begin{align*}
a_0(r)=&\frac{(1-r^2)^2}{8\pi^2}\int_{\mathbb{D}}\int_{\mathbb{D}}(1-|\zeta|^2)^2 (1-|\eta|^2)^2\sum_{k=2}^{\infty}(k^3-k)^2 \left|\chi_k^{(j)}\right|^2
\\& \times\left\{\sum_{s =1}^{k-1}s^2(k-s)r^{2s-2} +\frac{(1-r^2)}{4}\sum_{s =2}^{k}s^2(s-1)^2r^{2s-4}\right\}|\zeta|^{2k-4}|\eta|^{2k-4}d^2\zeta d^2\eta\\
=& \frac{(1-r^2)^2}{2}\sum_{k=2}^{\infty} \left|\chi_k^{(j)}\right|^2\left\{\sum_{s =1}^{k-1}s^2(k-s)r^{2s-2}+\frac{(1-r^2)}{4}\sum_{s =2}^{k}s^2(s-1)^2r^{2s-4}\right\}.
\end{align*}Now,
\begin{align*}
(1-r^2)\sum_{s =2}^{k}s^2(s-1)^2r^{2s-4}=&\sum_{s=1}^{k-1}s^2(s+1)^2r^{2s-2}-\sum_{s=1}^k s^2(s-1)^2r^{2s-2}\\=&4\sum_{s=1}^{k-1}s^3r^{2s-2}-k^2(k-1)^2r^{2k-2}.
\end{align*}Thus,
\begin{align*}
\sum_{s =1}^{k-1}s^2(k-s)r^{2s-2} +\frac{(1-r^2)}{4}\sum_{s =2}^{k}s^2(s-1)^2r^{2s-4}=k\sum_{s =1}^{k-1} s^2 r^{2s-2}-\frac{k^2(k-1)^2}{4}r^{2k-2}.
\end{align*}The formula \eqref{eq2_26_3} follows immediately. 
 
\end{proof}
Next, we need another proposition.
\begin{proposition}\label{prop2_27_1}
Let $n$ be a positive integer, $m$ be an integer larger than $1$, and $k$  an integer divisble by $m$.  If $\omega=e^{\frac{2\pi i }{m}}$, and  $\ell$  is an integer such that $1\leq \ell\leq m-1$, then
\begin{gather}
\sum_{s=1}^{k-1}s(k-s)\omega^{\ell(s-1)}= \frac{2k}{(1-\omega^{\ell})^2},\label{eq2_27_3}\\
 \sum_{s=2}^{k }s(s-1)\omega^{\ell (k-s) }= \frac{k(k-1)-k(k+1)\omega^{\ell}}{(1-\omega^{\ell})^2},\label{eq2_27_4}\\
   \sum_{\ell=1}^{m-1}\frac{\omega^{\ell n}}{1-\omega^{\ell}}=m\left[B_1\left(\left\{\frac{n-1}{m}\right\}\right)+\frac{1}{2m}\right], \label{eq2_27_7}\\
  \sum_{\ell=1}^{m-1}\frac{\omega^{\ell n}}{(1-\omega^{\ell})^2}=-\frac{m^2}{2}\left[B_2\left(\left\{\frac{n-1}{m}\right\}\right)-\frac{1}{6m^2}\right]. \label{eq2_27_5}
  \end{gather} Here $\di B_1(x)=x-\frac{1}{2}$ and $B_2(x)=\di x^2-x+\frac{1}{6}$ are respectively the first and second Bernoulli polynomials, and $\{x\}$ is the fractional part of $x$.

\end{proposition}

\begin{proof}It is elementary to verify that if $x\neq 1$,
\begin{align*}
\sum_{s=1}^{k-1}s(k-s)x^{ s-1}=&\frac{(k-1)x^{k+1}-(k+1)x^k+(k+1)x-k+1}{(x-1)^3},\\
 \sum_{s=2}^{k }s(s-1)x^{  k-s }=& \frac{2x^{k+1}-k(k+1)x^2+2(k^2-1)x-k(k-1)}{(x-1)^3}.
\end{align*}Notice that $\omega^{\ell}\neq 1$ since $1\leq \ell\leq m-1$. By setting $x=\omega^{\ell}$ and noting that $\omega^k=1$ since $k$ is divisible by $m$, we have
\begin{align*}
\sum_{s=1}^{k-1}s(k-s)\omega^{\ell( s-1)}=&\frac{2k\omega^{\ell}-2k}{(\omega^{\ell}-1)^3}=\frac{2k}{(1-\omega^{\ell})^2},\\
\sum_{s=2}^{k }s(s-1)\omega^{\ell (k-s) }=&\frac{2k^2\omega^{\ell}-k(k+1)\omega^{2\ell} -k(k-1)}{(\omega^{\ell}-1)^3}= \frac{k(k-1)-k(k+1)\omega^{\ell}}{(1-\omega^{\ell})^2}.
\end{align*}
 To prove \eqref{eq2_27_7}, we use the fact that
\begin{align*}
\sum_{s=1}^{m}sx^{s-1 }=&\frac{mx^{m +1}-(m+1)x^m+1}{(1-x)^2}.
\end{align*}By setting $x=\omega^{\ell}$ and using the fact that $\omega^{\ell m}=1$, we find that
\begin{align*}
\sum_{s=1}^{m }s\omega^{\ell (s-1)}=-\frac{m}{1-\omega^{\ell}}.
\end{align*}
Now multiply by $\omega^{\ell n}$ on both sides and sum over $\ell$ from 1 to $m-1$, we obtain
\begin{align}
\sum_{\ell=1}^{m-1}\frac{\omega^{\ell n}}{1-\omega^{\ell}}=&-\frac{1}{m}\sum_{\ell=1}^{m-1}\sum_{s=1}^{m }s\omega^{\ell (s+n-1)}\nonumber\\
=& -\frac{1}{m}\sum_{\ell=1}^{m }\sum_{s=1}^{m }s\omega^{\ell (s+n-1)}+\frac{1}{m}\sum_{s=1}^{m }s. \label{eq3_4_3}\end{align}
For the first term in \eqref{eq3_4_3}, sum over $\ell$ first  using
\begin{align}\label{eq2_27_8}
\sum_{\ell=1}^{m }\omega^{\ell q}=\begin{cases} 0,\quad & q\not\equiv 0\mod m\\m,\quad & q\equiv 0\mod m\end{cases},
\end{align}  Notice that when $s$ runs from 1 to $m$, there is exactly one $s$ such that $s+n-1$ is divisible by $m$, which is
$s=m-r$, where $r$ is the remainder when $n-1$ is divided by $m$. 
Therefore,
\begin{align*}
\sum_{\ell=1}^{m-1}\frac{\omega^{\ell n}}{1-\omega^{\ell}}=&-(m-r)+\frac{m+1}{2}\\
=&-\frac{m-1}{2}+r.
\end{align*}Since
 $$\left\{\frac{n-1}{m}\right\}= \frac{r}{m},$$ we find that
$$\sum_{\ell=1}^{m-1}\frac{\omega^{\ell n}}{1-\omega^{\ell}}=m\left[\left\{\frac{n-1}{m}\right\}-\frac{1}{2}\right]+\frac{1}{2},$$which is \eqref{eq2_27_7}.

Eq. \eqref{eq2_27_5} has been proved in \cite{TZ_index_3} as part of the proof of Lemma 2. We give a  proof here similar to the proof of \eqref{eq2_27_7}. Setting $k=m$ in \eqref{eq2_27_3}, we find that
\begin{align*}
\frac{1}{(1-\omega^{\ell})^2}=&\frac{1}{2m}\sum_{s=1}^{m}s(m-s)\omega^{\ell( s-1)}.
\end{align*}The summation over $s$ is extended to $s=m$ as this term  is zero.
Now multiply by $\omega^{\ell n}$ on both sides and sum over $\ell$ from 1 to $m-1$, we obtain
\begin{align}
\sum_{\ell=1}^{m-1}\frac{\omega^{\ell n}}{(1-\omega^{\ell})^2}=&\frac{1}{2m}\sum_{\ell=1}^{m-1}\sum_{s=1}^{m}s(m-s)\omega^{\ell( s+n-1)}\nonumber\\
=&\frac{1}{2m}\sum_{\ell=1}^{m }\sum_{s=1}^{m}s(m-s)\omega^{\ell( s+n-1)}-\frac{1}{2m}\sum_{s=1}^{m}s(m-s).\label{eq2_27_6}
\end{align}Using \eqref{eq2_27_8}, we have
$$\sum_{s=1}^{m}\sum_{\ell=1}^{m }s(m-s)\omega^{\ell( s+n-1)}=mr (m-r).$$
The second summation in \eqref{eq2_27_6} is elementary and it is given by 
$$\sum_{s=1}^{m}s(m-s)=\frac{m^3-m}{6}.$$ Hence
\begin{align*}
\sum_{\ell=1}^{m-1}\frac{\omega^{\ell n}}{(1-\omega^{\ell})^2}=&-\frac{1}{2 }\left\{m^2 \left[\left(\frac{r}{m}\right)^2- \frac{r}{m} +\frac{1}{6}\right]-\frac{1}{6}\right\},
\end{align*}which proves \eqref{eq2_27_5}.
\end{proof}

Now we can return to Theorem \ref{thm2_24_2}.
\begin{proof}[Proof of Theorem \ref{thm2_24_2}.] 
It is sufficient to consider the case where $\nu=\mu$. 

As discussed in Remark \ref{remark2_26_1},   $$\mathscr{E}_E=\sum_{j=1}^{v}\mathscr{E}_{E,j},$$ where $\mathscr{E}_{E,j}$ is the sum of all the elliptic elements of $\Gamma$ that are conjugate to $\tau_j^{\ell}$ for some $1\leq \ell\leq m_j-1$.  Corresponding to \eqref{eq2_24_3}, $\mathscr{E}_{E,j}$ can be written as the sum of $\mathscr{X}_{E,j}$, $\mathscr{Y}_{E,j}$ and $\mathscr{Z}_{E,j}$.

To simplify notation, let $\omega=e^{2i \theta_j}=e^{\frac{2\pi i }{m_j}}$. 
By Remark \ref{remark2_26_1} and \eqref{eq2_22_6}, we have
\begin{equation}\label{eq2_26_4}\begin{split}
\mathscr{X}_{E,j} =&-\frac{n}{\pi^2m_j}\sum_{\ell=1}^{m_j-1}\omega^{\ell n} \int_{\mathbb{D}}\int_{\mathbb{D}}  \frac{\check{\mu}_j(z)\overline{\check{\mu}_j( \eta)}}{(1-z\bar{\eta})^2(1-\omega^{\ell}z\bar{\eta})^2}  
\frac{( 1-|z|^2)^{2n-2} }{(1-\omega^{\ell}|z|^2 )^{2n-2}}      d^2\eta  d^2z,\\
\mathscr{Y}_{E,j}=&-\frac{2(n-1)}{\pi^2m_j}\sum_{\ell=1}^{m_j-1}\omega^{\ell n} \int_{\mathbb{D}}\int_{\mathbb{D}}  \frac{\check{\mu}_j(z)\overline{\check{\mu}_j( \eta)}}{(1-z\bar{\eta})^3(1-\omega^{\ell}z\bar{\eta})}  
\frac{( 1-|z|^2)^{2n-1} }{(1-\omega^{\ell}|z|^2 )^{2n-1}}      d^2\eta  d^2z.
\end{split}\end{equation} These integrals can be evaluated using the same method as in the proof of Proposition \ref{prop2_26_1}. Integrating out the $\eta$ variable and the circular part of $z$, and make a change of variables $r^2\mapsto r$, one immediately obtains
\begin{align*}
\mathscr{X}_{E,j} =&-\frac{n}{ 8m_j}\sum_{\ell=1}^{m_j-1}\omega^{\ell n}\sum_{k=2}^{\infty}(k^3-k) \left|\chi_k^{(j)}\right|^2 \int_0^1
\frac{( 1-r)^{2n} }{(1-\omega^{\ell}r )^{2n-2}} \sum_{s=1}^{k-1}s(k-s)\omega^{\ell(s-1)}r^{k-2}   dr,\\
\mathscr{Y}_{E,j} =&-\frac{n-1}{ 8m_j}\sum_{\ell=1}^{m_j-1}\omega^{\ell n}\sum_{k=2}^{\infty}(k^3-k) \left|\chi_k^{(j)}\right|^2 \int_0^1
\frac{( 1-r)^{2n+1} }{(1-\omega^{\ell}r )^{2n-1}} \sum_{s=2}^{k }s(s-1)\omega^{\ell (k-s) }r^{k-2}   dr.
\end{align*}
Notice that since $\chi_k=0$ if $k\not\equiv 0 \;\text{mod}\;m_j$, we only need to consider those terms with $k\equiv 0\;\text{mod}\;m_j$.   Using \eqref{eq2_27_3} and \eqref{eq2_27_4}, we have
\begin{align}
\mathscr{X}_{E,j} =&-\frac{n}{ 4m_j}\sum_{\ell=1}^{m_j-1}\frac{\omega^{\ell n}}{(1-\omega^{\ell})^2}\sum_{k=2}^{\infty}k(k^3-k) \left|\chi_k^{(j)}\right|^2 \int_0^1
\frac{( 1-r)^{2n} }{(1-\omega^{\ell}r )^{2n-2}}  r^{k-2}   dr,\label{eq3_4_5}\\
\mathscr{Y}_{E,j} =&-\frac{n-1}{ 8m_j}\sum_{\ell=1}^{m_j-1}\frac{\omega^{\ell n}}{(1-\omega^{\ell})^2}\sum_{k=2}^{\infty}(k^3-k) \left|\chi_k^{(j)}\right|^2\nonumber\\&\hspace{3cm}\times \int_0^1
\frac{( 1-r)^{2n+1} }{(1-\omega^{\ell}r )^{2n-1}} \left[k(k-1)-k(k+1)\omega^{\ell}\right]r^{k-2}   dr.\nonumber
\end{align}

For the term $\mathscr{Z}_{E,j}$, it is given by 
\begin{align*}
\mathscr{Z}_{E,j}=&-\frac{(n-1)(2n-1)}{\pi m_j}\sum_{\ell=1}^{m_j-1}\omega^{\ell n} \int_{\mathbb{D}} \frac{(1-|z|^2)^{2n-2}}{(1-\omega^{\ell}|z|^2)^{2n}}2(\Delta_0+2)^{-1}(\check{\mu}_j\bar{\check{\mu}}_j)(z)d^2z.
\end{align*} One can refer to \eqref{eq2_7_1} for the change of variables which results in an extra minus sign. Using polar coordinates and subsituting the Fourier expansion of $2(\Delta_0+2)^{-1}(\check{\mu}_j\bar{\check{\mu}}_j)(z)$ \eqref{eq2_26_2}, we find that only the $a_0(r)$ term \eqref{eq2_26_3} would contribute. 
\begin{align*}
\mathscr{Z}_{E,j}=&-\frac{2(n-1)(2n-1)}{  m_j}\sum_{\ell=1}^{m_j-1}\omega^{\ell n} \int_0^1 \frac{(1-r^2)^{2n-2}}{(1-\omega^{\ell}r^2)^{2n}}a_0(r)rdr\\
=&-\frac{ (n-1)(2n-1)}{ 2 m_j}\sum_{\ell=1}^{m_j-1}\omega^{\ell n} \sum_{k=2}^{\infty}k \left|\chi_k^{(j)}\right|^2\int_0^1 \frac{(1-r )^{2n }}{(1-\omega^{\ell}r )^{2n}} \left\{\mathfrak{V}_k(r)+\mathfrak{W}_k(r)\right\}dr\\
=&\mathfrak{Y}_1+\mathfrak{Y}_2,
\end{align*}where $\mathfrak{Y}_1$ and $\mathfrak{Y}_2$ are respectively the part that contains $\mathfrak{V}_k(r)$ and $\mathfrak{W}_k(r)$, with
\begin{align}
\mathfrak{V}_k(r)=&\sum_{s=1}^{k-1}s^2r^{s-1}=\frac{\mathfrak{U}_k(r)}{(1-r)^3},\nonumber\\
\mathfrak{U}_k(r)=& (2k^2-2k-1)r^{k}-k^2r^{k-1}-(k-1)^2r^{k+1}+r+1,\label{eq2_26_10}\\
\mathfrak{W}_k(r)=&- \frac{k (k-1)^2}{4}r^{k-1}.\nonumber
\end{align}

Applying integration by parts twice to   $\mathfrak{Y}_1$  give
\begin{align*}
\mathfrak{Y}_1=&-\frac{ (n-1)(2n-1)}{ 2 m_j}\sum_{\ell=1}^{m_j-1}\omega^{\ell n}   \sum_{k=2}^{\infty}k \left|\chi_k^{(j)}\right|^2 \int_0^1 \frac{(1-r )^{2n }}{(1-\omega^{\ell}r )^{2n}}  \mathfrak{V}_k(r) dr\\
=&A_0- \frac{ (n-1)  }{ 2 m_j}\sum_{\ell=1}^{m_j-1}\frac{\omega^{\ell n}  }{1-\omega^{\ell}}  \sum_{k=2}^{\infty}k \left|\chi_k^{(j)}\right|^2\int_0^1  \frac{(1-r )^{2n-1 }}{(1-\omega^{\ell}r )^{2n-1}}\left[\frac{d}{dr} \frac{\mathfrak{U}_k(r)}{1-r} \right]dr\\
=&A_0+B_0- \frac{ 1 }{ 4 m_j}\sum_{\ell=1}^{m_j-1}\frac{ \omega^{\ell n}}{(1-\omega^{\ell})^2}   \sum_{k=2}^{\infty}k^2(k-1) \left|\chi_k^{(j)}\right|^2\\
&\hspace{3cm}\times \int_0^1    \frac{(1-r )^{2n  }}{(1-\omega^{\ell}r )^{2n-2}} \left[ (k^2-1)r^{k-2}-(k^2-2k)r^{k-3}\right] dr.
\end{align*} 
Here $A_0$  and $B_0$ are the boundary terms at $r=0$ which appear from integration by parts.
\begin{equation}\label{eq2_27_9}\begin{split}
A_0=&- \frac{(n-1)}{2m_j}\sum_{\ell=1}^{m_j-1}\frac{\omega^{\ell n}}{1-\omega^{\ell}}  \sum_{k=2}^{\infty}k \left|\chi_k^{(j)}\right|^2  \mathfrak{U}_k(0)\\=& -\frac{(n-1)}{2m_j}\sum_{\ell=1}^{m_j-1}\frac{\omega^{\ell n}}{1-\omega^{\ell}}  \sum_{k=2}^{\infty}k \left|\chi_k^{(j)}\right|^2 , \\
B_0=&-\frac{1}{4m_j}\sum_{\ell=1}^{m_j-1}\frac{\omega^{\ell n}}{(1-\omega^{\ell})^2}\sum_{k=2}^{\infty}k \left|\chi_k^{(j)}\right|^2 \left(\mathfrak{U}_k(0)+\mathfrak{U}_k'(0)\right)\\=&-\frac{ 1 }{ 2 m_j}\sum_{\ell=1}^{m_j-1}\frac{\omega^{\ell n}}{(1-\omega^{\ell})^2}\sum_{k=2}^{\infty}k \left|\chi_k^{(j)}\right|^2 \left(1-2\delta_{k,2}\right).
\end{split}\end{equation} The term $\delta_{k,2}$ only appears if $\chi_k^{(j)}\neq 0$ when $k=2$. This can only happen if $m_j=2$.

For $\mathfrak{Y}_2$, we only apply integration by parts once. We have
\begin{align*}
\mathfrak{Y}_2=&-\frac{ (n-1)(2n-1)}{ 2 m_j}\sum_{\ell=1}^{m_j-1}\omega^{\ell n} \int_0^1 \frac{(1-r )^{2n }}{(1-\omega^{\ell}r )^{2n}}  \mathfrak{W}_k(r) dr\\
=& \frac{ (n-1) }{ 8 m_j}\sum_{\ell=1}^{m_j-1}\frac{\omega^{\ell n}}{(1-\omega^{\ell})} \sum_{k=2}^{\infty}k^2(k-1)^2 \left|\chi_k^{(j)}\right|^2 \int_0^1   \frac{(1-r )^{2n}}{(1-\omega^{\ell}r )^{2n-1}}\left[(k-1)r^{k-2}-(k+1)r^{k-1}\right]dr.
\end{align*}There is no boundary term that appear in this case. 

Now it is straightforward to find that
\begin{align*}
 \mathscr{Y}_{E,j}+\mathfrak{Y}_2
=& -\frac{ (n-1) }{ 4 m_j}\sum_{\ell=1}^{m_j-1}\frac{\omega^{\ell n}}{(1-\omega^{\ell})^2} \sum_{k=2}^{\infty}k^2(k-1)  \left|\chi_k^{(j)}\right|^2 \int_0^1   \frac{(1-r )^{2n}}{(1-\omega^{\ell}r )^{2n-1}}\\&\hspace{1cm}\times \left[ 2k\left(1-\omega^{\ell} \right)-(k+1)(1-\omega^{\ell}r)  \right]r^{k-2}dr.\end{align*}It follows that
\begin{align*}
  \mathscr{X}_{E,j}+\mathscr{Y}_{E,j}+\mathfrak{Y}_2=&\mathfrak{Y}_3+\mathfrak{Y}_4,
\end{align*}where 
\begin{align*}
\mathfrak{Y}_3=&-\frac{1}{ 4m_j}\sum_{\ell=1}^{m_j-1}\frac{\omega^{\ell n}}{(1-\omega^{\ell})^2}\sum_{k=2}^{\infty}k(k^3-k) \left|\chi_k^{(j)}\right|^2 \int_0^1
\frac{( 1-r)^{2n} }{(1-\omega^{\ell}r )^{2n-2}}  r^{k-2}   dr,\\
\mathfrak{Y}_4=&-\frac{ (n-1) }{ 2 m_j}\sum_{\ell=1}^{m_j-1}\frac{\omega^{\ell n}}{(1-\omega^{\ell}) } \sum_{k=2}^{\infty}k^3(k-1)  \left|\chi_k^{(j)}\right|^2 \int_0^1   \frac{(1-r )^{2n}}{(1-\omega^{\ell}r )^{2n-1}} r^{k-2}dr.
\end{align*}Applying integration by parts to $\mathfrak{Y}_4$, we have
\begin{align*}
\mathfrak{Y}_4=&C_0-\frac{ 1 }{ 4 m_j}\sum_{\ell=1}^{m_j-1}\frac{\omega^{\ell n}}{(1-\omega^{\ell})^2 } \sum_{k=2}^{\infty}k^3(k-1)  \left|\chi_k^{(j)}\right|^2\\
&\hspace{3cm}\times \int_0^1   \frac{(1-r )^{2n}}{(1-\omega^{\ell}r )^{2n-2}}\left((k-2)r^{k-3}-(k+1) r^{k-2}\right)dr,
\end{align*}where $C_0$ is the boundary term which only present in the $k=2$ term. It is given by
\begin{align*}C_0=&-\frac{ 1 }{ 4 m_j}\sum_{\ell=1}^{m_j-1}\frac{\omega^{\ell n}}{(1-\omega^{\ell})^2 } k^3(k-1)  \left|\chi_k^{(j)}\right|^2 \delta_{k,2}\\
=&-\frac{ 2 }{   m_j}\sum_{\ell=1}^{m_j-1}\frac{\omega^{\ell n}}{(1-\omega^{\ell})^2 }   \left|\chi_2^{(j)}\right|^2.
\end{align*}This term cancels with the $\delta_{k,2}$ term in $B_0$ \eqref{eq2_27_9}. 

Collecting all the terms, we find that
\begin{align*}
\mathscr{E}_{E,j}=&\mathscr{X}_{E,j}+\mathscr{Y}_{E,j}+\mathcal{Z}_{E,j}\\
=&\mathfrak{Y}_1+\mathfrak{Y}_3+\mathfrak{Y}_4\\
=&A_0+B_0+C_0.\end{align*}The terms with integrals cancel and we are left with those terms that come  from integration by parts. From \eqref{eq2_27_9} and Proposition \ref{prop2_27_1}, we find that
\begin{align*}
\mathscr{E}_{E,j}=&-\frac{1}{2m_j}\left\{\sum_{\ell=1}^{m_j-1}\frac{\omega^{\ell n}}{(1-\omega^{\ell})^2}+(n-1)\sum_{\ell=1}^{m_j-1}\frac{\omega^{\ell n}}{1-\omega^{\ell}} \right\}\langle \mu, \nu\rangle_{\text{TZ}, j}^{\text{ell}}\\
=&\left\{\frac{m_j}{4}\left[B_2\left(\left\{\frac{n-1}{m_j}\right\}\right)-\frac{1}{6m_j^2}\right]-\frac{n-1}{2}\left[B_1\left(\left\{\frac{n-1}{m_j}\right\}\right)+\frac{1}{2m_j}\right]\right\}
\langle \mu, \nu\rangle_{\text{TZ}, j}^{\text{ell}}.
\end{align*} This proves \eqref{main_result_E}.
\end{proof}

Setting $n=1$ in $\mathscr{X}_{E,j}$ \eqref{eq3_4_5}, we can directly integrate over $r$ to give
\begin{equation*}\begin{split}
\mathscr{X}_{E,j} =&-\frac{1}{ 2m_j}\sum_{\ell=1}^{m_j-1}\frac{\omega^{\ell  }}{(1-\omega^{\ell})^2}\sum_{k=2}^{\infty}k \left|\chi_k^{(j)}\right|^2.
\end{split}\end{equation*} Using \eqref{eq2_27_5} and \eqref{eq2_26_8}, this shows  that when $n=1$, 
\begin{align}\label{eq3_9_20}\mathscr{X}_{E,j}=\frac{m_j^2-1}{24  m_j} \langle \mu, \nu\rangle_{\text{TZ}, j}^{\text{ell}}.\end{align} From the definition of $\mathscr{X}_{E,j}$ \eqref{eq2_26_4},
 we  obtain an explicit integral formula for the elliptic TZ metric, which is an analog of Corollary \ref{cor3_3_1}.
\begin{corollary}\label{cor3_3_2}
Let $\mu, \nu\in\Omega_{-1,1}(X)$ and let $\tau_j$ be an elliptic generator of $\Gamma$ of order $m_j\geq 2$. Then 
\begin{align}\label{eq3_3_6}
\langle\mu, \nu\rangle_{\text{TZ},j}^{\text{ell}}=& -\frac{24m_j}{(m_j^2-1)}\frac{1}{\pi^2}\sum_{\gamma\in \tilde{\Gamma}[\tau_j]}\int_X\int_{\mathbb{U}}\frac{\mu(z)\overline{\nu(w)}}{(z-\bar{w})^2(\gamma z-\bar{w})^2}\gamma'(z)d^2wd^2z,
\end{align}where
\begin{align*}
\tilde{\Gamma}[\tau_j]=\left\{\alpha\tau_j^{\ell}\alpha^{-1}\,|\,\alpha\in \Gamma_{\tau_j}\backslash\Gamma, \ell=1, 2, \ldots, m_j-1\right\}.
\end{align*}
\end{corollary}
 
\smallskip
  \section{The $n=1$ case}\label{n1}
 In this section, we prove the local index theorem for the $n=1$ case.  

Most of the formulas appeared in the $n\geq 2$ cases look much simpler if we set $n=1$. However,  the Poincar$\acute{\text{e}}$ series involved might not be convergent.
This is implied by the result in \cite{Patterson}, which says that for a cofinite Fuchsian group, the Poincar$\acute{\text{e}}$ series
\begin{align}\label{eq3_5_2}
\sum_{\gamma\in\Gamma}\frac{1}{\left(u(\gamma z, w)+1\right)^s} 
\end{align}is convergent when $\text{Re}\,s>1$, divergent when $\text{Re}\,s<1$. Since $u(\gamma z, w)+1 $ is a positive number,  the series \eqref{eq3_5_2} is also divergent when $s=1$.

By \eqref{eq3_4_1}, the dimension of the space of holomorphic one-differentials on $X$ is equal to the genus $g$ of $X$. Hence, if the  Riemann surface $X$ has genus 0, it does not have holomorphic one-differentials and thus $N_1$ is not defined. In this case, we need to give an alternative interpretation of the local index theorem. 

 We also notice that in the $n=1$ case, the metric \eqref{eq3_8_1} defined on the space of integrable holomorphic one-forms $\Omega_1(X)$ does not depend on the hyperbolic metric on $X$. 

In the following, we will need some theories about abelian differentials (meromorphic one-forms) on $X$ and its connection to the Green's function $G_{0,1}(z,w)$. In the liteature \cite{Schiffer, Springer, Fay_Theta, Kra}, this theory is usually developed only for compact Riemann surfaces. Therefore, we will explain everything in detail for general cofinite Riemann surfaces.


\begin{figure}[ht]
\begin{center}
\includegraphics[scale=0.57]{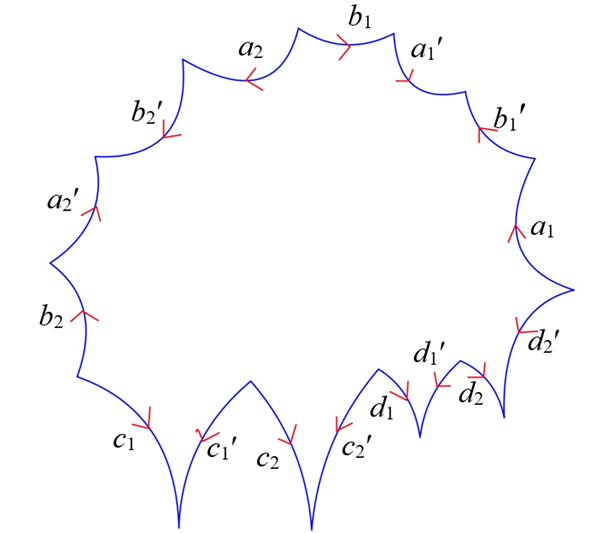}
\caption{\label{fig1} A fundamental domain of $X$.}
\end{center}
\end{figure}


On the Riemann surface $X$ whose Fuchsian group is generated by $2g$  hyperbolic elements $\alpha_1$, $\beta_1$,  $\ldots$, $\alpha_g$, $\beta_g$, $q$ parabolic elements $\kappa_1$, $\ldots$, $\kappa_q$, and $v$ elliptic elements $\tau_1, \ldots, \tau_v$, one can choose a fundamental domain $F$ of $X$ as shown in Figure.\ \ref{fig1}. The boundary of $F$ consists of $4g+2q+2v$ edges $a_1, \ldots, a_g$, $a_1', \ldots, a_g'$, $b_1, \ldots, b_g$, $b_1', \ldots, b_g'$, $c_1, \ldots, c_q$, $c_1', \ldots, c_q'$, $d_1, \ldots, d_v$, $d_1', \ldots, d_v'$, such that
\begin{gather*}  \alpha_j\left(a_j'\right)=a_j,\,\beta_j\left(b_j'\right)=b_j, \quad1\leq j\leq g,\\
 \kappa_j\left(c_j'\right)=c_j,\quad 1\leq j\leq q,\\\tau_j\left(d_j'\right)=d_j,\quad1\leq j\leq v.\end{gather*}

If $g\geq 1$, the projections of $a_1, \ldots, a_g$, $b_1, \ldots, b_g$ to the Riemann surface $X$ then form  a canonical  homology basis of $X$. We will abuse notation and denote this homology basis as $\{a_1, \ldots, a_g, b_1, \ldots, b_g\}$. There is a canocial basis of  holomorphic one-differentials $\{\phi_1, \ldots, \phi_g\}$ which satisfies
\begin{align}\label{eq3_9_1}
\oint_{a_j}\phi_i(z)dz=\delta_{i,j}.
\end{align}With a fixed homology basis, such a basis of $\Omega_1(X)$ is unique since the dimension of $\Omega_1(X)$ is $g$. Our first claim is this canonical basis varies holomorphically with respect to moduli.

\begin{proposition}Let $X$ be a Riemann surface with positive genus.
For each $X^{\mu}\in T(X)$, let $\{\phi_1^{\mu}, \ldots, \phi_g^{\mu}\}$ be the canonical basis of $\Omega_1(X^{\mu})$ with respect to the canonical homology basis $\{a_1^{\mu}, \ldots, a_g^{\mu}, b_1^{\mu}, \ldots, b_g^{\mu}\}$, where $a_j^{\mu}=f^{\mu}(a_j)$ and $b_j^{\mu}=f^{\mu}(b_j)$ for $1\leq j\leq g$. Then $\{\phi_1^{\mu}, \ldots, \phi_g^{\mu}\}$ varies holomorphically with respec to moduli.

\end{proposition}
\begin{proof} It is sufficient to prove that for each $1\leq i\leq g$ and $\mu\in\Omega_{-1,1}(X)$, $L_{\bar{\mu}}\phi_i=0$.
Since
\begin{equation}\label{eq3_7_1}\begin{split}
\delta_{i,j}=&\oint_{a_j^{\varepsilon\mu}}\phi_i^{\varepsilon\mu}(z)dz\\=&\oint_{a_j}\phi_i^{\varepsilon\mu}(f^{\varepsilon\mu}(z))\left(f^{\varepsilon\mu}_z(z)dz+f^{\varepsilon\mu}_{\bar{z}}(z)d\bar{z}\right)\\
=& \oint_{a_j}\phi_i^{\varepsilon\mu}(f^{\varepsilon\mu}(z))f^{\varepsilon\mu}_z(z)\left(dz+\varepsilon\mu(z) d\bar{z}\right),
\end{split}\end{equation}we immediately find that
\begin{align}\label{eq3_8_4}
0=&\left.\frac{\pa}{\pa\bar{\varepsilon}}\right|_{\varepsilon=0} \oint_{a_j}\phi_i^{\varepsilon\mu}(f^{\varepsilon\mu}(z))f^{\varepsilon\mu}_z(z)\left(dz+\varepsilon\mu(z) d\bar{z}\right)
=\oint_{a_j}(L_{\bar{\mu}}\phi_i)(z)dz.
\end{align}On the other hand,
\begin{equation}\label{eq3_10_9}\begin{split}
\frac{\pa}{\pa\bar{z}}(L_{\bar{\mu}}\phi_i)(z)=&\frac{\pa}{\pa\bar{z}}\left.\frac{\pa}{\pa\bar{\varepsilon}}\right|_{\varepsilon=0}\phi_i^{\varepsilon\mu}(f^{\varepsilon\mu}(z))f^{\varepsilon\mu}_z(z)\\
=&\left.\frac{\pa}{\pa\bar{\varepsilon}}\right|_{\varepsilon=0}\left\{(\phi_i^{\varepsilon\mu})_z(f^{\varepsilon\mu}(z))f^{\varepsilon\mu}_{ \bar{z}}(z)f^{\varepsilon\mu}_z(z)+\phi_i^{\varepsilon\mu}(f^{\varepsilon\mu}(z))f^{\varepsilon\mu}_{z\bar{z}}(z)\right\}\\
=&\left.\frac{\pa}{\pa\bar{\varepsilon}}\right|_{\varepsilon=0}\left\{\varepsilon(\phi_i^{\varepsilon\mu})_z(f^{\varepsilon\mu}(z))\mu(z)(f^{\varepsilon\mu}_z(z))^2+\phi_i^{\varepsilon\mu}(f^{\varepsilon\mu}(z))\frac{\pa}{\pa z }(\varepsilon\mu(z)f^{\varepsilon\mu}_{z}(z)) \right\}\\
=&0.
\end{split}\end{equation}Therefore,
$(L_{\bar{\mu}}\phi_i)(z)$ is a holomorphic one-differential. Since all its $a$-periods are zero by \eqref{eq3_8_4}, we find that $(L_{\bar{\mu}}\phi_i)(z)=0$. This shows that the canonical basis of holomorphic one-differentials   varies holomorphically. 
\end{proof}

The period matrix $\tau$ of $X$ is a $g\times g$ matrix defined by
$$\tau_{k,l}=\oint_{b_l}\phi_k(z)dz.$$It is a symmetric matrix with positive definite imaginary parts \cite{Springer, Kra}. 
For the Gram matrix $N_1$ associated to this basis $\{\phi_1, \ldots, \phi_g\}$, Riemann bilinear relation implies that  \cite{Springer, Kra} $N_1$ is the imaginary part of the matrix $\tau$. Namely,
\begin{align*}
N_1 =& \text{Im}\;\tau.
\end{align*}In particular, $N_1$ is real-valued. By Lemma \ref{lemma1_27_1}, we find that the projection kernel $K_1(z,w)$ (also known as Bergman kernel) is given by
\begin{align}\label{eq3_9_8}
K_1(z,w)=\sum_{k=1}^g\sum_{l=1}^g  [(\text{Im}\,\tau)^{-1}]_{k,l}\phi_l(z)\overline{\phi_k(w)}.
\end{align}

For the variation of the components of $N_1$, it can be deduced from the   Rauch's formula \cite{Rauch}. We reproduce the proof as follows. For convenience, we abbreviate $N_1$ as $N$. 
\begin{proposition} 
Let $\{\phi_1, \ldots, \phi_g\}$ be the canonical basis of $\Omega_1(X)$, and let $N$ be the corresponding Gram matrix. Rauch's formula  \cite{Rauch} implies that
\begin{align}\label{eq3_9_4}
L_{\mu}N_{k,l}=&-\int_X \phi_k(z) \phi_l(z)\mu(z)d^2z.
\end{align}
\end{proposition}
\begin{proof}
From \eqref{eq3_7_1}, we find that
\begin{align*}
\oint_{a_j}(L_{\mu}\phi_i)(z)dz+\oint_{a_j}\phi_i(z)\mu(z)d\bar{z}=0.
\end{align*}
This implies that $\left(\mathscr{P}_1L_{\mu}\phi_i\right)(z)$, the projection of $(L_{\mu}\phi_i)(z)$ to $\Omega_1(X)$, has $a_j$ period given by $$c_j=-\di \oint_{a_j}\phi_i(z)\mu(z)d\bar{z}.$$ Therefore, 
\begin{align}\label{eq3_8_6}
\left(\mathscr{P}_1L_{\mu}\phi_i\right)(z)=\sum_{j=1}^g c_j\phi_j(z)=-\sum_{j=1}^g \left(\oint_{a_j}\phi_i(w)\mu(w)d\bar{w}\right)\phi_j(z).
\end{align} 
From this, we have
\begin{equation}\label{eq3_8_7}\begin{split}
L_{\mu}N_{k,l}=&\int_X(\mathscr{P}_1L_{\mu}\phi_k)(z)\overline{\phi_l(z)}d^2z\\
=&-\sum_{j=1}^g \left(\oint_{a_j}\phi_k(w)\mu(w)d\bar{w}\right)\int_X\phi_j(z)\overline{\phi_l(z)}d^2z\\
=&-\sum_{j=1}^g N_{j,l}  \oint_{a_j}\phi_k(w)\mu(w)d\bar{w}.
\end{split}\end{equation}On the other hand, since $\overline{\phi_k(z)\mu(z)}$ is a $(1,0)$-differential, its projection to $\Omega_1(X)$ can be written as a linear combination  of the canonical basis $\phi_1(z), \ldots, \phi_g(z)$, with coefficients determined by its $a$-periods. This gives
\begin{align*}
\left(\mathscr{P}_1 [\overline{\phi_k\mu}]\right)(z)=\left(\oint_{a_j}\overline{\phi_k(w)\mu(w)}dw\right) \phi_j(z).
\end{align*}Therefore,
\begin{equation}\label{eq3_9_2}\begin{split}
\int_X \phi_k(z) \phi_l(z)\mu(z)d^2z=&\int_{X} \phi_l(z)\overline{(\mathscr{P}_1[\overline{\phi_k\mu}])(z)}d^2z\\
=&\sum_{j=1}^g\left(\oint_{a_j} \phi_k(w)\mu(w)d\bar{w}\right) \int_X\overline{\phi_j(z)}\phi_l(z)d^2z \\
=&\sum_{j=1}^g \overline{N_{j,l}}  \oint_{a_j}\phi_k(w)\mu(w)d\bar{w}.
\end{split}\end{equation}Using the fact that $N$ is a real-matrix, comparing \eqref{eq3_9_2} to \eqref{eq3_8_7} prove that
\begin{align*}
L_{\mu}N_{k,l}=&-\int_X \phi_k(z) \phi_l(z)\mu(z)d^2z.
\end{align*}
\end{proof}

Now we can derive the first and second variations of $\log\det N_1$.
\begin{theorem}\label{thm3_10_3}
Given $\mu, \nu\in\Omega_{-1,1}(X)$, 
\begin{align}\label{eq3_9_3}
\pa_{\mu}\log\det N_1=- \sum_{k=1}^g\sum_{l=1}^g(N^{-1})_{k,l}\int_X \phi_k(z) \phi_l(z)\mu(z)d^2z,
\end{align}
\begin{align}\label{eq3_9_14}
\pa_{\mu}\pa_{\bar{\nu}}\log\det N_1=&-\int_X\int_X K_1(z,w)K_1(z,w)\mu(z)\overline{\nu(w)}d^2w d^2 z.
\end{align}
\end{theorem}
\begin{proof}
For the first variation, one obtains from \eqref{v1} and \eqref{eq3_9_4} that 
\begin{align*}
\pa_{\mu}\log\det N_1=\sum_{k=1}^g\sum_{l=1}^g(N^{-1})_{k,l}L_{\mu}N_{l,k}=-\sum_{k=1}^g\sum_{l=1}^g(N^{-1})_{k,l}\int_X \phi_k(z) \phi_l(z)\mu(z)d^2z.
\end{align*}This proves \eqref{eq3_9_3}.

For the second variation, it is more convenient to apply $\pa_{\bar{\nu}}$ to \eqref{eq3_9_3} directly. Since $\phi_j(z)$, $1\leq j\leq g$ varies holomorphically, we find that
\begin{equation}\label{eq3_9_7}\begin{split}
\pa_{\mu}\pa_{\bar{\nu}}\log\det N_1=&- \sum_{k=1}^g\sum_{l=1}^g(L_{\bar{\nu}}N^{-1})_{k,l}\int_X \phi_k(z) \phi_l(z)\mu(z)d^2z\\&-\sum_{k=1}^g\sum_{l=1}^g(N^{-1})_{k,l}\int_X \phi_k(z) \phi_l(z)(L_{\bar{\nu}}\mu)(z)d^2z.
\end{split}\end{equation} Using Proposition \ref{vary_differential}, and the fact that $\phi_k(z)\phi_l(z)$ is holomorphic,  integration by parts shows that the second term in \eqref{eq3_9_7} is zero. Namely,
$$\sum_{k=1}^g\sum_{l=1}^g(N^{-1})_{k,l}\int_X \phi_k(z) \phi_l(z)(L_{\bar{\nu}}\mu)(z)d^2z=0.$$For the first term in \eqref{eq3_9_7}, \eqref{eq3_9_6} gives
\begin{align*}
(L_{\bar{\nu}}N^{-1})_{k,l}=-\sum_{m=1}^g\sum_{q=1}^g (N^{-1})_{k,m}\left(L_{\bar{\nu}}N_{m,q}\right)(N^{-1})_{q,l}.
\end{align*}Using   the complex conjugate of \eqref{eq3_9_4}, we find that
\begin{align*}
&\pa_{\mu}\pa_{\bar{\nu}}\log\det N_1\\=& -\sum_{k=1}^g\sum_{l=1}^g\sum_{m=1}^g\sum_{q=1}^g (N^{-1})_{k,m}
\int_X\overline{\phi_m(w)}\overline{\phi_q(w)}\overline{\nu(w)}d^2w (N^{-1})_{q,l}\int_X \phi_k(z) \phi_l(z)\mu(z)d^2z\\
=&-\int_X\int_X K_1(z,w)K_1(z,w)\mu(z)\overline{\nu(w)}d^2w d^2 z.
\end{align*}
This proves \eqref{eq3_9_14}.
\end{proof}

Notice that if 
\begin{align}\label{eq3_9_15}
K_1(z,w)=-\frac{1}{\pi}\sum_{\gamma\in\Gamma}\frac{\gamma'(z)}{(\gamma z-\bar{w})^2},
\end{align}then the formula \eqref{eq3_9_14} for $\pa_{\mu}\pa_{\bar{\nu}}\log\det N_1$ is equal to the $\mathscr{X}$ term in \eqref{eq2_22_6} when $n=1$. Since the $\mathscr{Y}$ and $\mathscr{Z}$ terms in \eqref{eq2_22_6} vanish if $n=1$, we find that heuristically, \eqref{eq3_9_14} is "proved" by Theorem \ref{thm2_22_3}. 
Unfortunately, since the series \eqref{eq3_5_2} is divergent when $s=1$, the series \eqref{eq3_9_15} does not converge absolutely, and we also do not know whether it converges conditionally.

If $\psi(z)$ is an integrable holomorphic function in $\mathbb{D}$ with Taylor series expansion
$$\psi(z)=\sum_{k=0}^{\infty}c_kz^{k},$$ one can show that
$$\frac{1}{\pi}\int_{\mathbb{D}}\frac{\psi(w)}{(1-z\bar{w})^2}d^2w=\psi(z).$$
Using the map $h$ \eqref{eq2_2_2} that maps $\mathbb{D}$ biholomorphically onto $\mathbb{U}$, we  find that if $\phi$ is holomorphic on $\mathbb{U}$, then
\begin{align}\label{eq3_9_12}
-\frac{1}{\pi}\int_{\mathbb{U}}\frac{\phi(w)}{(z-\bar{w})^2}d^2w=\phi(z).
\end{align}
If  $\phi\in \Omega_{1}(X)$, 
  we  can fold up the integral using the automorphy of $\phi$, which gives
\begin{equation}\label{eq3_9_9}\begin{split}
-\frac{1}{\pi}\int_{\mathbb{U}}\frac{\phi(w)}{(z-\bar{w})^2}d^2w=&-\frac{1}{\pi} \sum_{\gamma\in\Gamma}\int_{\gamma  (F)}\frac{\phi(w)}{(z-\bar{w})^2}d^2w\\
=&-\frac{1}{\pi} \sum_{\gamma\in\Gamma}\int_F \frac{\overline{\gamma'(w)}}{(z-\overline{\gamma w})^2}\phi(w)d^2w\\
=&-\frac{1}{\pi} \sum_{\gamma\in\Gamma}\int_F \frac{\gamma'(z)}{(\gamma z-\overline{  w})^2} \phi(w)d^2w.
\end{split}\end{equation}Hence, if interchanging the summation and integration can be justified, we will conclude that  the series
\eqref{eq3_9_15} has the reproducing property of the Bergman kernel $K_1(z,w)$. Unfortunately, since   the series \eqref{eq3_9_15} does not converge absolutely, it is difficult to justify the interchange of summation and integration. Nevertheless, Fay \cite[Theorem 2.5]{Fay} proved that if $X$ is compact, then for any $\alpha>1$,
\begin{align}\label{eq3_9_5}
K_1(z,w)=-\frac{1}{\pi}\lim_{R\rightarrow \infty}\sum_{\substack{\gamma\in\Gamma\\ u(\gamma z, w)\leq R}}\frac{\gamma'(z)}{(\gamma z-\bar{w})^2}\left(1-\frac{u(\gamma z, w)}{R}\right)^{\alpha}.
\end{align}In other words, the formula \eqref{eq3_9_15} holds after some regularization.  However, we do not want to pursue this line further.
Instead, we turn to the Green's function and the determinant of the Laplacian $\Delta_1$. 

As we discuss in Appendix \ref{resolvent}, when $\text{Re}\,s>1$, the resolvent kernel $G_{0,s}(z,w)$ can be defined by the absolutely convergent series 
\begin{align*}
G_{0,s}(z,w)=\sum_{\gamma\in\Gamma}\mathcal{G}_{0,s}(\gamma z, w),
\end{align*}where
\begin{align}\label{eq3_9_23}
\mathcal{G}_{0,s}( z, w)=\frac{1}{4\pi}\sum_{k=0}^{\infty}\frac{\Gamma(s+k)\Gamma(s+k)}{k!\,\Gamma(2s+k)}\frac{1}{(u(z,w)+1)^{s+k}}.
\end{align}
When $z\neq w$, it admits an a meromorphic continuation in $s$ to the entire complex $s$-plane \cite{Fay, Iwaniec}. When $z\neq w$, one can show that (see Appendix \ref{resolvent}):
\begin{align*}
G_{0,1}(z,w)=\lim_{s\rightarrow 1^+}\left\{G_{0,s}(z,w)-\frac{1}{|X|}\frac{1}{s(s-1)}\right\}.
\end{align*}It follows that 
\begin{align}
\pa_w\pa_zG_{0,1}(z,w)=\lim_{s\rightarrow 1^+}\di\frac{\pa}{\pa w}  \frac{\pa}{\pa z} G_{0,s}(z,w).
\end{align}A direct computation with \eqref{eq3_9_23} shows that
\begin{align}\label{eq3_10_5}
\lim_{s\rightarrow 1^+}\frac{\pa}{\pa w}  \frac{\pa}{\pa z} \mathcal{G}_{0,s}(z,w)=-\frac{1}{4\pi}\frac{1}{(z-w)^2}.
\end{align}
This implies that $\pa_w\pa_zG_{0,1}(z,w)$ is a symmetric bi-differential of the second kind, and as $z\rightarrow w$,
\begin{align}\label{eq3_8_3}
\frac{\pa}{\pa w}\frac{\pa}{\pa z}G_{0,1}(z,w)=-\frac{1}{4\pi}\frac{1}{(z-w)^2}+O(1).
\end{align}Using \eqref{eq3_8_3}, one can show that for any holomorphic one-differential $\phi(z)$,
\begin{align*}
\text{p.v.}\;\int_X\frac{\pa}{\pa w}\frac{\pa}{\pa z}G_{0,1}(z,w)\overline{\phi(w)}d^2w=0.
\end{align*}Here $\text{p.v.}$ stands for principal value of the integral. By the characterization property of the Schiffer kernel $\Omega(z,w)$ as the unique  symmetric bi-differential on $\bar{X}\times\bar{X}$ with a double pole of bi-residue 1 on the diagonal $z=w$, and 
$$\text{p.v.}\;\int_X\Omega(z,w)\overline{\phi(w)}d^2w=0$$ for any $\phi\in\Omega_1(X)$, \eqref{eq3_8_3}  shows that \cite[Theorem 2.3]{Fay}
\begin{align*}
\frac{\pa}{\pa w}\frac{\pa}{\pa z}G_{0,1}(z,w)=-\frac{1}{4\pi}\Omega(z,w).
\end{align*} 

For $C=0, H, P, E$, we can define the identity, hyperbolic, parabolic and elliptic parts of $\pa_z\pa_wG_{0,1}(z,w)$ by
\begin{align}\label{eq3_10_7}
\pa_z\pa_wG^{C}_{0,1}(z,w)=\lim_{s\rightarrow 1^+}\sum_{\substack{\gamma\in \Gamma\\\gamma \;\text{is in the set C}}}\frac{\pa^2}{\pa z\pa w}\mathcal{G}_{0,s}(\gamma z, w)\gamma'(z).
\end{align}
As is mentioned in Remark \ref{remark3_9_2}, if $\gamma_0$ is a representative of a primitive hyperbolic class, $\lambda_0$ is its multiplier, then
\begin{align*}
L_{\mu}\log\left[\prod_{k=0}^{\infty}\left(1-\lambda_0^{-k-1}\right)\right]=&\frac{1}{\pi}\int_X\sum_{\gamma\in\tilde{\Gamma}[\gamma_0]}\frac{\gamma'(z)}{(\gamma z-z)^2}\mu(z)d^2z.
\end{align*}Here $\tilde{\Gamma}[\gamma_0]$ is the set of elements in $\Gamma$ that are conjugate to $\gamma_0^{\ell}$ for some positive integer $\ell$.

Since $Z(s)$ \eqref{eq2_23_4} has a meromorphic continuation to the complex plane with a simple zero at $s=1$ (see e.g.\ \cite{Iwaniec}), as $s\rightarrow 1^+$,
$$\log Z(s)=\log(s-1)+\log Z'(1)+O(s-1).$$It follows that
\begin{align*}
L_{\mu}\log\Delta_1=&L_{\mu}\log Z'(1)\\
=&\lim_{s\rightarrow 1^+}L_{\mu}\log Z(s)\\
=&\lim_{s\rightarrow 1^+}L_{\mu}\log\left[\prod_{[\gamma_0]\in P}\prod_{k=0}^{\infty}\left(1-\lambda_0^{-k-s}\right)\right]\\
=&\lim_{s\rightarrow 1^+}\lim_{R\rightarrow \infty}L_{\mu}\log\left[\prod_{\substack{[\gamma_0]\in P\\\lambda_0\leq R}}\prod_{k=0}^{\infty}\left(1-\lambda_0^{-k-s}\right)\right].
\end{align*}By uniqueness of analytic continuation, one can interchange the two limits. Thus,
\begin{align}
L_{\mu}\log\Delta_1=&\lim_{R\rightarrow \infty}\lim_{s\rightarrow 1^+}L_{\mu}\log\left[\prod_{\substack{[\gamma_0]\in P\\\lambda_0\leq R}}\prod_{k=0}^{\infty}\left(1-\lambda_0^{-k-s}\right)\right]\nonumber\\
=&\lim_{R\rightarrow \infty} L_{\mu}\log\left[\prod_{\substack{[\gamma_0]\in P\\\lambda_0\leq R}}\prod_{k=0}^{\infty}\left(1-\lambda_0^{-k-1}\right)\right]\nonumber\\
=&\frac{1}{\pi}\lim_{R\rightarrow \infty}\int_X\sum_{\substack{[\gamma_0]\in P\\\lambda_0\leq R}}\sum_{\gamma\in\tilde{\Gamma}[\gamma_0]}\frac{\gamma'(z)}{(\gamma z-z)^2}\mu(z)d^2z\label{eq3_10_1}\\
=&\frac{1}{\pi}\lim_{R\rightarrow \infty}\int_X\sum_{\substack{\gamma \in\Gamma\\\gamma\;\text{hyperbolic}\\\lambda(\gamma)\leq R }}\frac{\gamma'(z)}{(\gamma z-z)^2} \mu(z)d^2z\label{eq3_10_2}\\
=&-4\lim_{R\rightarrow \infty}\lim_{s\rightarrow 1^+}\int_X\sum_{\substack{\gamma \in\Gamma\\\gamma\;\text{hyperbolic}\\\lambda(\gamma)\leq R }}\pa_z\pa_w\mathcal{G}_{0,s}(\gamma z, z)\gamma'(z)  \mu(z)d^2z\label{eq3_10_3}\\
=&-4 \lim_{s\rightarrow 1^+}\int_X\sum_{\substack{\gamma \in\Gamma\\\gamma\;\text{hyperbolic}  }}\pa_z\pa_w\mathcal{G}_{0,s}(\gamma z, z) \gamma'(z) \mu(z)d^2z \label{eq3_10_4}\\
=&-4\int_X\pa_z\pa_wG_{0,1}^H(\gamma z, z)  \mu(z)d^2z.\nonumber
\end{align}From \eqref{eq3_10_1} to \eqref{eq3_10_2}, we use the fact that the union of all elements in $\tilde{\Gamma}[\gamma_0]$ when $\gamma_0$ runs through primitive hyperbolic conjugacy classes  is the set of all hyperbolic elements. Eq.\ \eqref{eq3_10_5} is used for the equality between \eqref{eq3_10_2} and \eqref{eq3_10_3}. Using the uniqueness of analytic continuation again, we can interchange the two limit in \eqref{eq3_10_3}. For fixed $s>1$, one can take the $R\rightarrow \infty$ and obtain \eqref{eq3_10_4}.

Let us state this as a theorem.
\begin{theorem}\label{thm3_10_2}
Given $\mu\in\Omega_{-1,1}(X)$, the first variation of $\log\det\Delta_1$ is given by
\begin{align*}
\pa_{\mu}\log\Delta_1=&-4\int_X\pa_z\pa_wG_{0,1}^H(  z, z)  \mu(z)d^2z\\
=&\frac{1}{\pi}\lim_{R\rightarrow \infty}\int_X\sum_{\substack{\gamma \in\Gamma\\\gamma\;\text{hyperbolic}\\\lambda(\gamma)\leq R }}\frac{\gamma'(z)}{(\gamma z-z)^2} \mu(z)d^2z.
\end{align*}Here $G_{0,1}^H(\gamma z, z)$ is defined by \eqref{eq3_10_7} with $C=H$.
\end{theorem}
In other words, Theorem \ref{thm2_23_1} holds when $n=1$ after we "regularize" the summation in the integral by analytic continuation. Using the same approach, we can find the second variation.
\begin{theorem}\label{thm3_10_1}
Given $\mu,\nu\in\Omega_{-1,1}(X)$, the second variation of $\log\det\Delta_1$ is given by
\begin{align*}
\pa_{\mu}\pa_{\bar{\nu}}\log\Delta_1=-\frac{1}{\pi^2}\lim_{R\rightarrow \infty}\int_X\int_{\mathbb{U}}\sum_{\substack{\gamma \in\Gamma\\\gamma\;\text{hyperbolic}\\\lambda(\gamma)\leq R }}\frac{\gamma'(z)}{(\gamma z-\bar{w})^2(z-\bar{w})^2} \mu(z)\overline{\nu(w)}d^2wd^2z.
\end{align*} 
\end{theorem}
\begin{proof}
Again, due to the uniqueness of analytic continuation, we can interchange the partial derivative $\pa_{\bar{\nu}}$ with the $s\rightarrow 1^+$ and $R\rightarrow\infty$ limit and obtain
\begin{align*}
\pa_{\mu}\pa_{\bar{\nu}}\log\Delta_1=&\frac{1}{\pi}\lim_{R\rightarrow \infty}\int_X\sum_{\substack{\gamma \in\Gamma\\\gamma\;\text{hyperbolic}\\\lambda(\gamma)\leq R }}\left\{L_{\bar{\nu}}\frac{\gamma'(z)}{(\gamma z-z)^2}\right\} \mu(z)d^2z\\
&+\frac{1}{\pi}\lim_{R\rightarrow \infty}\int_X\sum_{\substack{\gamma \in\Gamma\\\gamma\;\text{hyperbolic}\\\lambda(\gamma)\leq R }}\frac{\gamma'(z)}{(\gamma z-z)^2} (L_{\bar{\nu}}\mu)(z)d^2z.
\end{align*}
 Now by Proposition \ref{prop_vary}, 
\begin{align*}
L_{\bar{\nu}}\frac{\gamma'(z)}{(\gamma z-z)^2}=-\frac{1}{\pi}\int_{\mathbb{U}}\frac{\gamma'(z)}{(\gamma z-\bar{w})^2(z-\bar{w})^2}\overline{\nu(w)}d^2w.
\end{align*}On the other hand, for any hyperbolic element $\gamma$,
$$\frac{\gamma'(z)}{(\gamma z-z)^2}$$ is holomorphic in $z$.  Proposition \ref{vary_differential} gives $$(L_{\bar{\nu}}\mu)(z)=-4\bar{\pa}\rho(z)^{-1}\bar{\pa}(\Delta_0+2)^{-1}(\mu\bar{\nu})(z).$$ Integration by parts then shows that 
$$\int_X \frac{\gamma'(z)}{(\gamma z-z)^2} (L_{\bar{\nu}}\mu)(z)d^2z=0.$$ The assertion follows.
\end{proof}
 
 To finish the proof of the local index theorem for the $n=1$ case, we use the approach in \cite{TZ_index_2, TZ_index_3}. First, we have the following.
\begin{theorem}
Given $\mu\in\Omega_{-1,1}(X)$, 
\begin{align}\label{eq3_10_6}
\int_X\pa_z\pa_wG_{0,1}^P(  z, z)  \mu(z)d^2z=\int_X\pa_z\pa_wG_{0,1}^E(  z, z)  \mu(z)d^2z=0.
\end{align}Hence, the first variation of $\log\det\Delta_1$ can also be written as
\begin{align}\label{eq3_10_8}
\pa_{\mu}\log\Delta_1=&-4\int_X\left(\pa_z\pa_w G_{0,1}(z, z)-\pa_z\pa_wG_{0,1}^0(z,z)\right)  \mu(z)d^2z.
\end{align}
\end{theorem}
Eq. \eqref{eq3_10_6} is the manifestation of the fact that the parabolic contribution and the elliptic contribution in the resolvent trace formula are moduli independent. The parabolic case in \eqref{eq3_10_6} was proved in Lemma 3 in \cite{TZ_index_2}, and the elliptic case was proved in Theorem 1 in \cite{TZ_index_3}.
\begin{proof}
Let us use the notation in Section \ref{a3} for the parabolic case.
By dividing the parabolic elements into the subsets $\tilde{\Gamma}[\kappa_j]$, $1\leq j\leq q$, we have
\begin{align*}
\int_X\pa_z\pa_wG_{0,1}^P(  z, z)  \mu(z)d^2z=&\sum_{j=1}^{q}\mathscr{P}_j,
\end{align*}where
\begin{equation}\label{eq3_10_13}\begin{split}
\mathscr{P}_j=&\int_{\mathfrak{S}}\sum_{\ell\neq 0}\pa_z\pa_w\mathcal{G}_{0,1}( T_{\ell} z, z)T_{\ell}'(z)  \mu(z)d^2z\\
=&-\frac{1}{4\pi}\sum_{\ell\neq 0}\int_{\mathfrak{S}}\frac{1}{(z+\ell -z)^2}\hat{\mu}_j(z)d^2z\\
=&-\frac{\pi}{12}\int_0^{\infty}\int_0^1 y^2\sum_{k=1}^{\infty}\beta_k^{(j)}e^{-2\pi i k (x-iy)}dxdy.
\end{split}\end{equation}Integrating over $x$ shows that $\mathscr{P}_j=0.$ This proves that the first term in \eqref{eq3_10_6} is zero.

Now we use the notation in Section \ref{a4} for the elliptic case. By dividing the elliptic elements into the subsets $\tilde{\Gamma}[\tau_j]$, $1\leq j\leq v$, 
we have
\begin{align*}
\int_X\pa_z\pa_wG_{0,1}^E(  z, z)  \mu(z)d^2z=&\sum_{j=1}^{v}\mathscr{Q}_j,
\end{align*}where
\begin{equation}\label{eq3_10_14}\begin{split}
\mathscr{Q}_j=&\frac{1}{m_j}\int_{\mathbb{D}}\sum_{\ell=1}^{m_j-1}\pa_z\pa_w\check{\mathcal{G}}_{0,1}( e^{2i\ell\theta_j} z, z)  e^{2i\ell\theta_j}  \check{\mu}_j(z)d^2z\\
=&-\frac{1}{4\pi m_j}\sum_{\ell=1}^{m_j-1}\int_{\mathbb{D}}\frac{e^{2i\ell\theta_j}}{( e^{2i\ell\theta_j} z -z)^2}\check{\mu}_j(z)d^2z\\
=&-\frac{1}{16\pi m_j}\sum_{\ell=1}^{m_j-1}\frac{e^{2i\ell\theta_j}}{(e^{2i\ell\theta_j}-1)^2}\int_{\mathbb{D}}\frac{(1-|z|^2)^2}{z^2}\sum_{k=2}^{\infty}(k^3-k)\chi_k^{(j)}\bar{z}^{k-2}d^2z.
\end{split}\end{equation}By \eqref{eq2_2_3}, 
$\mathscr{Q}_j=0.$ This proves that the second term in \eqref{eq3_10_6} is also zero.

Eq. \eqref{eq3_10_8} follows from the fact that 
\begin{align*}\pa_z\pa_w G_{0,1}(z, z)-\pa_z\pa_wG_{0,1}^0(z,z)=\pa_z\pa_w G_{0,1}^H(z, z)+\pa_z\pa_wG_{0,1}^P(z,z)+\pa_z\pa_wG_{0,1}^E(z,z),
\end{align*}Theorem \ref{thm3_10_2} and eq. \eqref{eq3_10_6}.
\end{proof}

Define the bi-differential $\omega(z,w)$ on $\bar{X}\times \bar{X}$ by
\begin{equation}\label{eq3_10_10}\begin{split}
\omega(z,w)=&\Omega(z,w)+\pi \sum_{k=1}^g\sum_{l=1}^g (N^{-1})_{k,l}\phi_k(z)\phi_l(w)\\=&-4\pi \frac{\pa^2}{\pa z\pa w}G_{0,1}(z,w)+\pi \sum_{k=1}^g\sum_{l=1}^g (N^{-1})_{k,l}\phi_k(z)\phi_l(w).
\end{split}\end{equation}It is a symmetric bi-differential with double pole of bi-residue 1 along the diagonal $z=w$. In the case the genus of the Riemann surface $X$ is zero, the second term in \eqref{eq3_10_10} is absent.
By the result of Theorem 2.3 in \cite{Fay}, $\omega(z,w)$ has zero $a$-periods. 
Takhtajan has provided us the following simpler proof of this result in a private communication.
\begin{proposition}
The bi-differential $\omega(z,w)$ on $\bar{X}\times \bar{X}$ defined by \eqref{eq3_10_10} has zero $a$-periods.
\end{proposition}
\begin{proof}If the genus of the Riemann surface $X$ is zero, there is nothing to prove.
If the genus $g$ is positive, let
\begin{align*}
\mathscr{B}(z,w)=-4\frac{\pa^2}{\pa\bar{w}\pa z}G_{0,1}(z,w).
\end{align*}One can compute directly that
\begin{align*}
\frac{\pa }{ \pa z}\mathcal{G}_{0,1}(z,w)=-\frac{1}{4\pi}\frac{1}{(z-w)}\frac{(w-\bar{w})}{(z-\bar{w})}.
\end{align*}Therefore, when $z\rightarrow w$, 
\begin{align*}
\frac{\pa }{ \pa z}G_{0,1}(z,w)=-\frac{1}{4\pi}\frac{1}{z-w}+O(1).
\end{align*}
For fixed $z$, $\mathscr{B}(z,w)$ is a $(0,1)$-differential in $w$. Hence, its projection $(\overline{\mathscr{P}_{1, w}}\mathscr{B})(z,w) $  to $\Omega_{0,1}(X)=\overline{\Omega_{1,0}(X)}$ can be written as a linear combination of $\{\overline{\phi_1(w)}, \ldots, \overline{\phi_g(w)}\}$, with coefficients depending on $z$. Namely,
\begin{align}\label{eq3_12_1}
(\overline{\mathscr{P}_{1,w}}\mathscr{B})(z,w) =\sum_{l=1}^g c_l(z)\overline{\phi_l(w)}.
\end{align}
Notice that  if $\phi\in\Omega_{1,0}(X)$,
\begin{equation}\label{eq3_12_4}\begin{split}
\int_X\left(\overline{\mathscr{P}_{1,w}}\mathscr{B}\right) (z,w) \phi(w)d^2w=&-4\int_X  \frac{\pa^2}{\pa\bar{w}\pa z}G_{0,1}(z,w) \phi(w)d^2w\\
=&-2i\lim_{\varepsilon\rightarrow 0^+}\oint_{|z-w|=\varepsilon}\frac{\pa }{ \pa z}G_{0,1}(z,w)\phi(w)dw\\
=&\frac{1}{2\pi i} \lim_{\varepsilon\rightarrow 0^+}\oint_{|z-w|=\varepsilon}\frac{ \phi(w) }{ w-z}dw\\=& \phi(z).
\end{split}\end{equation}Multiply $\phi_k(w)$ to both sides of \eqref{eq3_12_1} and integrate over $w$, \eqref{eq3_12_4} gives
\begin{align*}
\sum_{l=1}^g N_{k,l}c_l(z)=\phi_k(z).
\end{align*}Hence, 
\begin{equation}\label{eq3_12_2}\begin{split}
(\overline{\mathscr{P}_{1, w}}\mathscr{B})(z,w) 
=&\sum_{k=1}^{g}\sum_{l=1}^g(N^{-1})_{k,l}\phi_k(z)\overline{\phi_l(w)}.
\end{split}
\end{equation} This shows that $(\overline{\mathscr{P}_{1, w}}\mathscr{B})(z,w) $ is the Bergman kernel $K_1(z,w)$. Now for $1\leq j\leq g$, the $a_j$-period of $\Omega(z,w)$ (as a one-differential in $w$) is given by
\begin{align*}
\oint_{a_j}\omega(z,w)dw=& -4\pi \oint_{a_j}\frac{\pa^2}{\pa w\pa z}G_{0,1}(z,w)dw+\pi \sum_{k=1}^g\sum_{l=1}^g (N^{-1})_{k,l}\phi_k(z)\oint_{a_j}\phi_l(w)dw\\
=&4\pi \oint_{a_j}\frac{\pa^2}{\pa \bar{w}\pa z}G_{0,1}(z,w)d\bar{w}+\pi \sum_{k=1}^g\sum_{l=1}^g (N^{-1})_{k,l}\phi_k(z)\delta_{l,j}\\
=&-\pi \oint_{a_j}(\overline{\mathscr{P}_{1, w}}\mathscr{B})(z,w) d\bar{w}+\pi \sum_{k=1}^g (N^{-1})_{k,j}\phi_k(z)\\
=&-\pi \sum_{k=1}^g\sum_{l=1}^g (N^{-1})_{k,l}\phi_k(z)\overline{\oint_{a_j}\phi_l(w)dw}+\pi \sum_{k=1}^g (N^{-1})_{k,j}\phi_k(z)\\
=&0.
\end{align*}This proves the assertion. 
\end{proof}

Finally, we prove the local index theorem for the $n=1$ case.
\begin{theorem}\label{thm3_10_4}
Given $\mu, \nu\in\Omega_{-1,1}(X)$, we have
\begin{align*}
\pa_{\mu}\pa_{\bar{\nu}}\log\frac{\det\Delta_1}{\det N_1}=&-\mathscr{E}_0-\mathscr{E}_P-\mathscr{E}_E,
\end{align*}where $\mathscr{E}_0$, $\mathscr{E}_P$ and $\mathscr{E}_E$ are  the identity, parabolic and elliptic contributions given respectively by
\begin{align*}
\mathscr{E}_0=&-\frac{1}{12\pi}\langle\mu,\nu\rangle,\\
\mathscr{E}_P=&\frac{\pi}{9}\sum_{j=1}^q\langle\mu,\nu\rangle_{\text{TZ},j}^{\text{cusp}},\\
\mathscr{E}_E=& \sum_{j=1}^v\frac{m_j^2-1}{24m_j}\langle\mu,\nu\rangle_{\text{TZ},j}^{\text{ell}}.
\end{align*}
\end{theorem}
\begin{proof}
By \eqref{eq3_9_3} and \eqref{eq3_10_8}, we find that if $\mu\in\Omega_{-1,1}(X)$, 
\begin{align*}
\pa_{\mu}\log\frac{\det\Delta_1}{\det N_1}=\frac{1}{\pi}\int_X\left. \left(\omega(z,w)-\frac{1}{(z-w)^2}\right)\right|_{w=z}\mu(z)d^2z.
\end{align*}Using the same reasoning as \eqref{eq3_8_4} and \eqref{eq3_10_9}, we find that
$L_{\bar{\nu}}\omega(z,w)$ is a holomorphic one-differential in $z$ and in $w$ with zero $a$-periods. Hence, it must be zero. Therefore,
\begin{align*}
\pa_{\mu}\pa_{\bar{\nu}}\log\frac{\det\Delta_1}{\det N_1}=&-\frac{1}{\pi}\int_X\left.L_{\bar{\nu}}\frac{1}{(z-w)^2} \right|_{w=z} \mu(z)d^2z \\&+\frac{1}{\pi}\int_X\left. \left(\omega(z,w)-\frac{1}{(z-w)^2}\right)\right|_{w=z}(L_{\bar{\nu}}\mu)(z)d^2z\\
=&\mathscr{O}_1+\mathscr{O}_2.
\end{align*}Now by Proposition \ref{prop_vary}, 
\begin{align*}
\mathscr{O}_1=&-\frac{1}{\pi}\int_X\left.L_{\bar{\nu}}\frac{1}{(z-w)^2} \right|_{w=z} \mu(z)d^2z\\=&\frac{1}{\pi^2}\int_X\int_{\mathbb{U}}\frac{\overline{\nu(w)}}{(z-\bar{w})^2(z-\bar{w})^2}d^2wd^2z\\
=&\frac{1}{12\pi}\int_X\mu(z)\overline{\nu(z)}\rho(z)d^2z\\
=&\frac{1}{12\pi}\langle\mu,\nu\rangle_{\text{WP}}.
\end{align*}This is the identity-contribution term $-\mathscr{E}_0$.

For the term $\mathscr{O}_2$, notice that
\begin{align*}
\omega(z,w)-\frac{1}{(z-w)^2}=&-4 \pi \frac{\pa^2}{\pa z\pa w}G_{0,1}^H(z,w)+\pi \sum_{k=1}^g\sum_{l=1}^g (N^{-1})_{k,l}\phi_k(z)\phi_l(w)\\
&-4 \pi \frac{\pa^2}{\pa z\pa w}G_{0,1}^P(z,w)-4 \pi \frac{\pa^2}{\pa z\pa w}G_{0,1}^E(z,w).
\end{align*}
In the proofs of Theorem \ref{thm3_10_3} and Theorem \ref{thm3_10_1}, we have shown that
\begin{align*}
\int_X \frac{\pa^2}{\pa z\pa w}G_{0,1}^H(z,z) (L_{\bar{\nu}}\mu)(z)d^2z=\int_X \sum_{k=1}^g\sum_{l=1}^g (N^{-1})_{k,l}\phi_k(z)\phi_l(z)(L_{\bar{\nu}}\mu)(z)d^2z=0.
\end{align*}
Therefore,
\begin{align*}
\mathscr{O}_2=&-4\int_X \frac{\pa^2}{\pa z\pa w}G_{0,1}^P(z,z)(L_{\bar{\nu}}\mu)(z)d^2z-4\int_X \frac{\pa^2}{\pa z\pa w}G_{0,1}^E(z,z)(L_{\bar{\nu}}\mu)(z)d^2z\\
=&\mathscr{O}_3+\mathscr{O}_4.
\end{align*}

Consider now the parabolic contribution $\mathscr{O}_3$. Eq.\ \eqref{eq3_10_6} implies that
\begin{align*}
\mathscr{O}_3=&-4\int_X \frac{\pa^2}{\pa z\pa w}G_{0,1}^P(z,z)(L_{\bar{\nu}}\mu)(z)d^2z\\=&4\int_X \left(L_{\bar{\nu}}\frac{\pa^2}{\pa z\pa w}G_{0,1}^P(z,z)\right)\mu(z)d^2z\\
=&\frac{1}{\pi^2}\int_X\int_{\mathbb{U}}\sum_{\substack{\gamma\in\Gamma\\\gamma\;\text{is parabolic}}}\frac{\gamma'(z)}{(\gamma z-\bar{w})^2(z-\bar{w})^2}\mu(z)\overline{\nu(w)}d^2wd^2z.
\end{align*}This is just the term $-\mathscr{X}_P$ \eqref{eq2_22_6} when $n=1$. By \eqref{eq3_9_19}, 
\begin{align*}
\mathscr{O}_3= &-\frac{\pi}{9}\sum_{j=1}^q\langle\mu,\nu\rangle_{\text{TZ},j}^{\text{cusp}}=-\mathscr{E}_P.
\end{align*}Similarly,
 \begin{align*}
\mathscr{O}_4=&-4\int_X \frac{\pa^2}{\pa z\pa w}G_{0,1}^E(z,z)(L_{\bar{\nu}}\mu)(z)d^2z \\
=&\frac{1}{\pi^2}\int_X\int_{\mathbb{U}}\sum_{\substack{\gamma\in\Gamma\\\gamma\;\text{is elliptic}}}\frac{\gamma'(z)}{(\gamma z-\bar{w})^2(z-\bar{w})^2}\mu(z)\overline{\nu(w)}d^2wd^2z
\end{align*}  is just the term $-\mathscr{X}_E$ \eqref{eq2_22_6} when $n=1$. By \eqref{eq3_9_20}, 
\begin{align*}
\mathscr{O}_4= &- \sum_{j=1}^v\frac{m_j^2-1}{24m_j}\langle\mu,\nu\rangle_{\text{TZ},j}^{\text{ell}}=-\mathscr{E}_E.
\end{align*}
Thus,  
\begin{align*}
 \pa_{\mu}\pa_{\bar{\nu}}\log\frac{\det\Delta_1}{\det N_1}=&\mathscr{O}_1+\mathscr{O}_2=\mathscr{O}_1+ \mathscr{O}_3+\mathscr{O}_4=-\mathscr{E}_0-\mathscr{E}_P-\mathscr{E}_E
\end{align*}is the sum of the identity, parabolic  and  elliptic contributions. This completes the proof of the theorem.

\end{proof}

\smallskip
\section{ Conclusion}

We first conclude the local index theorem.
Collecting the results from Section \ref{a1}, Section \ref{vary_det}, Section \ref{a3}, Section \ref{a4}, Section \ref{n1}, we have proved the following.

\begin{theorem}[Local Index Theorem]~\\Let $n\geq 1$, and let $X$ be a cofinite hyperbolic Riemann surface of type $(g;q;m_1, m_2, \ldots, m_v)$, uniformized by the Fuchsian group $\Gamma$. Assume that the Teichm\"uller space $T(X)$ has positive dimension, then for any  $\mu, \nu\in\Omega_{-1,1}(X)$,  
\begin{equation}\label{eq2_27_11}\begin{split}
\pa_{\mu}\pa_{\bar{\nu}}\log\frac{\det \Delta_n}{\det N_n}=&\frac{6n^2-6n+1}{12\pi}\langle \mu,\nu\rangle_{\text{WP}}-\frac{\pi}{9}\langle \mu, \nu\rangle_{\text{TZ}}^{\text{cusp}}-\sum_{j=1}^v\mathfrak{B}(m_j,n)\langle \mu, \nu\rangle_{\text{TZ}, j}^{\text{ell}}
\end{split}\end{equation}on $T(X)$. Here  
$\mathfrak{B}(m,n)$ is the constant  
\begin{align*}
\mathfrak{B}(m,n)=\frac{m}{4}\left[B_2\left(\left\{\frac{n-1}{m}\right\}\right)-\frac{1}{6m^2}\right]-\frac{n-1}{2}\left[B_1\left(\left\{\frac{n-1}{m}\right\}\right)+\frac{1}{2m}\right],
\end{align*}where $B_1(x)=x-\di\frac{1}{2}$ and $\di B_2(x)=x^2-x+\frac{1}{6}$ are the first and second Bernoulli polynomials, and $\{x\}$ is the fractional part of $x$. 
\end{theorem}

Divide the group  elements of $\Gamma$ into four disjoint subsets containing respectively the identity, hyperbolic, parabolic and elliptic elements.  When $n\geq 2$, the formula for $\pa_{\mu}\pa_{\bar{\nu}}\log\det N_n$ can be written as a sum of four parts -- the identity, hyperbolic, parabolic and elliptic contributions.
The WP-metric term in \eqref{eq2_27_11} appears as the identity contribution of $\pa_{\mu}\pa_{\bar{\nu}}\log\det N_n$. The TZ-cusp metric and TZ-elliptic metric  terms appear respectively as the parabolic contribution and the elliptic contribution of $\pa_{\mu}\pa_{\bar{\nu}}\log\det N_n$. The term  $\pa_{\mu}\pa_{\bar{\nu}}\log\det \Delta_n$ only depends on the hyperbolic elements, and it is equal to the hyperbolic contribution of $\pa_{\mu}\pa_{\bar{\nu}}\log\det N_n$, thus cancelling each other in the right-hand side of \eqref{eq2_27_11}.

We have used Alhfors variation formulas presented in Proposition \ref{prop_vary} exclusively to calculate variations. Proposition \ref{prop2_2_1} is crucial in our computations as it is used to turn singular integrals into regular ones, thus allowing us to perform explicit calculations using appropriate coordinate systems and integration techniques, inclusive of contour integration techniques. 

As a byproduct of our computations, we presented two new formulas \eqref{eq3_3_5} and \eqref{eq3_3_6} for the parabolic TZ metric and the elliptic TZ metric. These formulas exhibit that the two TZ metrics can be defined using analogous formulas. They also  permit us to find the variations of these metrics using Ahlfors variation formulas. For example, if $\beta, \mu, \nu \in \Omega_{-1,1}(X)$ and $1\leq j\leq q$, Proposition \ref{prop_vary} and \eqref{eq3_3_5} give
\begin{align}\label{eq3_9_22}
\pa_{\beta} \langle\mu, \nu\rangle_{\text{TZ},j}^{\text{cusp}}=&  \frac{9}{\pi^4}\sum_{\gamma\in \tilde{\Gamma}[\kappa_j]}\int_X\int_{\mathbb{U}}\int_{\mathbb{U}} \mu(z)\beta(\zeta)\overline{\nu(w)}\left\{\frac{1}{(z-\zeta)^2(\zeta-\overline{\gamma w})^2(  z-\bar{w})^2}\right.\\&\hspace{5cm}\left.+\frac{1}{(z-\zeta)^2(\zeta-\bar{  w})^2( z-\overline{\gamma w})^2}\right\}\overline{\gamma'(w)}d^2\zeta d^2zd^2w.
\end{align}This formula clearly exhibits a symmetry in the role of $\beta$ and $\mu$, thus proving that
$$\pa_{\beta} \langle\mu, \nu\rangle_{\text{TZ},j}^{\text{cusp}}=\pa_{\mu} \langle\beta, \nu\rangle_{\text{TZ},j}^{\text{cusp}},$$
which means that the parabolic TZ metric is K\"ahler \cite{TZ_index_2}. Comparing the two formulas \eqref{eq3_3_5} and \eqref{eq3_3_6}, we conclude that the same computation shows that the elliptic TZ metric is also K\"ahler \cite{TZ_index_3}.

 \bigskip

\bigskip
\appendix


\section{The Resolvent Kernel}\label{resolvent}
In this section, we review some  facts about the resolvent kernels of the $n$-Laplacians (see for example \cite{Fay, Iwaniec}). We follow the convention in \cite{Iwaniec, Fischer, Teo_Sigma}. The normalization is  different from that of \cite{TZ_index_1, TZ_index_2, TZ_index_3}.

When $n$ is a nonnegative integer, $s$ is a complex number with $\text{Re}\,s>1-n$, the kernel $G_{n,s}(z,w)$ of the resolvent $(\Delta_n+s(s+2n-1))^{-1}$ of the $n$-Laplacian $\Delta_n$ on $X$ can be obtained by method of averaging images \cite{Iwaniec}:
\begin{align}\label{eq0903_3}
G_{n,s}(z,w)=\sum_{\gamma\in\Gamma}\mathcal{G}_{n,s}(\gamma z, w)\gamma'(z)^n, \hspace{1cm} z, w\in\mathbb{U}.
\end{align}Here $\mathcal{G}_{n,s}(z,w)$  is the kernel of the resolvent   $(\Delta_n+s(s+2n-1))^{-1}$  on the upper-half plane $\mathbb{U}$. It can be written as
\begin{equation}\label{eq0903_2}
\mathcal{G}_{n,s}(z,w)= \Psi_{n,s}(u(z,w))H_n(z,w),\end{equation} with  $u(z,w)$ being the point-pair invariant \eqref{eq1_24_1}, and
$$H_n(z,w)=\frac{(-4)^n}{(z-\bar{w})^{2n}}$$is up to a constant, equal to the kernel $\mathcal{K}_n(z,w)$ \eqref{eq2_17_2}.

Notice that $H_n(z,z)=\rho(z)^n$.
Since $H_n(\sigma z, \sigma w)\sigma'(z)^n\overline{\sigma'(w)}^n=H_n(z,w)$ for any $\sigma\in \text{PSL}\,(2,\mathbb{R})$, we find that
\begin{align*}
\mathcal{G}_{n,s}(\sigma z, \sigma w)\sigma'(z)^n\overline{\sigma'(w)}^n=\mathcal{G}_{n,s}(z,w)
\end{align*}for any $\sigma\in \text{PSL}\,(2,\mathbb{R})$.
In other words, $\mathcal{G}_{n,s}(z,w)$ is a point-pair invariant kernel on $\mathbb{H}$.

When $\text{Re}\,s>0$, $\Psi(u)=\Psi_{n,s}(u)$  is determined by the differential equation
 $$(\Delta_n+s(s+2n-1))\mathcal{G}_{n,s}(z,w)=0 \quad  \text{when} \;z\neq w,$$ and the condition that
 when $u\rightarrow 0^+$,
$$\Psi(u)\sim \frac{1}{4\pi} \log\frac{1}{u}.$$   
It is given by the explicit formula
\begin{align}\label{eq2_17_6}
  \Psi_{n,s}(u)=&\frac{1}{4\pi} \sum_{k=0}^{\infty}\frac{\Gamma(s+k)\Gamma(s+2n+k)}{k!\;\Gamma(2s+2n+k)}(u+1)^{-k-s}.  \end{align} One can check that when $n=0$, $s=2$, the kernel of $2(\Delta_0+2)^{-1}$ given by \eqref{eq3_3_1} is indeed equal to $2G_{0,2}(z,w)$, with $\mathcal{G}_{0,2}(z,w)$ given by \eqref{eq0903_2} and \eqref{eq2_17_6}.

When $s=0$, 
 \begin{align}\Psi_{n,0}(u)=\frac{1}{4\pi}\log\frac{1}{u},\label{eq2_24_1}\end{align} so that \begin{align*}
\Delta_n\mathcal{G}_{n,0}(z,w)=&-4\rho(z)^{n-1}\frac{\pa}{\pa z}\rho(z)^{-n}\frac{\pa}{\pa\bar{z}}\mathcal{G}_{n,0}(z,w)\\
=&\frac{1}{\pi}\rho(z)^{n-1}\frac{\pa}{\pa z} \frac{1}{\bar{z}-\bar{w}}\left(\frac{z-\bar{z}}{z-\bar{w}}\right)^{2n-1}\\
=&-\mathcal{K}_n(z,w),
\end{align*}where $\mathcal{K}_n(z,w)$ is the projection kernel \eqref{eq2_17_2}.

Another special case is the kernel $\mathcal{G}_{n-1,1}(z,w)$ when $n\geq 1$. By \eqref{eq2_17_6},
\begin{align*}
\Psi_{n-1,1}(u)
=&\frac{1}{4\pi}\sum_{k=0}^{\infty}\frac{ 1}{  2n+k-1}\frac{1}{(u+1)^{k+1}}.
\end{align*}It follows that\begin{align*}
\frac{\pa \Psi_{n-1,1}}{\pa u}=&-\frac{1}{4\pi}\sum_{k=0}^{\infty}\frac{ k+1}{  2n+k-1}\frac{1}{(u+1)^{k+2}}.
\end{align*}By straightforward computation, one finds that
\begin{align*}
&\rho(z)^{n-1}\frac{\pa}{\pa z}\rho(z)^{1-n}\mathcal{G}_{n-1,1}(z,w)\\
=&\Psi_{n-1,1}'(u)\frac{\pa u}{\pa z}\frac{(-4)^{n-1}}{(z-\bar{w})^{2n-2}}
 -(2n-2)\Psi_{n-1,1}(u)\frac{(-4)^{n-1}}{(z-\bar{w})^{2n-1}}+\frac{2n-2}{z-\bar{z}}\Psi_{n-1,1}(u)\frac{(-4)^{n-1}}{(z-\bar{w})^{2n-2}}\\
=&\frac{1}{4\pi}\frac{(-4)^{n-1}}{(z-\bar{w})^{2n-2}}\sum_{k=0}^{\infty}\frac{1}{ 2n+k-1}\frac{1}{(u+1)^{k+1}}
\left\{ -\frac{(k+1)}{(u+1)}\frac{\pa u}{\pa z}-\frac{2n-2}{z-\bar{w}}+\frac{2n-2}{z-\bar{z}}\right\}\\
=&\frac{1}{4\pi}\frac{(-4)^{n-1}}{(z-\bar{w})^{2n-2}}\sum_{k=0}^{\infty} \frac{1}{(u+1)^{k+1}}\frac{(\bar{z}-\bar{w})}{(z-\bar{w})(z-\bar{z})}\\
=&-\frac{1}{4\pi}\frac{(-4)^{n-1}}{(z-\bar{w})^{2n-1}}\frac{w-\bar{w}}{z-w}.
\end{align*}
It is straightforward to show that this is equal to $\di-\rho(w)^{-1}\frac{\pa}{\pa w}\mathcal{G}_{n, 0}(z,w)$. Namely,
\begin{align*}
-\rho(w)^{-1}\frac{\pa}{\pa w}\mathcal{G}_{n, 0}(z,w)
=&-\frac{1}{4\pi}\frac{w-\bar{w}}{z-w} \frac{ (-4)^{n-1}}{  (z-\bar{w})^{2n-1}}
= \rho(z)^{n-1}\frac{\pa}{\pa z}\rho(z)^{1-n}\mathcal{G}_{n-1,1}(z,w).\end{align*}
It follows that
\begin{equation}\label{eq2_17_7}\begin{split}
-\frac{\pa}{\pa w}\rho(w)^{-1}\frac{\pa}{\pa w}\mathcal{G}_{n, 0}(z,w)
=&-\frac{1}{4\pi}\frac{1}{(z-w)^2} \frac{ (-4)^{n-1}}{  (z-\bar{w})^{2n-2}}\\=&\frac{\pa}{\pa w} \rho(z)^{n-1}\frac{\pa}{\pa z}\rho(z)^{1-n}\mathcal{G}_{n-1,1}(z,w).
\end{split}\end{equation}
When $n=1$, \eqref{eq2_17_7} gives 
\begin{align*}
\frac{\pa}{\pa w}\rho(w)^{-1}\frac{\pa}{\pa w}\mathcal{G}_{1, 0}(z,w)=-\frac{\pa}{\pa w}  \frac{\pa}{\pa z} \mathcal{G}_{0,1}(z,w)
=&\frac{1}{4\pi}\frac{1}{(z-w)^2}.
\end{align*}In this case, one cannot use the method of averaging images \eqref{eq0903_3} to construct the kernel $\di-\pa_w\rho(w)^{-1}\pa_wG_{1, 0}(z,w)$ or 
$\di\pa_w\pa_zG_{0,1}(z,w)$ on $X$, since the Poincar$\acute{\text{e}}$ series
does not converge absolutely \cite{Patterson}.

Recall that $\Delta_0$ is a positive semi-definite operator. Therefore,  $(\Delta_0+s(s-1))$ is invertible when $\text{Re}\,s>1$. If $f$ is a square-integrable function on $X$, and $$\left(\Delta_0+s(s-1)\right)f(z)=g(z),$$ then the kernel $G_{0,s}(z,w)$ is defined so that
\begin{align*}
\int_X G_{0,s}(z,w)g(w)\rho(w)d^2w=f(z).
\end{align*} 
If $f(z)$ is the constant function that is equal to 1, $g(z)=s(s-1)$. This implies that
\begin{align}
\int_XG_{0,s}(z,w)\rho(w)d^2w=\frac{1}{s(s-1)}.
\end{align}

The null-space $W$ of the Laplacian operator $\Delta_0$ is the one-dimensional space consists of constant functons. For any square-integrable function $f:X\rightarrow\mathbb{C}$, its projection to the null-space $W$ is the constant 
$$f_0=\frac{\langle f, 1\rangle}{\langle 1, 1\rangle}  =\frac{1}{|X|}\int_X f(z)\rho(z)d^2z.$$
Hence, the projection of $f$ onto the orthogonal complement $W^{\perp}$ is $f-f_0$. 
If
$\Delta_0f =g$, then  the operator $\Delta_0^{-1}$ must map $g$ to the projection of $f$ to $W^{\perp}$, the orthogonal complement of $W$. In other words, we must have
\begin{align*}
\int_X G_{0,1}(z,w)g(w)d^2 w=f(z)-f_0.
\end{align*}

Given $f:X\rightarrow \mathbb{C}$ a square-integrable functon, let
$g_s:X\rightarrow\mathbb{C}$ be the function
$$g_s(z)=(\Delta_0+s(s-1))f(z).$$ Then
$$g_s(z)-g_{s'}(z)=\left(s(s-1)-s'(s'-1)\right)f(z).$$
From this, we see that $g_s(z)$ is an analytic function of $s$.  
Moreover, 
\begin{align*}
\int_X g_s(z)\rho(z)d^2z=&\int_X \left(-4\rho(z)^{-1}\frac{\pa^2}{\pa z\pa\bar{z}}f(z)+s(s-1)f(z)\right)\rho(z)d^2z\\
=&s(s-1)\langle f, 1\rangle.
\end{align*}Therefore,
\begin{align*}
& \int_X \left(G_{0,s}(z,w)-\frac{1}{|X|}\frac{1}{s(s-1)}\right)g_s(w)\rho(w)d^2w\\=& f(z)-\frac{1}{|X|}\frac{1}{s(s-1)}\int_Xg_s(w)\rho(w)d^2w\\
=&f(z)-\frac{\langle f, 1\rangle}{\langle 1, 1\rangle}\\
=&f(z)-f_0
\end{align*}

Taking $s\rightarrow 1^+$, we find that
\begin{align*}
 \int_X \lim_{s\rightarrow 1^+}\left(G_{0,s}(z,w)-\frac{1}{|X|}\frac{1}{s(s-1)}\right)g_1(w)\rho(w)d^2w=f(z)-f_0.
\end{align*}By the analyticity of $g_s(w)$ as well as the characterization property of $G_{0,1}(z,w)$, we find that when $z\neq w$ on $X$, then 
\begin{align*}
G_{0,s}(z,w)=\frac{1}{|X|}\frac{1}{s(s-1)}+G_{0,1}(z,w)+O(s-1), \hspace{1cm}\text{as} \; s\rightarrow 1^+.
\end{align*}Equivalently, when $z\neq w$,
\begin{align*}
G_{0,1}(z,w)=\lim_{s\rightarrow 1^+}\left\{G_{0,s}(z,w)-\frac{1}{|X|}\frac{1}{s(s-1)}\right\}.
\end{align*}It follows that when $z\neq w$, the limit
\begin{align*}
\lim_{s\rightarrow 1^+}\di\frac{\pa}{\pa w}  \frac{\pa}{\pa z} G_{0,s}(z,w)
\end{align*}exists and is equal to $\di\pa_w\pa_zG_{0,1}(z,w)$.  
This also implies that $\pa_w\pa_z G_{0,1}(z,w)$ is symmetric in $z$ and $w$, and it has a double pole along the diagonal of $X\times X$. Moreover, since
\begin{align*}
-4\rho(z)^{-1}\frac{\pa}{\pa \bar{z}}\left(\frac{\pa}{\pa w}  \frac{\pa}{\pa z} G_{0,1}(z,w)\right)=&\lim_{s\rightarrow 1^+}-4\rho(z)^{-1}\frac{\pa}{\pa \bar{z}}\left(\frac{\pa}{\pa w}  \frac{\pa}{\pa z} G_{0,s}(z,w)\right)\\
=&\lim_{s\rightarrow 1^+}\frac{\pa}{\pa w}\left(\Delta_{0, z}G_{0,s}(z,w)\right)\\
=&-\lim_{s\rightarrow 1^+}s(s-1)\frac{\pa}{\pa w}G_{0,s}(z,w)\\
=&0
\end{align*}
 when $z\neq w$, we find that $\pa_z\pa_wG(z,w)$ is holomorphic in $z$ and $w$ away from the diagonal.
\smallskip
\section{An Interesting Formula}\label{interesting_formula}
In this section, we prove  an alternative expression for $\pa_{\mu}\pa_{\bar{\nu}}\log\det N_n $ that might be interesting of its own right.
\begin{theorem}Let $n\geq 1$. 
Given $\mu, \nu\in \Omega_{-1,1}(X)$,  $\pa_{\mu}\pa_{\bar{\nu}}\log\det N_n $ can also be expressed as
 
\begin{align}\label{eq1_29_10}
\pa_{\mu}\pa_{\bar{\nu}}\log\det N_n 
=&-\int_X\int_X(L_{\mu}L_{\bar{\nu}}K_n)(z,w)K_n(w,z)\rho(w)^{1-n}\rho(z)^{1-n}d^2wd^2z.
\end{align}
\end{theorem}
\begin{proof}
Apply $\di\frac{\pa^2}{\pa\varepsilon\pa\bar{\varepsilon}}$ to \eqref{eq1_28_1} and set $\varepsilon=0$, we have
\begin{align*}
 &(L_{\mu}L_{\bar{\mu}}K_n)(z,w)\\=&\int_{X}(L_{\mu}L_{\bar{\mu}}K_n)(z,\zeta)K_n(\zeta,w)\rho(\zeta)^{1-n}d^2\zeta+\int_{X}K_n(z,\zeta)(L_{\mu}L_{\bar{\mu}}K_n)(\zeta,w)\rho(\zeta)^{1-n}d^2\zeta\\
&+\int_{X}(L_{\mu}K_n)(z,\zeta)(L_{\bar{\mu}}K_n)(\zeta,w)\rho(\zeta)^{1-n}d^2\zeta+\int_{X}(L_{\bar{\mu}}K_n)(z,\zeta)(L_{\mu}K_n)(\zeta,w)\rho(\zeta)^{1-n}d^2\zeta\\
&+(1-n)\int_{X}K_n(z,\zeta)K_n(\zeta,w)\rho(\zeta)^{1-n}\int_{X}G(\zeta,\eta)|\mu(\eta)|^2\rho(\eta)d^2\eta d^2\zeta\\&
-\int_{X}K_n(z,\zeta)K_n(\zeta,w)\rho(\zeta)^{1-n}|\mu(\zeta)|^2d^2\zeta.
\end{align*}Now multiplying by $K_n(w,z)\rho(z)^{1-z}\rho(w)^{1-n}$ and integrate over $w$ and $z$, we have
\begin{align*}
 &\int_{X}\int_{X}(L_{\mu}L_{\bar{\mu}}K_n)(z,w)K_n(w,z)\rho(w)^{1-n}\rho(z)^{1-n}d^2wd^2z\\=&\int_{X}\int_{X}(L_{\mu}L_{\bar{\mu}}K_n)(z,\zeta)K_n(\zeta,z)\rho(\zeta)^{1-n} \rho(z)^{1-n}d^2\zeta d^2z\\&+\int_{X}\int_{X}K_n(w,\zeta)(L_{\mu}L_{\bar{\mu}}K_n)(\zeta,w)\rho(\zeta)^{1-n}\rho(w)^{1-n}d^2wd^2\zeta\\
&+\int_{X}\int_{X}\int_{X}(L_{\mu}K_n)(z,\zeta)(L_{\bar{\mu}}K_n)(\zeta,w)K_n(w,z)\rho(z)^{1-z}\rho(w)^{1-n}\rho(\zeta)^{1-n}d^2\zeta d^2z d^2w\\&+\int_{X}\int_{X}\int_{X}(L_{\bar{\mu}}K_n)(z,\zeta)(L_{\mu}K_n)(\zeta,w)K_n(w,z)\rho(z)^{1-z}\rho(w)^{1-n}\rho(\zeta)^{1-n}d^2\zeta d^2z d^2w\\
&+(1-n)\int_{X} K_n(\zeta,\zeta)\rho(\zeta)^{1-n}\int_{X}G(\zeta,\eta)|\mu(\eta)|^2\rho(\eta)d^2\eta d^2\zeta
-\int_{X}K_n( \zeta, \zeta)\rho(\zeta)^{1-n}|\mu(\zeta)|^2d^2\zeta.
\end{align*}By \eqref{eq1_29_1},
\begin{align*}
\int_{X}\int_{X}\int_{X}(L_{\mu}K_n)(z,\zeta)(L_{\bar{\mu}}K_n)(\zeta,w)K_n(w,z)\rho(z)^{1-z}\rho(w)^{1-n}\rho(\zeta)^{1-n}d^2\zeta d^2z d^2w=0.
\end{align*}
By \eqref{eq1_29_2},
\begin{align*}
&\int_{X}\int_{X}\int_{X}(L_{\bar{\mu}}K_n)(z,\zeta)(L_{\mu}K_n)(\zeta,w)K_n(w,z)\rho(z)^{1-z}\rho(w)^{1-n}\rho(\zeta)^{1-n}d^2\zeta d^2z d^2w\\
=&\int_{X}\int_{X} (L_{\bar{\mu}}K_n)(z,\zeta)(L_{\mu}K_n)(\zeta,z) \rho(z)^{1-z} \rho(\zeta)^{1-n}d^2\zeta d^2z.
\end{align*}
 Therefore,
\begin{align*}
&\int_X\int_X(L_{\mu}K_n)(z,w)(L_{\bar{\nu}}K_n)(w,z)\rho(w)^{1-n}\rho(z)^{1-n}d^2wd^2z\\
&+(1-n)\int_{X} K_n(z, z)\rho(z)^{1-n}\int_{X}G(z,w) \mu(w)\overline{\nu(w)}\rho(w)d^2w d^2z
\\&-\int_{X}K_n( z, z)\rho(z)^{1-n}\mu(z)\overline{\nu(z)}d^2z\\=&-\int_{X}\int_{X}(L_{\mu}L_{\bar{\mu}}K_n)(z,w)K_n(w,z)\rho(w)^{1-n}\rho(z)^{1-n}d^2wd^2z.
\end{align*}By Theorem \ref{thm1_29_1}, this proves \eqref{eq1_29_10}.

\end{proof}

\bibliographystyle{amsalpha}
\bibliography{ref}

\providecommand{\bysame}{\leavevmode\hbox to3em{\hrulefill}\thinspace}
\providecommand{\MR}{\relax\ifhmode\unskip\space\fi MR }
\providecommand{\MRhref}[2]{%
  \href{http://www.ams.org/mathscinet-getitem?mr=#1}{#2}
}
\providecommand{\href}[2]{#2}
\begin{thebibliography}{VKF73}

\bibitem[Ahl61]{Ahlfors_remarks}
Lars~V. Ahlfors, \emph{Some remarks on {T}eichm\"{u}ller's space of {R}iemann
  surfaces}, Ann. of Math. (2) \textbf{74} (1961), 171--191. \MR{204641}

\bibitem[Ahl62]{Ahlfors_curvature}
\bysame, \emph{Curvature properties of {T}eichm\"{u}ller's space}, J. Analyse
  Math. \textbf{9} (1961/62), 161--176. \MR{136730}

\bibitem[Ber66]{Bers_integral}
Lipman Bers, \emph{A non-standard integral equation with applications to
  quasiconformal mappings}, Acta Math. \textbf{116} (1966), 113--134.
  \MR{192046}

\bibitem[BK86]{BK}
A.~A. Belavin and V.~G. Knizhnik, \emph{Complex geometry and the theory of
  quantum strings}, Zh. \`Eksper. Teoret. Fiz. \textbf{91} (1986), no.~2,
  364--390. \MR{888370}

\bibitem[Efr88]{Efrat}
Isaac Efrat, \emph{Determinants of {L}aplacians on surfaces of finite volume},
  Comm. Math. Phys. \textbf{119} (1988), no.~3, 443--451. \MR{969211}

\bibitem[Fay73]{Fay_Theta}
John~D. Fay, \emph{Theta functions on {R}iemann surfaces}, Lecture Notes in
  Mathematics, Vol. 352, Springer-Verlag, Berlin-New York, 1973. \MR{0335789}

\bibitem[Fay77]{Fay}
\bysame, \emph{Fourier coefficients of the resolvent for a {F}uchsian group},
  J. Reine Angew. Math. \textbf{293(294)} (1977), 143--203. \MR{506038}

\bibitem[Fis87]{Fischer}
J\"{u}rgen Fischer, \emph{An approach to the {S}elberg trace formula via the
  {S}elberg zeta-function}, Lecture Notes in Mathematics, vol. 1253,
  Springer-Verlag, Berlin, 1987. \MR{892317}

\bibitem[FK92]{Kra}
H.~M. Farkas and I.~Kra, \emph{Riemann surfaces}, second ed., Graduate Texts in
  Mathematics, vol.~71, Springer-Verlag, New York, 1992. \MR{1139765}

\bibitem[Gar75]{Gardiner}
Frederick~P. Gardiner, \emph{Schiffer's interior variation and quasiconformal
  mapping}, Duke Math. J. \textbf{42} (1975), 371--380. \MR{382637}

\bibitem[Hej83]{Hejhal_2}
Dennis~A. Hejhal, \emph{The {S}elberg trace formula for {${\rm PSL}(2,\,{\bf
  R})$}. {V}ol. 2}, Lecture Notes in Mathematics, vol. 1001, Springer-Verlag,
  Berlin, 1983. \MR{711197}

\bibitem[Iwa02]{Iwaniec}
Henryk Iwaniec, \emph{Spectral methods of automorphic forms}, second ed.,
  Graduate Studies in Mathematics, vol.~53, American Mathematical Society,
  Providence, RI; Revista Matem\'{a}tica Iberoamericana, Madrid, 2002.
  \MR{1942691}

\bibitem[Koy91]{Koyama_3}
Shin-ya Koyama, \emph{Determinant expression of {S}elberg zeta functions.
  {III}}, Proc. Amer. Math. Soc. \textbf{113} (1991), no.~2, 303--311.
  \MR{1062391}

\bibitem[Leh87]{Lehto}
Olli Lehto, \emph{Univalent functions and {T}eichm\"{u}ller spaces}, Graduate
  Texts in Mathematics, vol. 109, Springer-Verlag, New York, 1987. \MR{867407}

\bibitem[MT06]{Mcintyre}
A.~McIntyre and L.~A. Takhtajan, \emph{Holomorphic factorization of
  determinants of {L}aplacians on {R}iemann surfaces and a higher genus
  generalization of {K}ronecker's first limit formula}, Geom. Funct. Anal.
  \textbf{16} (2006), no.~6, 1291--1323. \MR{2276541}

\bibitem[MT08]{Mcintyre_Teo}
Andrew McIntyre and Lee-Peng Teo, \emph{Holomorphic factorization of
  determinants of {L}aplacians using quasi-{F}uchsian uniformization}, Lett.
  Math. Phys. \textbf{83} (2008), no.~1, 41--58. \MR{2377945}

\bibitem[Pat76]{Patterson}
S.~J. Patterson, \emph{The exponent of convergence of {P}oincar\'{e} series},
  Monatsh. Math. \textbf{82} (1976), no.~4, 297--315. \MR{425114}

\bibitem[Rau65]{Rauch}
H.~E. Rauch, \emph{A transcendental view of the space of algebraic {R}iemann
  surfaces}, Bull. Amer. Math. Soc. \textbf{71} (1965), 1--39. \MR{213543}

\bibitem[Spr57]{Springer}
George Springer, \emph{Introduction to {R}iemann surfaces}, Addison-Wesley
  Publishing Co., Inc., Reading, Mass., 1957. \MR{0092855}

\bibitem[SS54]{Schiffer}
Menahem Schiffer and Donald~C. Spencer, \emph{Functionals of finite {R}iemann
  surfaces}, Princeton University Press, Princeton, N. J., 1954. \MR{0065652}

\bibitem[Teo20]{Teo_LMP_2020}
Lee-Peng Teo, \emph{Ruelle zeta function for cofinite hyperbolic {R}iemann
  surfaces with ramification points}, Lett. Math. Phys. \textbf{110} (2020),
  no.~1, 61--82. \MR{4047682}

\bibitem[Teo21]{Teo_Sigma}
\bysame, \emph{Resolvent trace formula and determinants of {$n$} {L}aplacians
  on orbifold {R}iemann surfaces}, SIGMA Symmetry Integrability Geom. Methods
  Appl. \textbf{17} (2021), Paper No. 083, 40. \MR{4310915}

\bibitem[TT06]{Memoir}
Leon~A. Takhtajan and Lee-Peng Teo, \emph{Weil-{P}etersson metric on the
  universal {T}eichm\"{u}ller space}, Mem. Amer. Math. Soc. \textbf{183}
  (2006), no.~861, viii+119. \MR{2251887}

\bibitem[TZ87]{TZ_index_1}
L.~A. Takhtajan and P.~G. Zograf, \emph{A local index theorem for families of
  {$\overline\partial$}-operators on {R}iemann surfaces}, Russ. Math. Surv.
  \textbf{42} (1987), no.~6, 169--190; original Russian version in Uspekhi Mat.
  Nauk \textbf{42} (1987), no.6, 133--150. \MR{0933998}

\bibitem[TZ91]{TZ_index_2}
\bysame, \emph{A local index theorem for families of
  {$\overline\partial$}-operators on punctured {R}iemann surfaces and a new
  {K}\"{a}hler metric on their moduli spaces}, Comm. Math. Phys. \textbf{137}
  (1991), no.~2, 399--426. \MR{1101693}

\bibitem[TZ19]{TZ_index_3}
Leon~A. Takhtajan and Peter Zograf, \emph{Local index theorem for orbifold
  {R}iemann surfaces}, Lett. Math. Phys. \textbf{109} (2019), no.~5,
  1119--1143. \MR{3946487}

\bibitem[Ven82]{Venkov}
A.~B. Venkov, \emph{Spectral theory of automorphic functions}, Proc. Steklov
  Inst. Math. (1982), no.~4(153), ix+163 pp. (1983), A translation of Trudy
  Mat. Inst. Steklov. {{\bf{1}}53} (1981). \MR{692019}

\bibitem[VKF73]{Faddeev}
A.~B. Venkov, V.~L. Kalinin, and L.~D. Faddeev, \emph{A nonarithmetic
  derivation of the {S}elberg trace formula}, Zap. Nau\v{c}n. Sem. Leningrad.
  Otdel. Mat. Inst. Steklov. (LOMI) \textbf{37} (1973), 5--42, Differential
  geometry, Lie groups and mechanics. \MR{0506043}

\bibitem[Wol86]{Wolpert_ChernForm}
Scott~A. Wolpert, \emph{Chern forms and the {R}iemann tensor for the moduli
  space of curves}, Invent. Math. \textbf{85} (1986), no.~1, 119--145.
  \MR{842050}

\end{thebibliography}
\end{document}